\documentclass{amsart}
\usepackage{color}
\usepackage{amsmath}
\usepackage{kotex}
\usepackage{amssymb}
\usepackage{amsthm}
\usepackage{amscd}
\usepackage{bm}
\usepackage{eucal} 
\usepackage{enumitem}
\usepackage{tikz}
\usepackage[utf8]{inputenc}
\usepackage[all]{xy}
\allowdisplaybreaks

\usepackage{hyperref}
\hypersetup{
  colorlinks   = true,          
  urlcolor     = blue,          
  linkcolor    = blue,          
  citecolor   = red             
}

\theoremstyle{plain}
\newtheorem{theorem}{Theorem}[section]
\newtheorem{conjecture}[theorem]{Conjecture}
\newtheorem{lemma}[theorem]{Lemma}
\newtheorem{proposition}[theorem]{Proposition}

\newtheorem{theoremx}{Theorem}

\theoremstyle{definition}
\newtheorem{definition}[theorem]{Definition}
\newtheorem{remark}[theorem]{Remark}

\numberwithin{equation}{section}

\newcommand\fantome[1]{}

\def\bL{\mathbb L}

\def\bT{\mathbb T}

\def\cF{\mathcal F}

\def\Fq{\mathbb F_q}

\newcommand{\ppar}{${}$\par}

\DeclareMathOperator{\Init}{Init}

\DeclareMathOperator{\Li}{Li}
\DeclareMathOperator{\Si}{Si}

\newcommand{\F}{\mathbb{F}}
\newcommand{\C}{\mathbb{C}}

\newcommand{\bff}{\mathbf{f}}
\newcommand{\bg}{\mathbf{g}}

\newcommand{\bh}{\mathbf{h}}
\newcommand{\bv}{\mathbf{v}}

\newcommand{\fla}{\bm{\lambda}}
\newcommand{\fm}{\bm{\mu}}
\newcommand{\fs}{\mathfrak{s}}
\newcommand{\fe}{\bm{\epsilon}}
\newcommand{\fve}{\bm{\varepsilon}}
\newcommand{\frakL}{\mathfrak{L}}
\newcommand{\frakLi}{\mathfrak{Li}}
\newcommand{\fQ}{\mathfrak{Q}}

\newcommand{\N}{\ensuremath \mathbb{N}}

\DeclareMathOperator{\Mat}{Mat}

\DeclareMathOperator{\depth}{depth}

\newcommand{\invtwist}{^{(-1)}}
\newcommand{\twistinv}{^{(-1)}}

\author[B.-H. Im]{Bo-Hae Im}
\address{
Dept. of Mathematical Sciences, KAIST,
291 Daehak-ro, Yuseong-gu,
Daejeon 34141, South Korea
}
\email{bhim@kaist.ac.kr}

\author[H. Kim]{Hojin Kim}
\address{
Dept. of Mathematical Sciences, KAIST,
291 Daehak-ro, Yuseong-gu,
Daejeon 34141, South Korea
}
\email{hojinkim@kaist.ac.kr}

\author[K. N. Le]{Khac Nhuan Le}
\address{
Normandie Université,
Université de Caen Normandie - CNRS,
Laboratoire de Mathématiques Nicolas Oresme (LMNO), UMR 6139,
14000 Caen, France.
}
\email{khac-nhuan.le@unicaen.fr}

\author[T. Ngo Dac]{Tuan Ngo Dac}
\address{
Normandie Université,
Université de Caen Normandie - CNRS,
Laboratoire de Mathématiques Nicolas Oresme (LMNO), UMR 6139,
14000 Caen, France.
}
\email{tuan.ngodac@unicaen.fr}

\author[L. H. Pham]{Lan Huong Pham}
\address{
Institute of Mathematics, Vietnam Academy of Science and Technology, 18 Hoang Quoc Viet, 10307 Hanoi, Viet Nam
}
\email{plhuong@math.ac.vn}

\title{Zagier-Hoffman's conjectures in positive characteristic}

\subjclass[2010]{Primary 11M32; Secondary 11G09, 11J93, 11M38, 11R58}

\keywords{Anderson $t$-motives, Anderson-Brownawell-Papanikolas criterion, (alternating) multiple zeta values, (alternating) Carlitz multiple polylogarithms}

\begin{document}

\maketitle

\begin{abstract}
Multiples zeta values and alternating multiple zeta values in positive characteristic were introduced by Thakur and Harada as analogues of classical multiple zeta values of Euler and Euler sums. In this paper we determine all linear relations between alternating multiple zeta values and settle the main goals of these theories. As a consequence we completely establish Zagier-Hoffman's conjectures in positive characteristic formulated by Todd and Thakur which predict the dimension and an explicit basis of the span of multiple zeta values of Thakur of fixed weight.
\end{abstract}

\tableofcontents

\section*{Introduction} \label{introduction}

\subsection{Classical setting}

\subsubsection{Multiple zeta values}

Multiple zeta values of Euler (MZV's for short) are real positive numbers given by
	\[ \zeta(n_1,\dots,n_r)=\sum_{0<k_1<\dots<k_r} \frac{1}{k_1^{n_1} \dots k_r^{n_r}}, \quad \text{where } n_i \geq 1, n_r \geq 2. \]
Here $r$ is called the depth and $w=n_1+\dots+n_r$ is called the weight of the presentation $\zeta(n_1,\dots,n_r)$. These values cover the special values $\zeta(n)$ for $n \geq 2$ of the Riemann zeta function and have been studied intensively especially in the last three decades with important and deep connections to different branches of mathematics and physics, for example arithmetic geometry, knot theory and higher energy physics. We refer the reader to  \cite{BGF, Zag94} for more details.

The main goal of this theory is to understand all $\mathbb Q$-linear relations between MZV's. Goncharov \cite[Conjecture 4.2]{Gon97} conjectures that all $\mathbb Q$-linear relations between MZV's can be derived from those between MZV's of the same weight. As the next step, precise conjectures formulated by Zagier \cite{Zag94} and Hoffman \cite{Hof97} predict the dimension and an explicit basis for the $\mathbb Q$-vector space $\mathcal{Z}_k$ spanned by MZV's of weight $k$ for $k \in \mathbb{N}$.

\begin{conjecture}[Zagier's conjecture]
We define a Fibonacci-like sequence of integers $d_k$ as follows. Letting $d_0=1, d_1=0$ and $d_2=1$ we define $d_k=d_{k-2}+d_{k-3}$ for $k \geq 3$. Then for $k \in \N$ we have
	\[ \dim_{\mathbb Q} \mathcal Z_k = d_k. \]
\end{conjecture}

\begin{conjecture}[Hoffman's conjecture]
The $\mathbb Q$-vector space $\mathcal Z_k$ is generated by the basis consisting of MZV's of weight $k$ of the form $\zeta(n_1,\dots,n_r)$ with $n_i \in \{2,3\}$.
\end{conjecture}

The algebraic part of these conjectures which concerns upper bounds for $\dim_{\mathbb Q} \mathcal Z_k$ was solved by Terasoma \cite{Ter02}, Deligne-Goncharov \cite{DG05} and Brown \cite{Bro12} using the theory of mixed Tate motives.

\begin{theorem}[Deligne-Goncharov, Terasoma]
For $k \in \N$ we have $\dim_{\mathbb Q} \mathcal Z_k \leq d_k$.
\end{theorem}

\begin{theorem}[Brown]
The $\mathbb Q$-vector space $\mathcal Z_k$ is generated by MZV's of weight $k$ of the form $\zeta(n_1,\dots,n_r)$ with $n_i \in \{2,3\}$.
\end{theorem}

Unfortunately, the transcendental part which concerns lower bounds for $\dim_{\mathbb Q} \mathcal Z_k$ is completely open. We refer the reader to \cite{BGF,Del13,Zag94} for more details and more exhaustive references.

\subsubsection{Alternating multiple zeta values}

There exists a variant of MZV's called the alternating multiple zeta values (AMZV's for short), also known as Euler sums. They are real numbers given by
	\[ \zeta \begin{pmatrix}
	\epsilon_1 & \dots & \epsilon_r \\
	n_1 & \dots & n_r
	\end{pmatrix}=\sum_{0<k_1<\dots<k_r} \frac{\epsilon_1^{k_1} \dots \epsilon_r^{k_r}}{k_1^{n_1} \dots k_r^{n_r}} \]
where $\epsilon_i \in \{\pm 1\}$, $n_i \in \N$ and $(n_r,\epsilon_r) \neq (1,1)$. Similar to MZV's, these values have been studied by Broadhurst, Deligne–Goncharov, Hoffman, Kaneko–Tsumura and many others because of the many connections in different contexts. We refer the reader to \cite{Har21,Hof19,Zha16} for further references.

As before, it is expected that all $\mathbb Q$-linear relations between AMZV's can be derived from those between AMZV's of the same weight. In particular, it is natural to ask whether one could formulate conjectures similar to those of Zagier and Hoffman for AMZV's of fixed weight. By the work of Deligne-Goncharov \cite{DG05}, the sharp upper bounds are achieved:
\begin{theorem}[Deligne-Goncharov]
For $k \in \N$ if we denote by $\mathcal A_k$ the $\mathbb Q$-vector space spanned by AMZV's of weight $k$, then $\dim_{\mathbb Q} \mathcal A_k \leq F_{k+1}$. Here $F_n$ is the $n$-th Fibonacci number defined by $F_1=F_2=1$ and $F_{n+2}=F_{n+1}+F_n$ for all $n \geq 1$.
\end{theorem}

The fact that the previous upper bounds would be sharp was also explained by Deligne in \cite{Del10} (see also \cite{DG05}) using a variant of a conjecture of Grothendieck. In the direction of extending Brown's theorem for AMZV's, there are several sets of generators for $\mathcal A_k$ (see for example \cite{Cha21,Del10}). However, we mention that these generators are only linear combinations of AMZV's.

Finally, we know nothing about non-trivial lower bounds for $\dim_{\mathbb Q} \mathcal A_k$.

\subsection{Function field setting} \label{sec:function fields}

\subsubsection{MZV's of Thakur and analogues of Zagier-Hoffman's conjectures} \label{sec:ZagierHoffman}

By analogy between number fields and function fields, based on the pioneering work of Carlitz \cite{Car35}, Thakur \cite{Tha04} defined analogues of multiple zeta values in positive characteristic. We now need to introduce some notations. Let $A=\Fq[\theta]$ be the polynomial ring in the variable $\theta$ over a finite field $\Fq$ of $q$ elements of characteristic $p>0$. We denote by $A_+$ the set of monic polynomials in $A$.  Let $K=\Fq(\theta)$ be the fraction field of $A$ equipped with the rational point $\infty$. Let $K_\infty$ be the completion of $K$ at $\infty$ and $\C_\infty$ be the completion of a fixed algebraic closure $\overline K$ of $K$ at $\infty$. We denote by $v_\infty$ the discrete valuation on $K$ corresponding to the place $\infty$ normalized such that $v_\infty(\theta)=-1$, and by $\lvert\cdot\rvert_\infty= q^{-v_\infty}$ the associated absolute value on $K$.  The unique valuation of $\mathbb C_\infty$ which extends $v_\infty$ will still be denoted by $v_\infty$. Finally we denote by $\overline{\mathbb F}_q$ the algebraic closure of $\mathbb F_q$ in $\overline{K}$.

Let $\mathbb N=\{1,2,\dots\}$ be the set of positive integers and $\mathbb Z^{\geq 0}=\{0,1,2,\dots\}$ be the set of non-negative integers. In \cite{Car35} Carlitz introduced the Carlitz zeta values $\zeta_A(n)$ for $n \in \N$ given by
	\[ \zeta_A(n) := \sum_{a \in A_+} \frac{1}{a^n} \in K_\infty \]
which are analogues of classical special zeta values in the function field setting.  For any tuple of positive integers $\mathfrak s=(s_1,\ldots,s_r) \in \N^r$, Thakur \cite{Tha04} defined the characteristic $p$ multiple zeta value (MZV for short) $\zeta_A(\fs)$ or $\zeta_A(s_1,\ldots,s_r)$ by
\begin{equation*}
\zeta_A(\fs):=\sum \frac{1}{a_1^{s_1} \ldots a_r^{s_r}} \in K_\infty
\end{equation*}
where the sum runs through the set of tuples $(a_1,\ldots,a_r) \in A_+^r$ with $\deg a_1>\cdots>\deg a_r$. We call $r$ the depth of $\zeta_A(\fs)$ and $w(\fs)=s_1+\dots+s_r$ the weight of $\zeta_A(\fs)$. We note that Carlitz zeta values are exactly depth one MZV's. Thakur \cite{Tha09a} showed that all the MZV's do not vanish. We refer the reader to \cite{AT90,AT09,GP21,LRT14,LRT21,Pel12,Tha04,Tha09,Tha10,Tha17,Tha20,Yu91} for more details on these objects.

As in the classical setting, the main goal of the theory is to understand all linear relations over $K$ between MZV's.  We now state analogues of Zagier-Hoffman's conjectures in positive characteristic formulated by Thakur in \cite[\S 8]{Tha17} and by Todd in \cite{Tod18}. For $w \in \N$ we denote by $\mathcal Z_w$ the $K$-vector space spanned by the MZV's of weight $w$. We denote by $\mathcal T_w$ the set of $\zeta_A(\fs)$ where $\fs=(s_1,\ldots,s_r) \in \N^r$ of weight $w$ with $1\leq s_i\leq q$ for $1\leq i\leq r-1$ and $s_r<q$.

\begin{conjecture}[Zagier's conjecture in positive characteristic] \label{conj: dimension}
Letting
\begin{align*}
d(w)=\begin{cases}
1 & \text{ if } w=0, \\
2^{w-1} & \text{ if } 1 \leq w \leq q-1, \\
2^{w-1}-1 & \text{ if } w=q,
\end{cases}
\end{align*}
we put $d(w)=\sum_{i=1}^q d(w-i)$ for $w>q$. Then for any $w \in \N$, we have
	\[ \dim_K \mathcal Z_w = d(w). \]
\end{conjecture}

\begin{conjecture}[Hoffman's conjecture in positive characteristic] \label{conj: basis}
A $K$-basis for $\mathcal Z_w$ is given by  $\mathcal T_w$ consisting of $\zeta_A(s_1,\ldots,s_r)$ of weight $w$ with  $s_i \leq q$ for $1 \leq i <r$, and $s_r<q$.
\end{conjecture}

In \cite{ND21} one of the authors succeeded in proving the algebraic part of these conjectures (see \cite[Theorem A]{ND21}): for all $w \in \N$, we have
	\[ \dim_K \mathcal Z_w \leq d(w). \]
This part is based on shuffle relations for MZV's due to Chen and Thakur and some operations introduced by Todd. For the transcendental part, he used the Anderson-Brownawell-Papanikolas criterion in \cite{ABP04} and proved sharp lower bounds for small weights $w \leq 2q-2$ (see \cite[Theorem D]{ND21}). It has already been noted that it is very difficult to extend his method to general weights (see \cite{ND21} for more details).

\subsubsection{AMZV's in positive characteristic}

Recently, Harada \cite{Har21} introduced the alternating multiple zeta values in positive characteristic (AMZV's) as follows. Letting $\fs=(s_1,\dots,s_r) \in \N^n$ and $\fve=(\varepsilon_1,\dots,\varepsilon_r) \in (\Fq^\times)^n$, we define
\begin{equation*}
    \zeta_A \begin{pmatrix}
 \fve  \\
\fs  \end{pmatrix}  := \sum \dfrac{\varepsilon_1^{\deg a_1} \dots \varepsilon_r^{\deg a_r }}{a_1^{s_1} \dots a_r^{s_r}}  \in K_{\infty}
\end{equation*}
where the sum runs through the set of tuples $(a_1,\ldots,a_r) \in A_+^r$ with $\deg a_1>\cdots>\deg a_r$. The numbers $r$ and $w(\fs)=s_1+\dots+s_r$ are called the depth and the weight of $\zeta_A \begin{pmatrix}
 \fve  \\
\fs  \end{pmatrix}$, respectively. We set $\zeta_A \begin{pmatrix}
 \emptyset  \\
 \emptyset  \end{pmatrix}  = 1$. Harada \cite{Har21} extended basic properties of MZV's to AMZV's, i.e., non-vanishing, shuffle relations, period interpretation and linear independence. Again the main goal of this theory is to determine all linear relations over $K$ between AMZV's. It is natural to ask whether the previous work on analogues of the Zagier-Hoffman conjectures can be extended to this setting. More precisely, if for $w \in \N$ we denote by $\mathcal{AZ}_w$ the $K$ vector space spanned by the AMZV's of weight $w$, then we would like to determine the dimensions of $\mathcal{AZ}_w$ and show some nice bases of these vector spaces.

\subsection{Main results}

\subsubsection{Statements of the main results}

In this manuscript we present complete answers to all the previous conjectures and problems raised in \S \ref{sec:function fields}.

First, for all $w$ we calculate the dimension of $\mathcal{AZ}_w$ and give an explicit basis in the spirit of Hoffman.
\begin{theoremx} \label{thm: ZagierHoffman AMZV}
We define a Fibonacci-like sequence $s(w)$ as follows. We put
\begin{equation*}
    s(w) = \begin{cases}
(q - 1) q^{w-1}& \text{if } 1 \leq w < q, \\
            (q - 1) (q^{w-1} - 1) &\text{if } w = q,
		 \end{cases}
\end{equation*}
and for $w>q$, $s(w)=(q-1)\sum \limits_{i = 1}^{q-1} s(w-i) + s(w - q)$. Then for all $w \in \N$,
	\[ \dim_K \mathcal{AZ}_w = s(w). \]
	
Further, we can exhibit a Hoffman-like basis of $\mathcal{AZ}_w$.
\end{theoremx}

Second, we give a proof of both Conjectures \ref{conj: dimension} and \ref{conj: basis} which generalizes the previous work of the fourth author \cite{ND21}.
\begin{theoremx} \label{thm: ZagierHoffman}
For all $w \in \N$,  $\mathcal T_w$ is a $K$-basis for $\mathcal Z_w$. In particular,
	\[ \dim_K \mathcal Z_w = d(w). \]
\end{theoremx}

We recall that analogues of Goncharov's conjectures in positive characteristic were proved in \cite{Cha14}. As a consequence, we give a framework for understanding all linear relations over $K$ between MZV's and AMZV's and settle the main goals of these theories.

\subsubsection{Ingredients of the proofs}

Let us emphasize here that Theorem \ref{thm: ZagierHoffman AMZV} is much harder than Theorem \ref{thm: ZagierHoffman} and that it is not enough to work within the setting of AMZV's. On the one hand, although it is straightforward to extend the algebraic part for AMZV's following the same line in \cite[\S 2 and \S 3]{ND21}, we only obtain a weak version of Brown's theorem in this setting. More precisely, we get a set of generators for $\mathcal{AZ}_w$ but it is too large to be a basis of this vector space. For small weights, we find ad hoc arguments to produce a smaller set of generators but it does not work for arbitrary weights (see \S \ref{sec:without ACMPL}). Roughly speaking, in \cite[\S 2 and \S 3]{ND21} we have an algorithm which moves forward so that we can express any AMZV's as a linear combination of generators. But we lack some precise controls on the coefficients in these expressions so that we cannot go backward and change bases. On the other hand, the transcendental part for AMZV's shares the same difficulties with the case of MZV's as noted above.

In this paper we use a completely new approach which is based on the study of alternating Carlitz multiple polylogarithms (ACMPL's for short) defined as follows. We put $\ell_0:=1$ and $\ell_d:=\prod_{i=1}^d (\theta-\theta^{q^i})$ for all $d \in \mathbb N$. For any tuple $\fs=(s_1,\ldots,s_r) \in \N^r$ and $\fve=(\varepsilon_1,\dots,\varepsilon_r) \in (\Fq^\times)^r$, we introduce the corresponding alternating Carlitz multiple polylogarithm by
\begin{equation*}
\Li\begin{pmatrix}
 \fve  \\
\fs  \end{pmatrix} := \sum\limits_{d_1> \dots > d_r\geq 0} \dfrac{\varepsilon_1^{d_1} \dots \varepsilon_r^{d_r} }{\ell_{d_1}^{s_1} \dots \ell_{d_r}^{s_r}}   \in K_{\infty}.
\end{equation*}
We also set $\Li \begin{pmatrix}
 \emptyset  \\
 \emptyset  \end{pmatrix}  = 1$.

The key result is to establish a non-trivial connection between AMZV's and ACMPL's which allows us to go back and forth between these objects (see Theorem \ref{thm:bridge}). To do this, following \cite[\S 2 and \S 3]{ND21} we use stuffle relations to develop an algebraic theory for ACMPL's and obtain a weak version of Brown's theorem, i.e., a set of generators for the $K$-vector space $\mathcal{AL}_w$ spanned by ACMPL's of weight $w$. We observe that this set of generators is exactly the same as that for AMZV's. Thus $\mathcal{AL}_w=\mathcal{AZ}_w$, which provides a dictionary between AMZV's and ACMPL's.

We then determine all $K$-linear relations between ACMPL's (see Theorem \ref{thm:ACMPL}). The proof we give here, while using similar tools as in \cite{ND21}, differs in some crucial points and requires three new ingredients.

The first new ingredient is the construction of an appropriate Hoffman-like basis $\mathcal{AS}_w$ of $\mathcal{AL}_w$. In fact, our transcendental method dictates that we must find a complete system of bases $\mathcal{AS}_w$ of $\mathcal{AL}_w$ indexed by weights $w$ with strong constraints as given in Theorem \ref{theorem: linear independence}. The failure to find such a system of bases is the main obstacle to generalizing \cite[Theorem D]{ND21} (see \S \ref{sec: application AMZV} and \cite[Remark 6.3]{ND21} for more details).

The second new ingredient is formulating and proving (a strong version of) Brown's theorem for AMCPL's (see Theorem \ref{thm: strong Brown}). As mentioned before, the method in \cite{ND21} only gives a weak version of Brown's theorem for ACMPL's as the set of generators is not a basis. Roughly speaking, given any ACMPL's we can express it as a linear combination of generators. The fact that stuffle relations for ACMPL's are ``simpler" than shuffle relations for AMZV's gives more precise information about the coefficients of these expressions. Consequently, we show that a certain transition matrix is invertible and obtain Brown's theorem for ACMPL's. This completes the algebraic part for ACMPL's.

The last new ingredient is proving the transcendental part for ACMPL's in full generality, i.e., the ACMPL's in $\mathcal{AS}_w$ are linearly independent over $K$ (see Theorem \ref{thm: trans ACMPL}). We emphasize that we do need the full strength of the algebraic part to prove the transcendental part. The proof follows the same line in \cite[\S 4 and \S 5]{ND21} which is formulated in a more general setting in \S \ref{sec: dual motives}. First, we have to consider not only linear relations between ACMPL's in $\mathcal{AS}_w$ but also those between ACMPL's in $\mathcal{AS}_w$ and the suitable power $\widetilde \pi^w$ of the Carlitz period $\widetilde \pi$. Second, starting from such a non-trivial relation we apply the Anderson-Brownawell-Papanikolas criterion in \cite{ABP04} and reduce to solve a system of $\sigma$-linear equations. While in \cite[\S 4 and \S 5]{ND21} this system does not have a non-trivial solution which allows us to conclude, our system has a unique solution for even $w$ (i.e., $q-1$ divides $w$). This means that for such $w$ up to a scalar there is a unique linear relation between ACMPL's in $\mathcal{AS}_w$ and $\widetilde \pi^w$. The last step consists of showing that in this unique relation, the coefficient of $\widetilde \pi^w$ is nonzero. Unexpectedly, this is a consequence of Brown's theorem for AMCPL's mentioned above.

\subsubsection{Plan of the paper}

We will briefly explain the organization of the manuscript.
\begin{itemize}
\item In \S \ref{sec:algebraic part} we recall the definition and basic properties of ACMPL's. We then develop an algebraic theory for these objects and obtain weak and strong Brown's theorems (see Proposition \ref{prop: weak Brown} and Theorem \ref{thm: strong Brown}).

\item In \S \ref{sec: dual motives} we generalize some transcendental results in \cite{ND21} and give statements in a more general setting (see Theorem \ref{theorem: linear independence}).

\item In \S \ref{sec:transcendental part} we prove transcendental results for ACMPL's  and completely determine all linear relations between ACMPL's (see Theorems \ref{thm: trans ACMPL} and \ref{thm:ACMPL}).

\item Finally, in \S \ref{sec:applications} we present two applications and prove the main results, i.e., Theorems \ref{thm: ZagierHoffman AMZV} and \ref{thm: ZagierHoffman}. The first application is to prove the above connection between ACMPL's and AMZV's and then to determine all linear relations between AMZV's in positive characteristic (see \S \ref{sec: application AMZV}). The second application is a proof of Zagier-Hoffman's conjectures in positive characteristic which generalizes the main results of \cite{ND21} (see \S \ref{sec: application MZV}).
\end{itemize}

\subsection{Remark} When our work was released in arXiv:2205.07165, Chieh-Yu Chang informed us that Chen, Mishiba and he were working towards a proof of Theorem \ref{thm: ZagierHoffman} (e.g., the MZV version) by using a similar method, and their paper \cite{CCM22} is now available at arXiv:2205.09929.

\subsection{Acknowledgments}
The first named author (B.-H. Im) was supported by the National Research Foundation of Korea (NRF) grant funded by the Korea government (MSIT) (No.~2020R1A2B5B01001835). Two of the authors (KN. L. and T. ND.) were partially supported by the Excellence Research Chair ``$L$-functions in positive characteristic and applications'' funded by the Normandy Region. The fourth author (T. ND.) was partially supported by the ANR Grant COLOSS ANR-19-CE40-0015-02. The fifth author (LH. P.) was supported by the grant ICRTM.02-2021.05 funded by the International Center for Research and Postgraduate Training in Mathematics (VAST, Vietnam).


\section{Weak and strong Brown's theorems for ACMPL's} \label{sec:algebraic part}

In this section we first extend the work of \cite{ND21} and develop an algebraic theory for ACMPL's. Then we prove a weak version of Brown's theorem for ACMPL's (see Theorem \ref{prop: weak Brown}) which gives a set of generators for the $K$-vector space spanned by ACMPL's of weight $w$. The techniques of Sections \ref{sec: power sums}-\ref{sec:weak Brown} are similar to those of \cite{ND21} and the  reader may wish to skip the details.

Contrary to what happens in \cite{ND21}, it turns out that the previous set of generators is too large to be a basis. Consequently, in \S \ref{sec:strong Brown} we introduce another set of generators and prove a strong version of Brown's theorem for ACMPL's (see Theorem \ref{thm: strong Brown}).

\subsection{Analogues of power sums} \label{sec: power sums}

\subsubsection{}

We recall and introduce some notation in \cite{ND21}. A tuple $\fs$ is a sequence of the form $\fs=(s_1,\dots,s_n) \in \mathbb N^n$. We call $\depth(\fs) = n$ the depth and $w(\fs) = s_1 + \dots + s_n$ the weight of $\fs$. If $\fs$ is nonempty, we put $\fs_- := (s_2, \dotsc, s_n)$.

Let $\fs $ and $\mathfrak{t}$ be two tuples of positive integers. We set $s_i = 0$ (resp. $t_i = 0$) for all $i > \textup{depth}(\fs)$ (resp. $i > \textup{depth}(\mathfrak{t})$). We say that $\fs \leq \mathfrak{t}$ if $s_1 + \cdots + s_i \leq t_1 + \cdots + t_i$ for all $i \in \mathbb{N}$, and $w(\fs) = w(\mathfrak{t})$. This defines a partial order on tuples of positive integers.

For $i \in \mathbb{N}$, we define $T_i(\fs)$ to be the tuple $(s_1+\dots+s_i,s_{i+1},\dots,s_n)$. Further, for $i \in \mathbb N$, if $T_i(\fs) \leq T_i(\mathfrak t)$, then $T_k(\fs) \leq T_k(\mathfrak t)$ for all $k \geq i$.

Let $\fs=(s_1,\dots,s_n) \in \mathbb N^n$ be a tuple of positive integers. We denote by $0 \leq i \leq n$ the largest integer such that $s_j \leq q$ for all $1 \leq j \leq i$ and define the initial tuple $\Init(\fs)$ of $\fs$ to be the tuple
	\[ \Init(\fs):=(s_1,\dots,s_i). \]
In particular, if $s_1>q$, then $i=0$ and $\Init(\fs)$ is the empty tuple.

For two different tuples $\fs$ and $\mathfrak t$, we consider the lexicographic order for initial tuples and write $\Init(\mathfrak t) \preceq \Init(\fs)$ (resp. $\Init(\mathfrak t) \prec \Init(\fs)$, $\Init(\mathfrak t) \succeq \Init(\fs)$ and $\Init(\mathfrak t) \succ \Init(\fs)$).

\subsubsection{}

Letting $\fs = (s_1, \dotsc, s_n) \in \mathbb{N}^{n}$ and $\fve = (\varepsilon_1, \dotsc, \varepsilon_n) \in (\mathbb{F}_q^{\times})^{n}$, we set $\fs_- := (s_2, \dotsc, s_n)$ and $\fve_- := (\varepsilon_2, \dotsc, \varepsilon_n) $. By definition, an array $\begin{pmatrix}
 \fve  \\
\fs  \end{pmatrix} $ is an array of the form $$\begin{pmatrix}
 \fve  \\
\fs  \end{pmatrix}  = \begin{pmatrix}
 \varepsilon_1 & \dotsb & \varepsilon_n \\
s_1 & \dotsb & s_n \end{pmatrix}.$$
We call $\depth(\fs) = n$ the depth, $w(\fs) = s_1 + \dots + s_n$ the weight and $\chi(\fve) = \varepsilon_1  \dots  \varepsilon_n$ the character of $\begin{pmatrix}
 \fve  \\ \fs  \end{pmatrix}$.

We say that $\begin{pmatrix}
 \fve  \\
\fs  \end{pmatrix} \leq \begin{pmatrix}
 \fe  \\
\mathfrak{t}  \end{pmatrix} $
 if the following conditions are satisfied:
\begin{enumerate}
    \item $\chi(\fve) = \chi(\fe)$,
    \item $w(\fs) = w(\mathfrak{t})$,
    \item $s_1 + \dotsb + s_i \leq t_1 + \dotsb + t_i$ for all $i \in \N$.
\end{enumerate}
We note that this only defines a preorder on arrays.

\begin{remark} \label{rmk: compa depth}
We claim that if $\begin{pmatrix}
 \fve  \\
\fs  \end{pmatrix} \leq \begin{pmatrix}
 \fe  \\
\mathfrak{t}  \end{pmatrix} $, then $\depth(\fs) \geq \depth(\mathfrak{t})$. In fact, assume that $\depth(\fs) < \depth(\mathfrak{t})$. Thus
\begin{align*}
    w(\fs) = s_1 + \cdots + s_{\depth(\fs)} \leq t_1 + \cdots + t_{\depth(\fs)} < t_1 + \cdots + t_{\depth(\mathfrak{t})} = w(\mathfrak{t}),
\end{align*}
which contradicts the condition $w(\fs) = w(\mathfrak{t})$.
\end{remark}

\subsubsection{} \label{sec: definition Sd}
We recall the power sums and MZV's studied by Thakur \cite{Tha10}. For $d \in \mathbb{Z}$ and for $\fs=(s_1,\dots,s_n) \in \N^n$ we introduce
\begin{equation*}
    S_d(\fs) :=  \sum\limits_{\substack{a_1, \dots, a_n \in A_{+} \\ d = \deg a_1> \dots > \deg a_n\geq 0}} \dfrac{1}{a_1^{s_1} \dots a_n^{s_n}} \in K
\end{equation*}
and
\begin{equation*}
    S_{<d}(\mathfrak s) := \sum\limits_{\substack{a_1, \dots, a_n \in A_{+} \\ d > \deg a_1> \dots > \deg a_n\geq 0}} \dfrac{1}{a_1^{s_1} \dots a_n^{s_n}} \in K.
\end{equation*}
We define the multiple zeta value (MZV) by
\begin{equation*}
    \zeta_A(\fs) := \sum \limits_{d \geq 0} S_d(\fs) = \sum \limits_{d \geq 0} \sum\limits_{\substack{a_1, \dots, a_n \in A_{+} \\ d = \deg a_1> \dots > \deg a_n\geq 0}} \dfrac{1}{a_1^{s_1} \dots a_n^{s_n}} \in K_{\infty}.
\end{equation*}
We put
$\zeta_A(\emptyset)  = 1$. We call $\depth(\fs) = n$ the depth and $w(\fs) = s_1 + \dots + s_n$ the weight of  $\zeta_A(\fs)$.

We also recall that $\ell_0 := 1$ and $\ell_d := \prod^d_{i=1}(\theta - \theta^{q^i})$ for all $d \in \mathbb{N}$. Letting $\mathfrak s = (s_1 , \dots, s_n) \in \mathbb{N}^n$, for $d \in \mathbb{Z}$, we define analogues of power sums by
\begin{equation*}
        \Si_d(\mathfrak s) := \sum\limits_{d=d_1> \dots > d_n\geq 0} \dfrac{1}{\ell_{d_1}^{s_1} \dots \ell_{d_n}^{s_n}} \in K,
\end{equation*}
and
\begin{equation*}
        \Si_{<d}(\mathfrak s) := \sum\limits_{d>d_1> \dots > d_n\geq 0} \dfrac{1 }{\ell_{d_1}^{s_1} \dots \ell_{d_n}^{s_n}} \in K.
\end{equation*}
We introduce the Carlitz multiple polylogarithm (CMPL for short) by
\begin{equation*}
    \Li(\fs) := \sum \limits_{d \geq 0} \Si_d(\fs) = \sum \limits_{d \geq 0} \ \sum\limits_{d=d_1> \dots > d_n\geq 0} \dfrac{1}{\ell_{d_1}^{s_1} \dots \ell_{d_n}^{s_n}}   \in K_{\infty}.
\end{equation*}
We set
$\Li(\emptyset)  = 1$. We call $\depth(\fs) = n$ the depth and $w(\fs) = s_1 + \dots + s_n$ the weight of  $\Li(\fs)$.

\subsubsection{}
Let $\begin{pmatrix}
 \fve  \\
\fs  \end{pmatrix}  =  \begin{pmatrix}
 \varepsilon_1 & \dots & \varepsilon_n \\
s_1 & \dots & s_n \end{pmatrix} $ be an array. For $d \in \mathbb{Z}$, we define
\begin{equation*}
S_d \begin{pmatrix}
\fve \\ \fs
\end{pmatrix} := \sum\limits_{\substack{a_1, \dots, a_n \in A_{+} \\ d = \deg a_1> \dots > \deg a_n\geq 0}} \dfrac{\varepsilon_1^{\deg a_1} \dots \varepsilon_n^{\deg a_n }}{a_1^{s_1} \dots a_n^{s_n}} \in K
\end{equation*}
and
\begin{equation*}
    S_{<d} \begin{pmatrix}
 \fve  \\
\fs  \end{pmatrix} := \sum\limits_{\substack{a_1, \dots, a_n \in A_{+} \\ d > \deg a_1> \dots > \deg a_n\geq 0}} \dfrac{\varepsilon_1^{\deg a_1} \dots \varepsilon_n^{\deg a_n }}{a_1^{s_1} \dots a_n^{s_n}} \in K.
\end{equation*}
We also introduce
\begin{equation*}
        \Si_d \begin{pmatrix}
 \fve  \\
\fs  \end{pmatrix} := \sum\limits_{d=d_1> \dots > d_n\geq 0} \dfrac{\varepsilon_1^{d_1} \dots \varepsilon_n^{d_n} }{\ell_{d_1}^{s_1} \dots \ell_{d_n}^{s_n}} \in K,
\end{equation*}
and
\begin{equation*}
        \Si_{<d} \begin{pmatrix}
 \fve  \\
\fs  \end{pmatrix} := \sum\limits_{d>d_1> \dots > d_n\geq 0} \dfrac{\varepsilon_1^{d_1} \dots \varepsilon_n^{d_n} }{\ell_{d_1}^{s_1} \dots \ell_{d_n}^{s_n}} \in K.
\end{equation*}
One verifies easily the following formulas:
\begin{align*}
\Si_{d}
    \begin{pmatrix} \varepsilon \\ s  \end{pmatrix}
&= \varepsilon^d \Si_d(s),\\
\Si_d \begin{pmatrix}
 1& \dots & 1 \\
s_1 & \dots & s_n \end{pmatrix}  &= \Si_{d}(s_1, \dots, s_n),\\
\Si_{<d} \begin{pmatrix}
 1& \dots & 1 \\
s_1 & \dots & s_n \end{pmatrix}  &= \Si_{<d}(s_1, \dots, s_n),\\
\Si_{d} \begin{pmatrix}
 \fve  \\
\fs  \end{pmatrix}  &= \Si_{d} \begin{pmatrix}
 \varepsilon_1  \\
s_1  \end{pmatrix} \Si_{<d} \begin{pmatrix}
 \fve_{-}  \\
\fs_{-}  \end{pmatrix} .
\end{align*}
Then we define the alternating Carlitz multiple polylogarithm (ACMPL for short) by
\begin{equation*}
    \Li \begin{pmatrix}
 \fve  \\
\fs  \end{pmatrix} := \sum \limits_{d \geq 0} \Si_d \begin{pmatrix}
 \fve  \\
\fs  \end{pmatrix}  = \sum\limits_{d_1> \dots > d_n\geq 0} \dfrac{\varepsilon_1^{d_1} \dots \varepsilon_n^{d_n} }{\ell_{d_1}^{s_1} \dots \ell_{d_n}^{s_n}}   \in K_{\infty}.
\end{equation*}
Recall that
$\Li \begin{pmatrix}
 \emptyset  \\
 \emptyset  \end{pmatrix}  = 1$. We call $\depth(\fs) = n$ the depth, $w(\fs) = s_1 + \dots + s_n$ the weight and $\chi(\fve) = \varepsilon_1  \dots  \varepsilon_n$ the character of  $\Li \begin{pmatrix}
 \fve  \\ \fs  \end{pmatrix}$.

\begin{lemma} \label{agree}
For all $ \begin{pmatrix}
 \fve  \\
\fs  \end{pmatrix}$ as above such that $s_i \leq q$ for all $i$, we have
\begin{equation*}
   S_d \begin{pmatrix}
 \fve  \\
\fs  \end{pmatrix}  =  \Si_d \begin{pmatrix}
 \fve  \\
\fs  \end{pmatrix}  \quad \text{for all } d \in \mathbb{Z}.
\end{equation*}
Therefore,
\begin{equation*}
    \zeta_A  \begin{pmatrix}
 \fve  \\
\fs  \end{pmatrix}  = \Li \begin{pmatrix}
 \fve  \\
\fs  \end{pmatrix} .
\end{equation*}
\end{lemma}
\begin{proof}
We denote by $\mathcal{J}$ the set of all arrays $ \begin{pmatrix}
 \fve  \\
\fs  \end{pmatrix}  =  \begin{pmatrix}
 \varepsilon_1 & \dots & \varepsilon_n \\
s_1 & \dots & s_n \end{pmatrix} $ for some $n$ such that $s_1, \dots, s_n \leq q$.

The second statement follows at once from the first statement. We prove the first statement by induction on $\depth(\fs)$. For $\depth(\fs) = 1$, we let $ \begin{pmatrix}
 \fve  \\
\fs  \end{pmatrix}  =  \begin{pmatrix}
 \varepsilon \\
s  \end{pmatrix} $ with $s \le q$. It follows from  special cases of power sums in \cite[\S 3.3]{Tha09} that for all $d \in \mathbb{Z}$,
$
    S_{d} \begin{pmatrix}
 \varepsilon \\
s  \end{pmatrix}  =  \dfrac{\varepsilon^d}{\ell^s_d} = \Si_{d} \begin{pmatrix}
 \varepsilon \\
s  \end{pmatrix} .
$ Suppose that the first statement holds for all arrays $ \begin{pmatrix}
 \fve  \\
\fs  \end{pmatrix}  \in \mathcal{J}$ with $\depth(\fs) = n - 1$ and for all $ d \in \mathbb{Z}$. Let $ \begin{pmatrix}
 \fve  \\
\fs  \end{pmatrix}  =  \begin{pmatrix}
 \varepsilon_1 & \dots & \varepsilon_n \\
s_1 & \dots & s_n \end{pmatrix} $ be an element of $\mathcal{J}$. Note that if $ \begin{pmatrix}
 \fve  \\
\fs  \end{pmatrix}  \in \mathcal{J}$, then $ \begin{pmatrix}
 \fve_{-}  \\
\fs_{-}  \end{pmatrix}  \in \mathcal{J}$. It follows from induction hypothesis and the fact $s_1 \leq q$ that for all $d \in \mathbb{Z}$
\begin{equation*}
    S_d \begin{pmatrix}
 \fve  \\
\fs  \end{pmatrix}  = S_d \begin{pmatrix}
 \varepsilon_1  \\
s_1  \end{pmatrix}  S_{<d} \begin{pmatrix}
 \fve_{-}  \\
\fs_{-}  \end{pmatrix}
    = \Si_d \begin{pmatrix}
 \varepsilon_1  \\
s_1  \end{pmatrix}  \Si_{<d} \begin{pmatrix}
 \fve_{-}  \\
\fs_{-}  \end{pmatrix}  = \Si_d \begin{pmatrix}
 \fve  \\
\fs  \end{pmatrix} .
\end{equation*}
This proves the lemma.
\end{proof}

\subsubsection{}

Let $\begin{pmatrix}
 \fve  \\
\fs  \end{pmatrix}$, $\begin{pmatrix}
 \fe  \\
\mathfrak{t}  \end{pmatrix}$ be two arrays. We recall $s_i = 0$ and $\varepsilon_i = 1$ for all $i > \depth(\fs)$; $t_i = 0$ and $\epsilon_i = 1$ for all $i > \depth(\mathfrak{t})$. We define the following operation
\begin{equation*}
    \begin{pmatrix}
 \fve  \\
\fs  \end{pmatrix} + \begin{pmatrix}
 \fe  \\
\mathfrak{t}  \end{pmatrix} := \begin{pmatrix}
 \fve \fe  \\
\fs + \mathfrak{t}  \end{pmatrix},
\end{equation*}
where $\fve \fe$ and $\fs + \mathfrak{t}$ are defined by component multiplication and component addition, respectively. 

We now consider some formulas related to analogues of power sums. It is easily seen that
\begin{equation}\label{eq: redsum}
  \Si_d \begin{pmatrix}
 \varepsilon \\
s  \end{pmatrix}   \Si_d \begin{pmatrix}
 \epsilon  \\
t  \end{pmatrix}   = \Si_d \begin{pmatrix}
 \varepsilon \epsilon \\
s + t  \end{pmatrix} ,
\end{equation}
hence, for  $\mathfrak{t} = (t_1, \dots, t_n)$,
\begin{equation}\label{redsum}
  \Si_d \begin{pmatrix}
 \varepsilon \\
s  \end{pmatrix}   \Si_d \begin{pmatrix}
 \fe  \\
\mathfrak{t}  \end{pmatrix}   = \Si_d \begin{pmatrix}
 \varepsilon \epsilon_1 & \fe_{-}  \\
s + t_1  & \mathfrak{t}_{-} \end{pmatrix}.
\end{equation}
More generally, we deduce the following proposition which will be used frequently later.
\begin{proposition} \label{polysums}
Let $ \begin{pmatrix}
 \fve  \\
\fs  \end{pmatrix} $, $ \begin{pmatrix}
 \fe  \\
\mathfrak{t}  \end{pmatrix} $ be two arrays. Then we have the following:
\begin{enumerate}
    \item There exist $f_i \in \mathbb{F}_q$ and arrays $ \begin{pmatrix}
 \fm_i  \\
\mathfrak{u}_i  \end{pmatrix} $ with $ \begin{pmatrix}
 \fm_i  \\
\mathfrak{u}_i  \end{pmatrix}  \leq  \begin{pmatrix}
 \fve  \\
\fs  \end{pmatrix}  +  \begin{pmatrix}
 \fe  \\
\mathfrak{t}  \end{pmatrix}  $ and $\depth(\mathfrak{u}_i) \leq \depth(\fs) + \depth(\mathfrak{t})$ for all $i$  such that
    \begin{equation*}
        \Si_d \begin{pmatrix}
 \fve  \\
\fs  \end{pmatrix} \Si_d \begin{pmatrix}
 \fe  \\
\mathfrak{t}  \end{pmatrix}  = \sum \limits_i f_i \Si_d \begin{pmatrix}
 \fm_i  \\
\mathfrak{u}_i  \end{pmatrix}  \quad \text{for all } d \in \mathbb{Z}.
    \end{equation*}
    \item There exist $f'_i \in \mathbb{F}_q$ and arrays $ \begin{pmatrix}
 \fm'_i  \\
\mathfrak{u}'_i  \end{pmatrix} $ with $ \begin{pmatrix}
 \fm'_i  \\
\mathfrak{u}'_i  \end{pmatrix}  \leq  \begin{pmatrix}
 \fve  \\
\fs  \end{pmatrix}  +  \begin{pmatrix}
 \fe  \\
\mathfrak{t}  \end{pmatrix}  $ and $\depth(\mathfrak{u}'_i) \leq \depth(\fs) + \depth(\mathfrak{t})$ for all $i$  such that
    \begin{equation*}
        \Si_{<d} \begin{pmatrix}
 \fve  \\
\fs  \end{pmatrix} \Si_{<d} \begin{pmatrix}
 \fe  \\
\mathfrak{t}  \end{pmatrix}  = \sum \limits_i f'_i \Si_{<d} \begin{pmatrix}
 \fm'_i  \\
\mathfrak{u}'_i  \end{pmatrix}  \quad \text{for all } d \in \mathbb{Z}.
    \end{equation*}
        \item There exist $f''_i \in \mathbb{F}_q$ and arrays $ \begin{pmatrix}
 \fm''_i  \\
\mathfrak{u}''_i  \end{pmatrix} $ with $ \begin{pmatrix}
 \fm''_i  \\
\mathfrak{u}''_i  \end{pmatrix}  \leq  \begin{pmatrix}
 \fve  \\
\fs  \end{pmatrix}  +  \begin{pmatrix}
 \fe  \\
\mathfrak{t}  \end{pmatrix}  $ and $\depth(\mathfrak{u}''_i) \leq \depth(\fs) + \depth(\mathfrak{t})$ for all $i$  such that
    \begin{equation*}
        \Si_d \begin{pmatrix}
 \fve  \\
\fs  \end{pmatrix} \Si_{<d} \begin{pmatrix}
 \fe  \\
\mathfrak{t}  \end{pmatrix}  = \sum \limits_i f''_i \Si_d \begin{pmatrix}
 \fm''_i  \\
\mathfrak{u}''_i  \end{pmatrix}  \quad \text{for all } d \in \mathbb{Z}.
    \end{equation*}
\end{enumerate}
\end{proposition}

\begin{proof}
The proof follows the same line as in \cite[Proposition 2.1]{ND21}. We write down the proof for the reader's convenience. 

We first prove Part 1 and Part 2 by induction on $\depth(\fs) + \depth(\mathfrak{t})$. For $\depth(\fs) + \depth(\mathfrak{t}) = 2$, we have $\depth(\fs) = \depth(\mathfrak{t}) = 1$, hence we may assume that $ \begin{pmatrix}
 \fve  \\
\fs  \end{pmatrix} = \begin{pmatrix}
 \varepsilon  \\
s  \end{pmatrix}$ and $ \begin{pmatrix}
 \fe  \\
\mathfrak{t}  \end{pmatrix} = \begin{pmatrix}
 \epsilon  \\
t  \end{pmatrix}$, where $s, t \in \mathbb{N}$ and $\varepsilon, \epsilon  \in \mathbb{F}_q^{\times}$. For Part 1, it follows from \eqref{eq: redsum} that
\begin{align*}
    \Si_d \begin{pmatrix}
 \varepsilon \\
s  \end{pmatrix}   \Si_d \begin{pmatrix}
 \epsilon  \\
t  \end{pmatrix}   = \Si_d \begin{pmatrix}
 \varepsilon \epsilon \\
s + t  \end{pmatrix}.
\end{align*}
We have $\begin{pmatrix}
 \varepsilon \epsilon \\
s + t  \end{pmatrix} = \begin{pmatrix}
 \varepsilon \\
s  \end{pmatrix} + \begin{pmatrix}
 \epsilon  \\
t  \end{pmatrix}$ and $\depth(s + t) = 1 < \depth(s) + \depth(t) = 2$, which shows that Part 1 holds in this case. For Part 2, it follows from \eqref{eq: redsum} that
\begin{align*}
    \Si_{<d} \begin{pmatrix}
 \varepsilon \\
s  \end{pmatrix}   \Si_{<d} \begin{pmatrix}
 \epsilon  \\
t  \end{pmatrix} =& \ \left(\sum_{m < d}\Si_{m} \begin{pmatrix}
 \varepsilon \\
s  \end{pmatrix}\right)\left(\sum_{n < d}\Si_{n} \begin{pmatrix}
 \epsilon  \\
t  \end{pmatrix}\right)\\
=& \ \sum_{m < d}\Si_{m} \begin{pmatrix}
 \varepsilon \\
s  \end{pmatrix}\Si_{<m} \begin{pmatrix}
 \epsilon  \\
t  \end{pmatrix} + \sum_{n < d}\Si_{n} \begin{pmatrix}
 \epsilon  \\
t  \end{pmatrix}\Si_{<n} \begin{pmatrix}
 \varepsilon \\
s  \end{pmatrix}\\
&+ \sum_{m = n < d}\Si_{m} \begin{pmatrix}
 \varepsilon \\
s  \end{pmatrix}\Si_{n} \begin{pmatrix}
 \epsilon  \\
t  \end{pmatrix}\\
=& \ \sum_{m < d}\Si_{m} \begin{pmatrix}
 \varepsilon & \epsilon \\
s & t \end{pmatrix} + \sum_{n < d}\Si_{n} \begin{pmatrix}
 \epsilon &  \varepsilon\\
t  & s\end{pmatrix} + \sum_{m < d}\Si_{m} \begin{pmatrix}
 \varepsilon\epsilon \\
s + t \end{pmatrix}\\
=& \ \Si_{<d} \begin{pmatrix}
 \varepsilon & \epsilon \\
s & t \end{pmatrix} + \Si_{<d} \begin{pmatrix}
 \epsilon &  \varepsilon\\
t  & s\end{pmatrix} + \Si_{<d} \begin{pmatrix}
 \varepsilon\epsilon \\
s + t \end{pmatrix}.
\end{align*}
One verifies that $\begin{pmatrix}
 \varepsilon & \epsilon \\
s & t \end{pmatrix} \leq \begin{pmatrix}
 \varepsilon \\
s  \end{pmatrix} + \begin{pmatrix}
 \epsilon  \\
t  \end{pmatrix}, \begin{pmatrix}
 \epsilon &  \varepsilon\\
t  & s\end{pmatrix} \leq \begin{pmatrix}
 \varepsilon \\
s  \end{pmatrix} + \begin{pmatrix}
 \epsilon  \\
t  \end{pmatrix}$ and $\begin{pmatrix}
 \varepsilon \epsilon \\
s + t  \end{pmatrix} = \begin{pmatrix}
 \varepsilon \\
s  \end{pmatrix} + \begin{pmatrix}
 \epsilon  \\
t  \end{pmatrix}$. Moreover, we have $\depth(s, t) = \depth(t, s) = \depth(s) + \depth(t) = 2$ and $\depth(s + t) = 1 < \depth(s) + \depth(t) = 2$, which shows that Part 2 holds in this case. 

We assume that Part 1 and Part 2 hold when $\depth(\fs) + \depth(\mathfrak{t}) < d$. We need to show that Part 1 and Part 2 hold when $\depth(\fs) + \depth(\mathfrak{t}) = d$. 

For Part 1, we deduce from \eqref{eq: redsum} that
\begin{align} \label{eq: Part 1 Si}
    \Si_d \begin{pmatrix}
 \fve  \\
\fs  \end{pmatrix} \Si_d \begin{pmatrix}
 \fe  \\
\mathfrak{t}  \end{pmatrix} &= \Si_d \begin{pmatrix}
 \varepsilon_1  \\
s_1  \end{pmatrix}\Si_{<d} \begin{pmatrix}
 \fve_- \\
\fs_-  \end{pmatrix} \Si_d \begin{pmatrix}
 \epsilon_1  \\
t_1  \end{pmatrix}\Si_{<d} \begin{pmatrix}
 \fe_-  \\
\mathfrak{t}_-  \end{pmatrix}\\ \notag
&= \Si_d \begin{pmatrix}
 \varepsilon_1  \\
s_1  \end{pmatrix}\Si_d \begin{pmatrix}
 \epsilon_1  \\
t_1  \end{pmatrix}\Si_{<d} \begin{pmatrix}
 \fve_- \\
\fs_-  \end{pmatrix} \Si_{<d} \begin{pmatrix}
 \fe_-  \\
\mathfrak{t}_-  \end{pmatrix}\\ \notag
&= \Si_d \begin{pmatrix}
 \varepsilon_1\epsilon_1  \\
s_1 + t_1 \end{pmatrix}\Si_{<d} \begin{pmatrix}
 \fve_- \\
\fs_-  \end{pmatrix} \Si_{<d} \begin{pmatrix}
 \fe_-  \\
\mathfrak{t}_-  \end{pmatrix}.
\end{align}
Since $\depth(\fs_- ) + \depth(\mathfrak{t}_-) = \depth(\fs) + \depth(\mathfrak{t}) - 2 < \depth(\fs) + \depth(\mathfrak{t}) = d$, it follows from the induction hypothesis that there exist $f'_j \in \mathbb{F}_q$ and arrays $ \begin{pmatrix}
 \fm'_j  \\
\mathfrak{u}'_j  \end{pmatrix} $ with $ \begin{pmatrix}
 \fm'_j  \\
\mathfrak{u}'_j  \end{pmatrix}  \leq  \begin{pmatrix}
 \fve_-  \\
\fs_-  \end{pmatrix}  +  \begin{pmatrix}
 \fe_-  \\
\mathfrak{t}_-  \end{pmatrix}  $ and $\depth(\mathfrak{u}'_j) \leq \depth(\fs_-) + \depth(\mathfrak{t}_-)$ for all $j$  such that
    \begin{equation*}
        \Si_{<d} \begin{pmatrix}
 \fve_-  \\
\fs_-  \end{pmatrix} \Si_{<d} \begin{pmatrix}
 \fe_-  \\
\mathfrak{t}_-  \end{pmatrix}  = \sum \limits_j f'_j \Si_{<d} \begin{pmatrix}
 \fm'_j  \\
\mathfrak{u}'_j  \end{pmatrix}  \quad \text{for all } d \in \mathbb{Z}.
    \end{equation*}
Thus we deduce from \eqref{eq: Part 1 Si} that 
\begin{align*}
     \Si_d \begin{pmatrix}
 \fve  \\
\fs  \end{pmatrix} \Si_d \begin{pmatrix}
 \fe  \\
\mathfrak{t}  \end{pmatrix} = \sum \limits_j f'_j\Si_d \begin{pmatrix}
 \varepsilon_1\epsilon_1  \\
s_1 + t_1 \end{pmatrix} \Si_{<d} \begin{pmatrix}
 \fm'_j  \\
\mathfrak{u}'_j  \end{pmatrix} = \sum \limits_j f'_j \Si_{d} \begin{pmatrix}
 \varepsilon_1\epsilon_1 & \fm'_j  \\
s_1 + t_1 & \mathfrak{u}'_j  \end{pmatrix}.
\end{align*}
One verifies that $\begin{pmatrix}
 \varepsilon_1\epsilon_1 & \fm'_j  \\
s_1 + t_1 & \mathfrak{u}'_j  \end{pmatrix} \leq \begin{pmatrix}
 \varepsilon_1 & \fve_-  \\
s_1 & \fs_-  \end{pmatrix}  +  \begin{pmatrix}
\epsilon_1 &  \fe_-  \\
t_1 & \mathfrak{t}_-  \end{pmatrix} = \begin{pmatrix}
 \fve  \\
\fs\end{pmatrix}  +  \begin{pmatrix}
 \fe  \\
\mathfrak{t}  \end{pmatrix}$, and 
$\depth(s_1 + t_1 , \mathfrak{u}'_j) = 1 + \depth(\mathfrak{u}'_j) \leq 1 + \depth(\fs_-) + \depth(\mathfrak{t}_-) < \depth(\fs) + \depth(\mathfrak{t})$, which proves Part 1.

For Part 2, we have
\begin{align} \label{eq: Part 2 Si}
    \Si_{<d} \begin{pmatrix}
 \fve \\
\fs  \end{pmatrix} &  \Si_{<d} \begin{pmatrix}
 \fe  \\
\mathfrak{t}  \end{pmatrix}\\ \notag
=& \ \left(\sum_{m < d}\Si_{m} \begin{pmatrix}
 \fve \\
\fs  \end{pmatrix}\right)\left(\sum_{n < d}\Si_{n} \begin{pmatrix}
 \fe  \\
\mathfrak{t}  \end{pmatrix}\right)\\ \notag
=& \ \sum_{m < d}\Si_{m} \begin{pmatrix}
 \fve \\
\fs  \end{pmatrix}\Si_{<m} \begin{pmatrix}
 \fe  \\
\mathfrak{t}  \end{pmatrix} + \sum_{n < d}\Si_{n} \begin{pmatrix}
 \fe  \\
\mathfrak{t}  \end{pmatrix}\Si_{<n} \begin{pmatrix}
 \fve \\
\fs  \end{pmatrix} + \sum_{m = n < d}\Si_{m} \begin{pmatrix}
 \fve \\
\fs  \end{pmatrix}\Si_{n} \begin{pmatrix}
 \fe  \\
\mathfrak{t}  \end{pmatrix}\\ \notag
=& \ \sum_{m < d}\Si_{m} \begin{pmatrix}
 \varepsilon_1 \\
s_1  \end{pmatrix}\Si_{<m} \begin{pmatrix}
 \fve_- \\
\fs_-  \end{pmatrix}\Si_{<m} \begin{pmatrix}
 \fe  \\
\mathfrak{t}  \end{pmatrix} + \sum_{n < d}\Si_{n} \begin{pmatrix}
 \epsilon_1 \\
t_1  \end{pmatrix}\Si_{<n} \begin{pmatrix}
 \fe_-  \\
\mathfrak{t}_-  \end{pmatrix}\Si_{<n} \begin{pmatrix}
 \fve \\
\fs  \end{pmatrix}\\ \notag
&+ \sum_{m < d}\Si_{m} \begin{pmatrix}
 \fve \\
\fs  \end{pmatrix}\Si_{m} \begin{pmatrix}
 \fe  \\
\mathfrak{t}  \end{pmatrix}.
\end{align}
Since $\depth(\fs_- ) + \depth(\mathfrak{t}) = \depth(\fs) + \depth(\mathfrak{t}_-) = \depth(\fs) + \depth(\mathfrak{t}) - 1 < \depth(\fs) + \depth(\mathfrak{t}) = d$, it follows from the induction hypothesis that there exist $g_j, g_j' \in \mathbb{F}_q$ and 
\begin{itemize}
    \item arrays $ \begin{pmatrix}
 \boldsymbol{\eta}_j  \\
\mathfrak{v}_j  \end{pmatrix} $ with $ \begin{pmatrix}
 \boldsymbol{\eta}_j  \\
\mathfrak{v}_j  \end{pmatrix}\leq  \begin{pmatrix}
 \fve_-  \\
\fs_-  \end{pmatrix}  +  \begin{pmatrix}
 \fe  \\
\mathfrak{t}  \end{pmatrix}  $ and $\depth(\mathfrak{v}_j) \leq \depth(\fs_-) + \depth(\mathfrak{t})$ for all $j$,
\item arrays $ \begin{pmatrix}
 \boldsymbol{\eta}'_j  \\
\mathfrak{v}'_j  \end{pmatrix} $ with $ \begin{pmatrix}
 \boldsymbol{\eta}'_j  \\
\mathfrak{v}'_j  \end{pmatrix} \leq  \begin{pmatrix}
 \fe_-  \\
\mathfrak{t}_-  \end{pmatrix} + \begin{pmatrix}
 \fve  \\
\fs  \end{pmatrix} $ and $\depth(\mathfrak{v}'_j) \leq \depth(\mathfrak{t}_-) + \depth(\fs)$ for all $j$,
\end{itemize}
  such that
    \begin{align*}
        \Si_{<m} \begin{pmatrix}
 \fve_- \\
\fs_-  \end{pmatrix}\Si_{<m} \begin{pmatrix}
 \fe  \\
\mathfrak{t}  \end{pmatrix}  &= \sum \limits_j g_j \Si_{<m} \begin{pmatrix}
 \boldsymbol{\eta}_j  \\
\mathfrak{v}_j  \end{pmatrix}  \quad \text{for all } m \in \mathbb{Z},\\
\Si_{<n} \begin{pmatrix}
 \fe_-  \\
\mathfrak{t}_-  \end{pmatrix}\Si_{<n} \begin{pmatrix}
 \fve \\
\fs  \end{pmatrix} &= \sum \limits_j g_j' \Si_{<n} \begin{pmatrix}
 \boldsymbol{\eta}'_j  \\
\mathfrak{v}'_j  \end{pmatrix}  \quad \text{for all } n \in \mathbb{Z}.
    \end{align*}
Moreover, from Part 1, there exist $f_j \in \mathbb{F}_q$ and arrays $ \begin{pmatrix}
 \fm_j  \\
\mathfrak{u}_j  \end{pmatrix} $ with $ \begin{pmatrix}
 \fm_j  \\
\mathfrak{u}_j  \end{pmatrix}  \leq  \begin{pmatrix}
 \fve  \\
\fs  \end{pmatrix}  +  \begin{pmatrix}
 \fe  \\
\mathfrak{t}  \end{pmatrix}  $ and $\depth(\mathfrak{u}_j) \leq \depth(\fs) + \depth(\mathfrak{t})$ for all $j$  such that
    \begin{equation*}
        \Si_m \begin{pmatrix}
 \fve  \\
\fs  \end{pmatrix} \Si_m \begin{pmatrix}
 \fe  \\
\mathfrak{t}  \end{pmatrix}  = \sum \limits_j f_j \Si_m \begin{pmatrix}
 \fm_j  \\
\mathfrak{u}_j  \end{pmatrix}  \quad \text{for all } m \in \mathbb{Z}.
    \end{equation*}
We deduce from \eqref{eq: Part 2 Si} that 
\begin{align*}
    \Si_{<d} \begin{pmatrix}
 \fve \\
\fs  \end{pmatrix} &  \Si_{<d} \begin{pmatrix}
 \fe  \\
\mathfrak{t}  \end{pmatrix}\\
=& \ \sum_{m < d}\Si_{m} \begin{pmatrix}
 \varepsilon_1 \\
s_1  \end{pmatrix}\Si_{<m} \begin{pmatrix}
 \fve_- \\
\fs_-  \end{pmatrix}\Si_{<m} \begin{pmatrix}
 \fe  \\
\mathfrak{t}  \end{pmatrix} + \sum_{n < d}\Si_{n} \begin{pmatrix}
 \epsilon_1 \\
t_1  \end{pmatrix}\Si_{<n} \begin{pmatrix}
 \fe_-  \\
\mathfrak{t}_-  \end{pmatrix}\Si_{<n} \begin{pmatrix}
 \fve \\
\fs  \end{pmatrix}\\ \notag
&+ \sum_{m < d}\Si_{m} \begin{pmatrix}
 \fve \\
\fs  \end{pmatrix}\Si_{m} \begin{pmatrix}
 \fe  \\
\mathfrak{t}  \end{pmatrix} \\
=& \ \sum \limits_j g_j \sum_{m < d}\Si_{m} \begin{pmatrix}
 \varepsilon_1 \\
s_1  \end{pmatrix}\Si_{<m} \begin{pmatrix}
 \boldsymbol{\eta}_j  \\
\mathfrak{v}_j  \end{pmatrix} + \sum \limits_j g_j' \sum_{n < d}\Si_{n} \begin{pmatrix}
 \epsilon_1 \\
t_1  \end{pmatrix}\Si_{<n} \begin{pmatrix}
 \boldsymbol{\eta}'_j  \\
\mathfrak{v}'_j  \end{pmatrix}\\ \notag
&+ \sum \limits_j f_j \sum_{m < d}\Si_m \begin{pmatrix}
 \fm_j  \\
\mathfrak{u}_j  \end{pmatrix} \\
=& \ \sum \limits_j g_j \sum_{m < d}\Si_{m} \begin{pmatrix}
 \varepsilon_1 & \boldsymbol{\eta}_j\\
s_1 & \mathfrak{v}_j \end{pmatrix} + \sum \limits_j g_j' \sum_{n < d}\Si_{n} \begin{pmatrix}
 \epsilon_1 & \boldsymbol{\eta}'_j\\
t_1 & \mathfrak{v}'_j \end{pmatrix} + \sum \limits_j f_j \Si_{<d} \begin{pmatrix}
 \fm_j  \\
\mathfrak{u}_j  \end{pmatrix} \\
=& \ \sum \limits_j g_j \Si_{<d} \begin{pmatrix}
 \varepsilon_1 & \boldsymbol{\eta}_j\\
s_1 & \mathfrak{v}_j \end{pmatrix} + \sum \limits_j g_j' \Si_{<d} \begin{pmatrix}
 \epsilon_1 & \boldsymbol{\eta}'_j\\
t_1 & \mathfrak{v}'_j \end{pmatrix} + \sum \limits_j f_j \Si_{<d} \begin{pmatrix}
 \fm_j  \\
\mathfrak{u}_j  \end{pmatrix}.
\end{align*}
One verifies that
\begin{align*}
    \begin{pmatrix}
 \varepsilon_1 & \boldsymbol{\eta}_j\\
s_1 & \mathfrak{v}_j \end{pmatrix} &\leq \bigg( \begin{pmatrix}
 \varepsilon_1 \\
s_1\end{pmatrix} , \begin{pmatrix}
 \fve_-\\
\fs_- \end{pmatrix}   +  \begin{pmatrix}
 \fe  \\
\mathfrak{t}  \end{pmatrix} \bigg) \leq  \begin{pmatrix}
 \fve\\
\fs \end{pmatrix}   +  \begin{pmatrix}
 \fe  \\
\mathfrak{t}  \end{pmatrix},\\
\begin{pmatrix}
 \epsilon_1 & \boldsymbol{\eta}'_j\\
t_1 & \mathfrak{v}'_j \end{pmatrix} & \leq \bigg(\begin{pmatrix}
 \epsilon_1\\
t_1\end{pmatrix}, \begin{pmatrix}
 \fe_-\\
\mathfrak{t}_- \end{pmatrix} + \begin{pmatrix}
 \fve\\
\fs \end{pmatrix}\bigg) \leq \begin{pmatrix}
 \fe  \\
\mathfrak{t}  \end{pmatrix} + \begin{pmatrix}
 \fve\\
\fs \end{pmatrix},
\end{align*}
and 
\begin{align*}
    \depth(s_1, \mathfrak{v}_j) &= 1 + \depth(\mathfrak{v}_j) \leq 1 + \depth(\fs_-) + \depth(\mathfrak{t}) = \depth(\fs) + \depth(\mathfrak{t}),\\
    \depth(t_1, \mathfrak{v}'_j)&= 1 + \depth(\mathfrak{v}'_j) \leq 1 + \depth(\mathfrak{t}_-) + \depth(\fs) = \depth(\mathfrak{t}) + \depth(\fs),
\end{align*}
which proves Part 2. The proof of Part 1 and Part 2 is finished.

For Part 3, we have
\begin{align*}
    \Si_d \begin{pmatrix}
 \fve  \\
\fs  \end{pmatrix} \Si_{<d} \begin{pmatrix}
 \fe  \\
\mathfrak{t}  \end{pmatrix} = \Si_d \begin{pmatrix}
 \varepsilon_1  \\
s_1  \end{pmatrix}\Si_{<d} \begin{pmatrix}
 \fve_- \\
\fs_-  \end{pmatrix} \Si_{<d} \begin{pmatrix}
 \fe  \\
\mathfrak{t}  \end{pmatrix}.
\end{align*}
From Part 2, there exist $f'_k \in \mathbb{F}_q$ and arrays $ \begin{pmatrix}
 \fm'_k  \\
\mathfrak{u}'_k  \end{pmatrix} $ with $ \begin{pmatrix}
 \fm'_k  \\
\mathfrak{u}'_k  \end{pmatrix}  \leq  \begin{pmatrix}
 \fve_-  \\
\fs_-  \end{pmatrix}  +  \begin{pmatrix}
 \fe  \\
\mathfrak{t} \end{pmatrix}  $ and $\depth(\mathfrak{u}'_k) \leq \depth(\fs_-) + \depth(\mathfrak{t})$ for all $k$  such that
    \begin{equation*}
        \Si_{<d} \begin{pmatrix}
 \fve_-  \\
\fs_-  \end{pmatrix} \Si_{<d} \begin{pmatrix}
 \fe  \\
\mathfrak{t}  \end{pmatrix}  = \sum \limits_k f'_k \Si_{<d} \begin{pmatrix}
 \fm'_k  \\
\mathfrak{u}'_k  \end{pmatrix}  \quad \text{for all } d \in \mathbb{Z}.
    \end{equation*}
We deduce that 
\begin{align*}
     \Si_d \begin{pmatrix}
 \fve  \\
\fs  \end{pmatrix} \Si_{<d} \begin{pmatrix}
 \fe  \\
\mathfrak{t}  \end{pmatrix} &= \sum \limits_k f'_k \Si_d \begin{pmatrix}
 \varepsilon_1  \\
s_1  \end{pmatrix}
\Si_{<d} \begin{pmatrix}
 \fm'_k  \\
\mathfrak{u}'_k  \end{pmatrix}= \sum \limits_k f'_k \Si_d \begin{pmatrix}
 \varepsilon_1 & \fm'_k \\
s_1 & \mathfrak{u}'_k \end{pmatrix}.
\end{align*}
One verifies that
\begin{align*}
    \begin{pmatrix}
 \varepsilon_1 & \fm'_k\\
s_1 & \mathfrak{u}'_k \end{pmatrix} \leq \bigg( \begin{pmatrix}
 \varepsilon_1 \\
s_1\end{pmatrix} , \begin{pmatrix}
 \fve_-\\
\fs_- \end{pmatrix}   +  \begin{pmatrix}
 \fe  \\
\mathfrak{t}  \end{pmatrix} \bigg) \leq  \begin{pmatrix}
 \fve\\
\fs \end{pmatrix}   +  \begin{pmatrix}
 \fe  \\
\mathfrak{t}  \end{pmatrix}
\end{align*}
and $\depth(s_1, \mathfrak{u}'_k) = 1 + \depth(\mathfrak{u}'_k) \leq 1 + \depth(\fs_-) + \depth(\mathfrak{t}) = \depth(\fs) + \depth(\mathfrak{t})$, which proves Part 3.
\end{proof}

We denote by $\mathcal{AL}$ (resp. $\mathcal{L}$)  the $K$-vector space generated by the ACMPL's (resp. by the CMPL's) and by $\mathcal{AL}_w$ (resp. $\mathcal{L}_w$) the $K$-vector space generated by the ACMPL's of weight $w$ (resp. by the CMPL's of weight $w$). It follows from Proposition~\ref{polysums} that $\mathcal{AL}$ is a $K$-algebra. By considering only arrays with trivial characters, Proposition~\ref{polysums} implies that $\mathcal{L}$ is also a $K$-algebra. 

\subsection{Operators $\mathcal B^*$, $\mathcal C$ and $\mathcal{BC}$} \label{sec:Todd} 

In this section we extend operators $\mathcal B^*$ and $\mathcal C$ of Todd \cite{Tod18} and the operator $\mathcal{BC}$ of Ngo Dac  \cite{ND21} in the case of ACMPL's.

\begin{definition}
A binary relation is a $K$-linear combination of the form
\begin{equation*}
    \sum \limits_i a_i \Si_d \begin{pmatrix}
 \fve_i  \\
\fs_i  \end{pmatrix}  + \sum \limits_i b_i \Si_{d+1} \begin{pmatrix}
 \fe_i  \\
\mathfrak{t}_i  \end{pmatrix}  =0 \quad \text{for all } d \in \mathbb{Z},
\end{equation*}
where $a_i,b_i \in K$ and $ \begin{pmatrix}
 \fve_i  \\
\fs_i  \end{pmatrix} ,  \begin{pmatrix}
 \fe_i  \\
\mathfrak{t}_i  \end{pmatrix} $ are arrays of the same weight.

A binary relation is called a fixed relation if $b_i = 0$ for all $i$.
\end{definition}

We denote by $\mathfrak{BR}_{w}$ the set of all binary relations of weight $w$. One verifies at once that $\mathfrak{BR}_{w}$ is a $K$-vector space. It follows from the fundamental relation in \cite[\S 3.4.6]{Tha09} and Lemma \ref{agree}, an important example of binary relations
\begin{equation*}
 R_{\varepsilon} \colon \quad  \Si_d \begin{pmatrix}
 \varepsilon\\
q  \end{pmatrix}  + \varepsilon^{-1}D_1 \Si_{d+1} \begin{pmatrix}
 \varepsilon& 1 \\
1 & q-1  \end{pmatrix}  =0,
\end{equation*}
where $D_1 = \theta^q - \theta \in A$.

For later definitions, let $R \in \mathfrak{BR}_w$ be a binary relation of the form
\begin{equation} \label{eq: Rd0}
    R(d) \colon \quad \sum \limits_i a_i \Si_d \begin{pmatrix}
 \fve_i  \\
\fs_i  \end{pmatrix}  + \sum \limits_i b_i \Si_{d+1} \begin{pmatrix}
 \fe_i  \\
\mathfrak{t}_i  \end{pmatrix}  =0,
\end{equation}
where $a_i,b_i \in K$ and $ \begin{pmatrix}
 \fve_i  \\
\fs_i  \end{pmatrix} ,  \begin{pmatrix}
 \fe_i  \\
\mathfrak{t}_i  \end{pmatrix} $ are arrays of the same weight. We now define some operators on $K$-vector spaces of binary relations.

\subsubsection{Operators $\mathcal B^*$}

Let $ \begin{pmatrix}
 \sigma  \\
v  \end{pmatrix} $ be an array. We define an operator
\begin{equation*}
    \mathcal B^*_{\sigma,v} \colon \mathfrak{BR}_{w} \longrightarrow \mathfrak{BR}_{w+v}
\end{equation*}
as follows: for each $R \in \mathfrak{BR}_{w}$ given as in \eqref{eq: Rd0},
the image $\mathcal B^*_{\sigma,v}(R) = \Si_d \begin{pmatrix}
 \sigma  \\
v  \end{pmatrix} \sum_{j < d} R(j)$ is a fixed relation of the form
\begin{align*}
    0 &= \Si_d \begin{pmatrix}
 \sigma  \\
v  \end{pmatrix}   \left(\sum \limits_ia_i \Si_{<d} \begin{pmatrix}
 \fve_i  \\
\fs_i  \end{pmatrix}  + \sum \limits_i  b_i \Si_{<d+1} \begin{pmatrix}
 \fe_i  \\
\mathfrak{t}_i  \end{pmatrix} \right)  \\
    &= \sum \limits_i a_i \Si_d \begin{pmatrix}
 \sigma  \\
v  \end{pmatrix} \Si_{<d} \begin{pmatrix}
 \fve_i  \\
\fs_i  \end{pmatrix}  + \sum \limits_i  b_i \Si_d \begin{pmatrix}
 \sigma  \\
v  \end{pmatrix}  \Si_{<d} \begin{pmatrix}
 \fe_i  \\
\mathfrak{t}_i  \end{pmatrix}  + \sum \limits_i  b_i \Si_d \begin{pmatrix}
 \sigma  \\
v  \end{pmatrix}  \Si_{d} \begin{pmatrix}
 \fe_i  \\
\mathfrak{t}_i  \end{pmatrix} \\
    &= \sum \limits_i a_i \Si_d \begin{pmatrix}
\sigma & \fve_i  \\
v& \fs_i  \end{pmatrix}  + \sum \limits_i  b_i \Si_d \begin{pmatrix}
\sigma & \fe_i  \\
v& \mathfrak{t}_i  \end{pmatrix}  + \sum \limits_i  b_i  \Si_d \begin{pmatrix}
 \sigma \epsilon_{i1}  & \fe_{i-} \\
v + t_{i1} & \mathfrak{t}_{i-} \end{pmatrix} .
\end{align*}
The last equality follows from \eqref{redsum}.

Let $ \begin{pmatrix}
 \Sigma  \\
V  \end{pmatrix}  =  \begin{pmatrix}
 \sigma_1 & \dots & \sigma_n \\
v_1 & \dots & v_n \end{pmatrix} $ be an array. We define an operator $\mathcal{B}^*_{\Sigma,V}(R) $ by
\begin{equation*}
    \mathcal B^*_{\Sigma,V}(R) := \mathcal B^*_{\sigma_1,v_1} \circ \dots \circ \mathcal B^*_{\sigma_n,v_n}(R).
\end{equation*}

\begin{lemma} \label{polybesao}
Let $ \begin{pmatrix}
 \Sigma  \\
V  \end{pmatrix}  =  \begin{pmatrix}
 \sigma_1 & \dots & \sigma_n \\
v_1 & \dots & v_n \end{pmatrix} $ be an array. Then $\mathcal B^*_{\Sigma,V}(R)$ is of the form
\begin{align*}
 \sum \limits_i a_i \Si_d \begin{pmatrix}
\Sigma & \fve_i  \\
V& \fs_i  \end{pmatrix} & + \sum \limits_i  b_i \Si_d \begin{pmatrix}
\Sigma & \fe_i  \\
V& \mathfrak{t}_i  \end{pmatrix}
 + \sum \limits_i  b_i \Si_d  \begin{pmatrix}
 \sigma_1 & \dots & \sigma_{n-1} & \sigma_n \epsilon_{i1} &\fe_{i-}  \\
v_1 & \dots & v_{n-1} & v_n+ t_{i1} & \mathfrak{t}_{i-}  \end{pmatrix}  = 0 .
\end{align*}
\end{lemma}

\begin{proof}
From the definition and \eqref{redsum}, we have $\mathcal{B}^*_{\sigma_n,v_n}(R)$ is of the form
\begin{align*}
    \sum \limits_i a_i \Si_d \begin{pmatrix}
\sigma_n & \fve_i  \\
v_n& \fs_i  \end{pmatrix}  + \sum \limits_i  b_i \Si_d \begin{pmatrix}
\sigma_n & \fe_i  \\
v_n& \mathfrak{t}_i  \end{pmatrix}  + \sum \limits_i  b_i \Si_d \begin{pmatrix}
 \sigma_n \epsilon_{i1} & {\fe_i}_- \\
v_n +  t_{i1} & {\mathfrak{t}_i}_- \end{pmatrix} = 0.
\end{align*}
Apply the operator $\mathcal B^*_{\sigma_1,v_1} \circ \dots \circ \mathcal B^*_{\sigma_{n - 1},v_{n - 1}}$ to $\mathcal{B}^*_{\sigma_n,v_n}(R)$, the result then follows from the definition.
\end{proof}

\subsubsection{Operators $\mathcal C$}

Let $ \begin{pmatrix}
 \Sigma  \\
V  \end{pmatrix} $ be an array of weight $v$. We define an operator
\begin{equation*}
     \mathcal C_{\Sigma,V}(R) \colon \mathfrak{BR}_{w} \longrightarrow \mathfrak{BR}_{w+v}
\end{equation*}
as follows: for each $R \in \mathfrak{BR}_{w}$ given as in \eqref{eq: Rd0},
the image $\mathcal C_{\Sigma,V}(R) = R(d) \Si_{<d+1} \begin{pmatrix}
 \Sigma  \\
V  \end{pmatrix} $ is a binary relation of the form
\begin{align*}
    0 &= \left( \sum \limits_i a_i \Si_d \begin{pmatrix}
 \fve_i  \\
\fs_i  \end{pmatrix}  + \sum \limits_i b_i \Si_{d+1} \begin{pmatrix}
 \fe_i  \\
\mathfrak{t}_i  \end{pmatrix} \right) \Si_{<d+1} \begin{pmatrix}
 \Sigma  \\
V  \end{pmatrix}   \\
    &= \sum \limits_i a_i \Si_d \begin{pmatrix}
 \fve_i  \\
\fs_i  \end{pmatrix} \Si_{d} \begin{pmatrix}
 \Sigma  \\
V  \end{pmatrix}  + \sum \limits_i a_i \Si_d \begin{pmatrix}
 \fve_i  \\
\fs_i  \end{pmatrix} \Si_{<d} \begin{pmatrix}
 \Sigma  \\
V  \end{pmatrix}  + \sum \limits_i b_i \Si_{d+1} \begin{pmatrix}
 \fe_i  \\
\mathfrak{t}_i  \end{pmatrix} \Si_{<d+1} \begin{pmatrix}
 \Sigma  \\
V  \end{pmatrix} \\
    &= \sum \limits_i c_i \Si_d \begin{pmatrix}
 \fm_i  \\
\mathfrak{u}_i  \end{pmatrix}  + \sum \limits_i c'_i \Si_{d+1} \begin{pmatrix}
 \fm'_i  \\
\mathfrak{u}'_i  \end{pmatrix} .
\end{align*}
The last equality follows from Proposition \ref{polysums}.

In particular, the following proposition gives the form of $\mathcal C_{\Sigma,V}(R_{\varepsilon})$.

\begin{proposition} \label{polycer1}
Let $ \begin{pmatrix}
 \Sigma  \\ V  \end{pmatrix}$ be an array with $V = (v_1,V_{-})$ and $\Sigma = (\sigma_1, \Sigma_{-})$. Then $\mathcal C_{\Sigma,V}(R_{\varepsilon})$ is of the form
\begin{equation*}
  \Si_d \begin{pmatrix}
 \varepsilon\sigma_1 & \Sigma_{-} \\
q + v_1 & V_{-} \end{pmatrix}  +  \Si_d \begin{pmatrix}
 \varepsilon& \Sigma \\
q & V \end{pmatrix}  + \sum \limits_i b_i \Si_{d+1} \begin{pmatrix}
\varepsilon& \fe_i  \\
1 & \mathfrak{t}_i  \end{pmatrix}  =0,
\end{equation*}
where $b_i \in A$ are divisible by $D_1$ and $ \begin{pmatrix}
 \fe_i  \\
\mathfrak{t}_i  \end{pmatrix} $ are arrays satisfying $ \begin{pmatrix}
 \fe_i  \\
\mathfrak{t}_i  \end{pmatrix}  \leq  \begin{pmatrix}
 1  \\
q - 1 \end{pmatrix}  +  \begin{pmatrix}
 \Sigma  \\
V  \end{pmatrix} $ for all $i$.
\end{proposition}

\begin{proof}
We see that $\mathcal C_{\Sigma,V}(R_{\varepsilon})$ is of the form
\begin{equation*}
    \Si_d \begin{pmatrix}
\varepsilon \\
q  \end{pmatrix} \Si_{d} \begin{pmatrix}
 \Sigma  \\
V  \end{pmatrix}  + \Si_d \begin{pmatrix}
 \varepsilon \\
q  \end{pmatrix} \Si_{<d} \begin{pmatrix}
 \Sigma  \\
V  \end{pmatrix}  + \varepsilon^{-1} D_1 \Si_{d+1} \begin{pmatrix}
\varepsilon&  1  \\
1 & q-1 \end{pmatrix} \Si_{<d+1} \begin{pmatrix}
 \Sigma  \\
V  \end{pmatrix}  = 0.
\end{equation*}
It follows from \eqref{redsum} and Proposition \ref{polysums} that
\begin{align*}
        \Si_d \begin{pmatrix}
\varepsilon \\
q  \end{pmatrix} \Si_{d} \begin{pmatrix}
 \Sigma  \\
V  \end{pmatrix}  + \Si_d \begin{pmatrix}
 \varepsilon \\
q  \end{pmatrix} \Si_{<d} \begin{pmatrix}
 \Sigma  \\
V  \end{pmatrix}  &=   \Si_d \begin{pmatrix}
 \varepsilon\sigma_1 & \Sigma_{-} \\
q + v_1 & V_{-} \end{pmatrix}  +  \Si_d \begin{pmatrix}
 \varepsilon& \Sigma \\
q & V \end{pmatrix},  \\
  \varepsilon^{-1}D_1 \Si_{d+1} \begin{pmatrix}
\varepsilon&  1  \\
1 & q-1 \end{pmatrix} \Si_{<d+1} \begin{pmatrix}
 \Sigma  \\
V  \end{pmatrix}   &= \sum \limits_i b_i \Si_{d+1} \begin{pmatrix}
 \varepsilon& \fe_i  \\
1 & \mathfrak{t}_i  \end{pmatrix} ,
\end{align*}
where $b_i \in A$ are divisible by $D_1$ and $ \begin{pmatrix}
 \fe_i  \\
\mathfrak{t}_i  \end{pmatrix} $ are arrays satisfying $ \begin{pmatrix}
 \fe_i  \\
\mathfrak{t}_i  \end{pmatrix}  \leq  \begin{pmatrix}
 1  \\
q - 1 \end{pmatrix}  +  \begin{pmatrix}
 \Sigma  \\
V  \end{pmatrix} $ for all $i$.
This proves the proposition.
\end{proof}

\subsubsection{Operators $\mathcal{BC}$}

Let $\varepsilon \in \mathbb{F}_q^{\times}$. We define an operator
\begin{equation*}
   \mathcal{BC}_{\varepsilon,q} \colon \mathfrak{BR}_{w} \longrightarrow \mathfrak{BR}_{w+q}
\end{equation*}
as follows: for each $R \in \mathfrak{BR}_{w}$ given as in \eqref{eq: Rd0},
the image $\mathcal{BC}_{\varepsilon,q}(R)$ is a binary relation given by
\begin{align*}
    \mathcal{BC}_{\varepsilon,q}(R) = \mathcal B^*_{\varepsilon,q}(R) - \sum\limits_i b_i \mathcal C_{\fe_i,\mathfrak{t}_i} (R_{\varepsilon}).
\end{align*}

Let us clarify the definition of $\mathcal{BC}_{\varepsilon,q}$. We know that $\mathcal B^*_{\varepsilon,q}(R)$ is of the form
\begin{equation*}
    \sum \limits_i a_i \Si_d \begin{pmatrix}
\varepsilon& \fve_i  \\
q& \fs_i  \end{pmatrix}  + \sum \limits_i  b_i \Si_d \begin{pmatrix}
\varepsilon& \fe_i  \\
q& \mathfrak{t}_i  \end{pmatrix}  + \sum \limits_i  b_i \Si_d \begin{pmatrix}
 \varepsilon\epsilon_{i1} &  \fe_{i-} \\
q + t_{i1} & \mathfrak{t}_{i-} \end{pmatrix} = 0.
\end{equation*}
Moreover, $\mathcal C_{\fe_i,\mathfrak{t}_i} (R_{\varepsilon})$ is of the form
\begin{equation*}
\Si_d \begin{pmatrix}
\varepsilon& \fe_i  \\
q& \mathfrak{t}_i  \end{pmatrix}  +  \Si_d \begin{pmatrix}
 \varepsilon\epsilon_{i1} &  \fe_{i-} \\
q + t_{i1} & \mathfrak{t}_{i-} \end{pmatrix}  + \varepsilon^{-1}D_1 \Si_{d+1} \begin{pmatrix}
 \varepsilon \\
1  \end{pmatrix} \Si_{<d+1} \begin{pmatrix}
1  \\
q-1  \end{pmatrix} \Si_{<d+1} \begin{pmatrix}
\fe_i  \\
\mathfrak{t}_i  \end{pmatrix} = 0.
\end{equation*}
Combining with Proposition \ref{polysums}, Part 2, we have that $\mathcal{BC}_{\varepsilon,q}(R)$ is of the form
\begin{equation*}
   \sum \limits_i a_i \Si_d \begin{pmatrix}
\varepsilon& \fve_i  \\
q& \fs_i  \end{pmatrix}  + \sum \limits_{i,j} b_{ij} \Si_{d+1} \begin{pmatrix}
\varepsilon& \fe_{ij}  \\
1& \mathfrak{t}_{ij}  \end{pmatrix} =0,
\end{equation*}
where $b_{ij} \in K$ and $ \begin{pmatrix}
\fe_{ij}  \\
\mathfrak{t}_{ij}  \end{pmatrix} $ are arrays satisfying $ \begin{pmatrix}
\fe_{ij}  \\
\mathfrak{t}_{ij}  \end{pmatrix}  \leq  \begin{pmatrix}
1  \\
q-1  \end{pmatrix}  +  \begin{pmatrix}
\fe_{i}  \\
\mathfrak{t}_{i}  \end{pmatrix} $ for all $j$.

\subsection{A weak version of Brown's theorem for ACMP's} \label{sec:weak Brown}

\subsubsection{Preparatory results}

\begin{proposition}\label{polydecom}
1) Let $ \begin{pmatrix}
 \fve  \\
\fs  \end{pmatrix}  =  \begin{pmatrix}
 \varepsilon_1 & \dots & \varepsilon_n \\
s_1 & \dots & s_n \end{pmatrix} $ be an array such that $\Init(\fs) = (s_1, \dots, s_{k-1})$ for some $1 \leq k \leq n$, and let $\varepsilon$ be an element in $\mathbb{F}_q^{\times}$.  Then $\Li \begin{pmatrix}
 \fve  \\
\fs  \end{pmatrix} $ can be decomposed as follows:
    \begin{equation*}
    \Li \begin{pmatrix}
 \fve  \\
\fs  \end{pmatrix}
    = \underbrace{ - \Li \begin{pmatrix}
 \fve'  \\
\fs'  \end{pmatrix} }_\text{type 1} + \underbrace{\sum\limits_i b_i\Li \begin{pmatrix}
 \fe_i'  \\
\mathfrak{t}'_i  \end{pmatrix} }_\text{type 2} + \underbrace{\sum\limits_i c_i\Li \begin{pmatrix}
 \fm_i  \\
\mathfrak{u}_i  \end{pmatrix} }_\text{type 3}  ,
    \end{equation*}
    where $ b_i, c_i \in A$ are divisible by $D_1$ such that for all $i$, the following properties are satisfied:
    \begin{itemize}
        \item For all arrays $ \begin{pmatrix}
 \fe  \\
\mathfrak{t}  \end{pmatrix} $ appearing on the right hand side,
        \begin{equation*}
            \depth(\mathfrak{t}) \geq \depth(\fs) \quad \text{and} \quad T_k(\mathfrak{t}) \leq T_k(\fs).
     \end{equation*}
     \item For the array $ \begin{pmatrix}
 \fve'  \\
\fs'  \end{pmatrix} $ of type $1$ with respect to $ \begin{pmatrix}
 \fve  \\
\fs  \end{pmatrix} $, we have
\begin{align*}
\begin{pmatrix}
 \fve'  \\
\fs'  \end{pmatrix} =
\begin{pmatrix}
\varepsilon_1 & \dots & \varepsilon_{k-1} & \varepsilon & \varepsilon^{-1}\varepsilon_{k} & \varepsilon_{k+1} & \dots & \varepsilon_n \\
s_1 & \dots & s_{k-1} & q &  s_k- q & s_{k+1} & \dots & s_n \end{pmatrix}.
\end{align*}
Moreover, for all $k \leq \ell \leq n$,
        \begin{equation*}
s'_{1} +  \dots + s'_\ell < s_1 +  \dots + s_\ell.
     \end{equation*}

\item For the array $ \begin{pmatrix}
 \fe'  \\
\mathfrak{t}'  \end{pmatrix} $ of type $2$ with respect to $ \begin{pmatrix}
 \fve  \\
\fs  \end{pmatrix} $, for all $k \leq \ell \leq n$,
        \begin{equation*}
t'_{1} +  \dots + t'_\ell < s_1 +  \dots + s_\ell.
     \end{equation*}

        \item For the array $ \begin{pmatrix}
 \fm  \\
\mathfrak{u}  \end{pmatrix} $ of type $3$ with respect to $ \begin{pmatrix}
 \fve  \\
\fs  \end{pmatrix} $, we have $\Init(\fs) \prec\Init(\mathfrak{u})$.
    \end{itemize}

\noindent 2) Let $ \begin{pmatrix}
 \fve  \\
\fs  \end{pmatrix}  =  \begin{pmatrix}
 \varepsilon_1 & \dots & \varepsilon_k \\
s_1 & \dots & s_k \end{pmatrix} $ be an array such that $\Init(\fs) = \fs$ and $s_k = q$. Then $\Li \begin{pmatrix}
 \fve  \\
\fs  \end{pmatrix} $ can be decomposed as follows:
    \begin{equation*}
    \Li \begin{pmatrix}
 \fve  \\
\fs  \end{pmatrix}
    =   \underbrace{\sum\limits_i b_i\Li \begin{pmatrix}
 \fe'_i  \\
\mathfrak{t}'_i  \end{pmatrix} }_\text{type 2} + \underbrace{\sum\limits_i c_i\Li \begin{pmatrix}
 \fm_i  \\
\mathfrak{u}_i  \end{pmatrix} }_\text{type 3}  ,
    \end{equation*}
    where $ b_i, c_i \in A$ divisible by $D_1$ such that for all $i$, the following properties are satisfied:
    \begin{itemize}
        \item For all arrays $ \begin{pmatrix}
 \fe  \\
\mathfrak{t}  \end{pmatrix} $ appearing on the right hand side,
        \begin{equation*}
            \depth(\mathfrak{t}) \geq \depth(\fs) \quad \text{and} \quad T_k(\mathfrak{t}) \leq T_k(\fs).
     \end{equation*}

\item For the array $ \begin{pmatrix}
 \fe'  \\
\mathfrak{t}'  \end{pmatrix} $ of type $2$ with respect to $ \begin{pmatrix}
 \fve  \\
\fs  \end{pmatrix} $,
        \begin{equation*}
t'_{1} +  \dots + t'_k < s_1 +  \dots + s_k.
     \end{equation*}

        \item For the array $ \begin{pmatrix}
 \fm  \\
\mathfrak{u}  \end{pmatrix} $ of type $3$ with respect to $ \begin{pmatrix}
 \fve  \\
\fs  \end{pmatrix} $, we have $\Init(\fs) \prec\Init(\mathfrak{u})$.
\end{itemize}
\end{proposition}

\begin{proof}
For Part 1, since $\Init(\fs) = (s_1, \dots, s_{k-1})$, we get $s_k > q$. Set $ \begin{pmatrix}
\Sigma  \\
V  \end{pmatrix}  =  \begin{pmatrix}
\varepsilon^{-1} \varepsilon_{k} & \varepsilon_{k+1} &\dots & \varepsilon_n \\
s_k - q & s_{k+1} &\dots & s_n \end{pmatrix} $. By Proposition \ref{polycer1}, $\mathcal{C}_{\Sigma,V}(R_{\varepsilon})$ is of the form
\begin{equation} \label{polyr1}
  \Si_d \begin{pmatrix}
 \varepsilon_{k} & \dots & \varepsilon_n \\
s_{k} & \dots & s_n \end{pmatrix}  +  \Si_d \begin{pmatrix}
 \varepsilon & \varepsilon^{-1}\varepsilon_{k} & \varepsilon_{k+1} & \dots & \varepsilon_n \\
q &  s_k- q & s_{k+1} & \dots & s_n \end{pmatrix}  + \sum \limits_i b_i \Si_{d+1} \begin{pmatrix}
\varepsilon & \fe_i  \\
1 & \mathfrak{t}_i  \end{pmatrix}  =0,
\end{equation}
where $b_i \in A$ divisible by $D_1$ and $ \begin{pmatrix}
 \fe_i  \\
\mathfrak{t}_i  \end{pmatrix} $ are arrays satisfying for all $i$,
\begin{equation*}
     \begin{pmatrix}
 \fe_i  \\
\mathfrak{t}_i  \end{pmatrix}  \leq  \begin{pmatrix}
 1  \\
q - 1 \end{pmatrix}  +  \begin{pmatrix}
 \Sigma  \\
V  \end{pmatrix}  =  \begin{pmatrix}
\varepsilon^{-1} \varepsilon_{k} & \varepsilon_{k+1} &\dots & \varepsilon_n \\
s_k - 1 & s_{k+1} &\dots & s_n \end{pmatrix} .
\end{equation*}

For $m \in \mathbb{N}$, we denote by $q^{\{m\}}$ the sequence of length $m$ with all terms equal to $q$. We agree by convention that $q^{\{0\}}$ is the empty sequence. Setting $s_0 = 0$, we may assume that there exists a maximal index $j$ with $0 \leq j \leq k-1$ such that $s_j < q$, hence $\Init(\fs) = (s_1, \dots, s_j, q^{\{k-j-1\}})$.

Then the operator $ \mathcal{BC}_{\varepsilon_{j+1},q} \circ \dots \circ \mathcal{BC}_{\varepsilon_{k-1},q}$ applied to the relation \eqref{polyr1} gives
\begin{align*}
 \Si_d &\begin{pmatrix}
 \varepsilon_{j+1} & \dots & \varepsilon_{k-1} & \varepsilon_{k} & \dots & \varepsilon_n \\
q & \dots & q & s_{k} & \dots & s_n \end{pmatrix}  \\
& +  \Si_d \begin{pmatrix}
\varepsilon_{j+1} & \dots & \varepsilon_{k-1} & \varepsilon & \varepsilon^{-1}\varepsilon_{k} & \epsilon_{k+1} & \dots & \epsilon_n \\
q & \dots & q & q &  s_k- q & s_{k+1} & \dots & s_n \end{pmatrix}  \\
&
+ \sum \limits_i b_{i_1 \dots i_{k-j}} \Si_{d+1} \begin{pmatrix}
\varepsilon_{j+1} & \fe_{i_1 \dots i_{k-j}}  \\
1 & \mathfrak{t}_{i_1 \dots i_{k-j}}  \end{pmatrix}  =0,
\end{align*}
where $b_{i_1 \dots i_{k-j}} \in A$ are divisible by $D_1$ and for $2 \leq l \leq k-j$, $ \begin{pmatrix}
 \fe_{i_1 \dots i_l}  \\
 \mathfrak{t}_{i_1 \dots i_l}  \end{pmatrix} $ are arrays satisfying
\begin{equation*}
     \begin{pmatrix}
 \fe_{i_1 \dots i_l}  \\
 \mathfrak{t}_{i_1 \dots i_l}  \end{pmatrix}  \leq  \begin{pmatrix}
 1  \\
q - 1 \end{pmatrix}  +  \begin{pmatrix}
\varepsilon_{k-l+2} & \fe_{i_1 \dots i_{l-1}}  \\
1 & \mathfrak{t}_{i_1 \dots i_{l-1}}  \end{pmatrix}  =  \begin{pmatrix}
\varepsilon_{k-l+2} & \fe_{i_1 \dots i_{l-1}}  \\
q & \mathfrak{t}_{i_1 \dots i_{l-1}}  \end{pmatrix} .
\end{equation*}

Thus 
\begin{equation} \label{eq: part 1 compa}
     \begin{pmatrix}
 \fe_{i_1 \dots i_{k-j}}  \\
 \mathfrak{t}_{i_1 \dots i_{k-j}}  \end{pmatrix}  \leq  \begin{pmatrix}
 \varepsilon_{j+2}& \dots & \varepsilon_{k-1} & \varepsilon & \varepsilon^{-1}\varepsilon_{k} & \varepsilon_{k+1}& \dots & \varepsilon_n \\
q& \dots & q &  q & s_{k} - 1 & s_{k+1} &\dots & s_n \end{pmatrix} .
\end{equation}

Letting $ \begin{pmatrix}
\Sigma'  \\
V'  \end{pmatrix}  =  \begin{pmatrix}
 \varepsilon_{1} &\dots & \varepsilon_j \\
 s_1 &\dots & s_j \end{pmatrix} $, by Lemma \ref{polybesao}, we continue to apply $\mathcal{B}^*_{\Sigma',V'}$ to the above relation and get
\begin{align*}
  \Si_d \begin{pmatrix}
\fve \\ \fs \end{pmatrix}  &+  \Si_d \begin{pmatrix}
\varepsilon_1 & \dots & \varepsilon_{k-1} & \varepsilon & \varepsilon^{-1}\varepsilon_{k} & \varepsilon_{k+1} & \dots & \varepsilon_n \\
s_1 & \dots & s_{k-1} & q &  s_k- q & s_{k+1} & \dots & s_n \end{pmatrix}  \\
& + \sum \limits_i b_{i_1 \dots i_{k-j}} \Si_d \begin{pmatrix}
\varepsilon_1 & \dots & \varepsilon_j & \varepsilon_{j+1} & \fe_{i_1 \dots i_{k-j}}  \\
s_1 & \dots & s_j & 1 & \mathfrak{t}_{i_1 \dots i_{k-j}}  \end{pmatrix}  \\
& + \sum \limits_i b_{i_1 \dots i_{k-j}} \Si_d \begin{pmatrix}
\varepsilon_1 & \dots & \varepsilon_{j-1} & \varepsilon_j \varepsilon_{j+1} & \fe_{i_1 \dots i_{k-j}}  \\
s_1 & \dots & s_{j-1} & s_j+1 & \mathfrak{t}_{i_1 \dots i_{k-j}}  \end{pmatrix}  =0.
\end{align*}
Hence
\begin{align*}
  \Li \begin{pmatrix}
\fve \\ \fs \end{pmatrix}  &+  \Li \begin{pmatrix}
\varepsilon_1 & \dots & \varepsilon_{k-1} & \varepsilon & \varepsilon^{-1}\varepsilon_{k} & \varepsilon_{k+1} & \dots & \varepsilon_n \\
s_1 & \dots & s_{k-1} & q &  s_k- q & s_{k+1} & \dots & s_n \end{pmatrix}  \\
& + \sum \limits_i b_{i_1 \dots i_{k-j}} \Li \begin{pmatrix}
\varepsilon_1 & \dots & \varepsilon_j & \varepsilon_{j+1} & \fe_{i_1 \dots i_{k-j}}  \\
s_1 & \dots & s_j & 1 & \mathfrak{t}_{i_1 \dots i_{k-j}}  \end{pmatrix}  \\
& + \sum \limits_i b_{i_1 \dots i_{k-j}} \Li \begin{pmatrix}
\varepsilon_1 & \dots & \varepsilon_{j-1} & \varepsilon_j \varepsilon_{j+1} & \fe_{i_1 \dots i_{k-j}}  \\
s_1 & \dots & s_{j-1} & s_j+1 & \mathfrak{t}_{i_1 \dots i_{k-j}}  \end{pmatrix}  =0.
\end{align*}

In other words, we have
\begin{align} \label{eq: part 1 Li}
  \Li \begin{pmatrix}
\fve \\ \fs \end{pmatrix} =&  -  \Li \begin{pmatrix}
\varepsilon_1 & \dots & \varepsilon_{k-1} & \varepsilon & \varepsilon^{-1}\varepsilon_{k} & \varepsilon_{k+1} & \dots & \varepsilon_n \\
s_1 & \dots & s_{k-1} & q &  s_k- q & s_{k+1} & \dots & s_n \end{pmatrix}  \\ \notag
&  - \sum \limits_i b_{i_1 \dots i_{k-j}} \Li \begin{pmatrix}
\varepsilon_1 & \dots & \varepsilon_j & \varepsilon_{j+1} & \fe_{i_1 \dots i_{k-j}}  \\
s_1 & \dots & s_j & 1 & \mathfrak{t}_{i_1 \dots i_{k-j}}  \end{pmatrix}  \\\notag
&  - \sum \limits_i b_{i_1 \dots i_{k-j}} \Li \begin{pmatrix}
\varepsilon_1 & \dots & \varepsilon_{j-1} & \varepsilon_j \varepsilon_{j+1} & \fe_{i_1 \dots i_{k-j}}  \\
s_1 & \dots & s_{j-1} & s_j+1 & \mathfrak{t}_{i_1 \dots i_{k-j}}  \end{pmatrix}.
\end{align}
The first term, the second term, and the third term on the right hand side of \eqref{eq: part 1 Li} are referred to as type 1, type 2, and type 3 respectively.

We now verify the conditions of arrays of type 1, type 2, and type 3 with respect to $ \begin{pmatrix}
\fve  \\
\fs  \end{pmatrix} $. We first note that $\fs = (s_1, \dots, s_j, q^{\{k-j-1\}}, s_k, \dotsc, s_n)$.

\textit{Type 1: } For $\begin{pmatrix}
 \fve'  \\
\fs'  \end{pmatrix} =
\begin{pmatrix}
\varepsilon_1 & \dots & \varepsilon_{k-1} & \varepsilon & \varepsilon^{-1}\varepsilon_{k} & \varepsilon_{k+1} & \dots & \varepsilon_n \\
s_1 & \dots & s_{k-1} & q &  s_k- q & s_{k+1} & \dots & s_n \end{pmatrix}$, we have
\begin{align*}
    \depth(\fs') = (k - 1) + 1 + (n - k + 1) = n + 1 > \depth(\fs).
\end{align*}
For $\ell = k$, since $s_k > q$, we have 
\begin{align*}
    s'_1 + \cdots + s'_k = s_1 + \cdots + s_{k - 1} + q < s_1 + \cdots + s_{k}.
\end{align*}
For $k < \ell \leq n$, one verifies that
\begin{align*}
    s'_1 + \cdots + s'_{\ell} = s_1 + \cdots + s_{\ell - 1} < s_1 + \cdots + s_{\ell}.
\end{align*}
Since $w(\fs') = w(\fs)$, one deduces that $T_k(\fs') \leq T_k(\fs)$.

\textit{Type 2:} For $ \begin{pmatrix}
 \fe'  \\
\mathfrak{t}'  \end{pmatrix} = \begin{pmatrix}
\varepsilon_1 & \dots & \varepsilon_j & \varepsilon_{j+1} & \fe_{i_1 \dots i_{k-j}}  \\
s_1 & \dots & s_j & 1 & \mathfrak{t}_{i_1 \dots i_{k-j}}  \end{pmatrix}$, it follows from \eqref{eq: part 1 compa} and Remark \ref{rmk: compa depth} that 
\begin{align*}
    \depth(\mathfrak{t}')  &= j + 1 + \depth(\mathfrak{t}_{i_1 \dots i_{k-j}})\\
    &\geq j + 1 + (k - j - 2) + 1 + (n - k + 1) = n + 1  > \depth(\fs).
\end{align*}
For $\ell = k$, since $s_k > q$, it follows from \eqref{eq: part 1 compa} that
\begin{align*}
    t'_1 + \cdots + t'_k \leq  s_1 + \cdots + s_j + 1 + q(k - j - 1) < s_1 + \cdots + s_{k}.
\end{align*}
For $k < \ell \leq n$, it follows from \eqref{eq: part 1 compa} that
\begin{align*}
    t'_1 + \cdots + t'_{\ell} \leq s_1 + \cdots + s_{\ell - 1} < s_1 + \cdots + s_{\ell}.
\end{align*}
Since $w(\mathfrak{t}') = w(\fs)$, one deduces that $T_k(\mathfrak{t}') \leq T_k(\fs)$.

\textit{Type 3:} For $ \begin{pmatrix}
 \fm  \\
\mathfrak{u}  \end{pmatrix} = \begin{pmatrix}
\varepsilon_1 & \dots & \varepsilon_{j-1} & \varepsilon_j \varepsilon_{j+1} & \fe_{i_1 \dots i_{k-j}}  \\
s_1 & \dots & s_{j-1} & s_j+1 & \mathfrak{t}_{i_1 \dots i_{k-j}}  \end{pmatrix}$, it follows from \eqref{eq: part 1 compa} and Remark \ref{rmk: compa depth} that 
\begin{align*}
    \depth(\mathfrak{u})  &= j + \depth(\mathfrak{t}_{i_1 \dots i_{k-j}}) \geq j + (k - j - 2) + 1 + (n - k + 1)  = \depth(\fs).
\end{align*}
For $k \leq \ell \leq n$, it follows from \eqref{eq: part 1 compa} that
\begin{align*}
    u_1 + \cdots + u_{\ell} \leq s_1 + \cdots + s_{\ell}.
\end{align*}
Since $w(\mathfrak{u}) = w(\fs)$, one deduces that $T_k(\mathfrak{u}) \leq T_k(\fs)$. Moreover, we have $\Init(\fs) \prec (s_1, \dotsc, s_{j - 1}, s_j + 1) \preceq \Init(\mathfrak{u})$. We have proved Part 1.

For Part $2$, we recall the relation $R_{\varepsilon_k}$ given by
\begin{equation*}
 R_{\varepsilon_k} \colon \quad  \Si_d \begin{pmatrix}
 \varepsilon_k\\
q  \end{pmatrix}  + \varepsilon_k^{-1}D_1 \Si_{d+1} \begin{pmatrix}
 \varepsilon_k& 1 \\
1 & q-1  \end{pmatrix}  =0.
\end{equation*}
Setting $s_0 = 0$, we may assume that there exists a maximal index $j$ with $0 \leq j \leq k-1$ such that $s_j < q$. We note that $s_k = q$, hence $\Init(\fs) = \fs = (s_1, \dots, s_j, q^{\{k-j-1\}}, q)$. Then  $ \mathcal{BC}_{\varepsilon_{j+1},q} \circ \dots \circ \mathcal{BC}_{\varepsilon_{k-1},q}(R_{\varepsilon_k})$ is of the form
\begin{align} \label{eq: Part 2 Li}
 \Si_d &\begin{pmatrix}
 \varepsilon_{j+1} & \dots & \varepsilon_{k} \\
q & \dots & q \end{pmatrix} + \sum \limits_i b_{i_1 \dots i_{k-j}} \Si_{d+1} \begin{pmatrix}
\varepsilon_{j+1} & \fe_{i_1 \dots i_{k-j}}  \\
1 & \mathfrak{t}_{i_1 \dots i_{k-j}}  \end{pmatrix}  =0,
\end{align}
where $b_{i_1 \dots i_{k-j}} \in A$ are divisible by $D_1$ and for $2 \leq l \leq k-j$, $ \begin{pmatrix}
 \fe_{i_1 \dots i_l}  \\
 \mathfrak{t}_{i_1 \dots i_l}  \end{pmatrix} $ are arrays satisfying
\begin{equation*}
     \begin{pmatrix}
 \fe_{i_1 \dots i_l}  \\
 \mathfrak{t}_{i_1 \dots i_l}  \end{pmatrix}  \leq  \begin{pmatrix}
 1  \\
q - 1 \end{pmatrix}  +  \begin{pmatrix}
\varepsilon_{k-l+2} & \fe_{i_1 \dots i_{l-1}}  \\
1 & \mathfrak{t}_{i_1 \dots i_{l-1}}  \end{pmatrix}  =  \begin{pmatrix}
\varepsilon_{k-l+2} & \fe_{i_1 \dots i_{l-1}}  \\
q & \mathfrak{t}_{i_1 \dots i_{l-1}}  \end{pmatrix} .
\end{equation*}
Thus 
\begin{equation} \label{eq: part 2 compa}
     \begin{pmatrix}
 \fe_{i_1 \dots i_{k-j}}  \\
 \mathfrak{t}_{i_1 \dots i_{k-j}}  \end{pmatrix}  \leq  \begin{pmatrix}
 \varepsilon_{j+2}& \dots & \varepsilon_{k} & 1 \\
q& \dots & q &  q - 1 \end{pmatrix} .
\end{equation}

Letting $ \begin{pmatrix}
\Sigma'  \\
V'  \end{pmatrix}  =  \begin{pmatrix}
 \varepsilon_{1} &\dots & \varepsilon_j \\
 s_1 &\dots & s_j \end{pmatrix} $, by Lemma \ref{polybesao}, we continue to apply $\mathcal{B}^*_{\Sigma',V'}$ to \eqref{eq: Part 2 Li} and get
\begin{align*}
  \Si_d \begin{pmatrix}
\fve \\ \fs \end{pmatrix}  & + \sum \limits_i b_{i_1 \dots i_{k-j}} \Si_d \begin{pmatrix}
\varepsilon_1 & \dots & \varepsilon_j & \varepsilon_{j+1} & \fe_{i_1 \dots i_{k-j}}  \\
s_1 & \dots & s_j & 1 & \mathfrak{t}_{i_1 \dots i_{k-j}}  \end{pmatrix}  \\
& + \sum \limits_i b_{i_1 \dots i_{k-j}} \Si_d \begin{pmatrix}
\varepsilon_1 & \dots & \varepsilon_{j-1} & \varepsilon_j \varepsilon_{j+1} & \fe_{i_1 \dots i_{k-j}}  \\
s_1 & \dots & s_{j-1} & s_j+1 & \mathfrak{t}_{i_1 \dots i_{k-j}}  \end{pmatrix}  =0,
\end{align*}
hence 
\begin{align*}
  \Li \begin{pmatrix}
\fve \\ \fs \end{pmatrix}  & + \sum \limits_i b_{i_1 \dots i_{k-j}} \Li \begin{pmatrix}
\varepsilon_1 & \dots & \varepsilon_j & \varepsilon_{j+1} & \fe_{i_1 \dots i_{k-j}}  \\
s_1 & \dots & s_j & 1 & \mathfrak{t}_{i_1 \dots i_{k-j}}  \end{pmatrix}  \\
& + \sum \limits_i b_{i_1 \dots i_{k-j}} \Li \begin{pmatrix}
\varepsilon_1 & \dots & \varepsilon_{j-1} & \varepsilon_j \varepsilon_{j+1} & \fe_{i_1 \dots i_{k-j}}  \\
s_1 & \dots & s_{j-1} & s_j+1 & \mathfrak{t}_{i_1 \dots i_{k-j}}  \end{pmatrix}  =0.
\end{align*}
In other words, we have
\begin{align} \label{eq: part 2 Li 2}
  \Li \begin{pmatrix}
\fve \\ \fs \end{pmatrix}  =& - \sum \limits_i b_{i_1 \dots i_{k-j}} \Li \begin{pmatrix}
\varepsilon_1 & \dots & \varepsilon_j & \varepsilon_{j+1} & \fe_{i_1 \dots i_{k-j}}  \\
s_1 & \dots & s_j & 1 & \mathfrak{t}_{i_1 \dots i_{k-j}}  \end{pmatrix}  \\ \notag
& - \sum \limits_i b_{i_1 \dots i_{k-j}} \Li \begin{pmatrix}
\varepsilon_1 & \dots & \varepsilon_{j-1} & \varepsilon_j \varepsilon_{j+1} & \fe_{i_1 \dots i_{k-j}}  \\
s_1 & \dots & s_{j-1} & s_j+1 & \mathfrak{t}_{i_1 \dots i_{k-j}}  \end{pmatrix}.
\end{align}
The first term and the second term on the right hand side of \eqref{eq: part 2 Li 2} are referred to as type 2 and type 3 respectively.

We now verify the conditions of arrays of type 2 and type 3 with respect to $ \begin{pmatrix}
\fve  \\
\fs  \end{pmatrix} $.

\textit{Type 2:} For $ \begin{pmatrix}
 \fe'  \\
\mathfrak{t}'  \end{pmatrix} = \begin{pmatrix}
\varepsilon_1 & \dots & \varepsilon_j & \varepsilon_{j+1} & \fe_{i_1 \dots i_{k-j}}  \\
s_1 & \dots & s_j & 1 & \mathfrak{t}_{i_1 \dots i_{k-j}}  \end{pmatrix}$, it follows from \eqref{eq: part 2 compa} and Remark \ref{rmk: compa depth} that 
\begin{align*}
    \depth(\mathfrak{t}')  = j + 1 + \depth(\mathfrak{t}_{i_1 \dots i_{k-j}}) \geq j + 1 + (k - j)= k + 1  > \depth(\fs).
\end{align*}
It follows from \eqref{eq: part 2 compa} that
\begin{align*}
    t'_1 + \cdots + t'_k \leq  s_1 + \cdots + s_j + 1 + q(k - j - 1) < s_1 + \cdots + s_{k}.
\end{align*}
Since $w(\mathfrak{t}') = w(\fs)$, one deduces that $T_k(\mathfrak{t}') \leq T_k(\fs)$.

\textit{Type 3:} For $ \begin{pmatrix}
 \fm  \\
\mathfrak{u}  \end{pmatrix} = \begin{pmatrix}
\varepsilon_1 & \dots & \varepsilon_{j-1} & \varepsilon_j \varepsilon_{j+1} & \fe_{i_1 \dots i_{k-j}}  \\
s_1 & \dots & s_{j-1} & s_j+1 & \mathfrak{t}_{i_1 \dots i_{k-j}}  \end{pmatrix}$, it follows from \eqref{eq: part 2 compa} and Remark \ref{rmk: compa depth} that 
\begin{align*}
    \depth(\mathfrak{u})  = j + \depth(\mathfrak{t}_{i_1 \dots i_{k-j}}) \geq j + (k - j) = \depth(\fs).
\end{align*}
It follows from \eqref{eq: part 2 compa} that
\begin{align*}
    u_1 + \cdots + u_k \leq s_1 + \cdots + s_{j - 1} + (s_{j} + 1) + q(k - j - 1) + (q - 1) = s_1 + \cdots + s_k.
\end{align*}
Since $w(\mathfrak{u}) = w(\fs)$, one deduces that $T_k(\mathfrak{u}) \leq T_k(\fs)$. Moreover, we have $\Init(\fs) \prec (s_1, \dotsc, s_{j - 1}, s_j + 1) \preceq \Init(\mathfrak{u})$. We finish the proof.
\end{proof}

We recall the following definition of \cite{ND21} (see \cite[Definition 3.1]{ND21}):
\begin{definition}
Let $k \in \mathbb N$ and $\fs$ be a tuple of positive integers. We say that $\fs$ is $k$-admissible if it satisfies the following two conditions:
\begin{itemize}
\item[1)] $s_1,\dots,s_k \leq q$.

\item[2)] $\fs$ is not of the form $(s_1,\dots,s_r)$ with $r \leq k$, $s_1,\dots,s_{r-1} \leq q$, and $s_r=q$.
\end{itemize}
Here we recall $s_i=0$ for $i > \depth(\fs)$. By convention the empty array $\begin{pmatrix}
 \emptyset  \\
 \emptyset  \end{pmatrix}$ is always $k$-admissible.

An array is $k$-admissible if the corresponding tuple is $k$-admissible.
\end{definition}

\begin{proposition} \label{polyalgpartlem}
For all $k \in \mathbb{N}$ and for all arrays $ \begin{pmatrix}
 \fve  \\
\fs  \end{pmatrix} $, $\Li \begin{pmatrix}
 \fve  \\
\fs  \end{pmatrix} $ can be expressed as a $K$-linear combination of $\Li \begin{pmatrix}
 \fe  \\
\mathfrak{t}  \end{pmatrix} $'s of the same weight such that $\mathfrak{t}$ is $k$-admissible.
\end{proposition}

\begin{proof}
The proof follows the same line as that of \cite[Proposition 3.2]{ND21}. We write down the proof for the reader's convenience. 

We consider the following statement:\\
$(H_k) \quad$ For all arrays $ \begin{pmatrix}
 \fve  \\
\fs  \end{pmatrix} $, we can express $\Li \begin{pmatrix}
 \fve  \\ 
\fs  \end{pmatrix} $  as a $K$-linear combination of $\Li \begin{pmatrix}
 \fe  \\
\mathfrak{t}  \end{pmatrix} $'s of the same weight such that $\mathfrak{t}$ is $k$-admissible.\\

We need to show that $(H_k)$ holds for all $k \in \mathbb{N}$. We proceed the proof by induction on $k$. For $k = 1$, we consider all the cases for the first component $s_1$ of $\fs$. If $s_1 \leq q$, then either $\fs$ is  $1$-admissible,  or $\begin{pmatrix}
 \fve  \\ 
\fs  \end{pmatrix} = \begin{pmatrix}
 \varepsilon \\ 
q \end{pmatrix}$. For the case $\begin{pmatrix}
 \fve  \\ 
\fs  \end{pmatrix} = \begin{pmatrix}
 \varepsilon \\ 
q \end{pmatrix}$, we deduce from the relation $R_{\varepsilon}$ that 
\begin{equation*}
\Li \begin{pmatrix}
 \varepsilon\\
q  \end{pmatrix}  = - \varepsilon^{-1}D_1 \Li \begin{pmatrix}
 \varepsilon& 1 \\
1 & q-1  \end{pmatrix},
\end{equation*}
hence $(H_1)$ holds in this case since $(1, q - 1)$ is $1$-admissible. If $s_1 > q$, we assume that $ \begin{pmatrix}
 \fve  \\ 
\fs  \end{pmatrix} = \begin{pmatrix}
 \varepsilon_1 & \cdots & \varepsilon_n  \\ 
s_1 & \cdots & s_n  \end{pmatrix}$. Set $\begin{pmatrix}
 \Sigma \\ 
V  \end{pmatrix} = \begin{pmatrix}
 \varepsilon_1 & \varepsilon_2 & \cdots & \varepsilon_n  \\ 
s_1 - q & s_2 & \cdots & s_n  \end{pmatrix}$. From Proposition \ref{polycer1}, we have $\mathcal C_{\Sigma,V}(R_{1})$ is of the form
\begin{equation*}
  \Si_d \begin{pmatrix}
 \varepsilon_1 & \cdots & \varepsilon_n \\
s_1 & \cdots & s_n \end{pmatrix}  +  \Si_d \begin{pmatrix}
 1& \varepsilon_1 & \varepsilon_2 & \cdots & \varepsilon_n  \\
q & s_1 - q & s_2 & \cdots & s_n \end{pmatrix}  + \sum \limits_i b_i \Si_{d+1} \begin{pmatrix}
1 & \fe_i  \\
1 & \mathfrak{t}_i  \end{pmatrix}  =0,
\end{equation*}
where $b_i \in K$ for all $i$. Thus we deduce that 
\begin{align*}
     \Li \begin{pmatrix}
 \fve  \\ 
\fs  \end{pmatrix} =  -  \Li \begin{pmatrix}
 1& \varepsilon_1 & \varepsilon_2 & \cdots & \varepsilon_n  \\
q & s_1 - q & s_2 & \cdots & s_n \end{pmatrix}  -  \sum \limits_i b_i \Li \begin{pmatrix}
1 & \fe_i  \\
1 & \mathfrak{t}_i  \end{pmatrix},
\end{align*}
which proves that $(H_1)$ holds in this case since $(q, s_1 - q, s_2, \dotsc, s_n)$ and $(1, \mathfrak{t}_i)$ are $1$-admissible. 

We next assume that $(H_{k - 1})$ holds. We need to show that $(H_k)$ holds. It suffices to consider the array $\begin{pmatrix}
 \fve  \\ 
\fs  \end{pmatrix}$ where $\fs$ is not $k$-admissible. Moreover, from the induction hypothesis of $(H_{k - 1})$, we may assume that $\fs$ is $(k - 1)$-admissible. For such array $\begin{pmatrix}
 \fve  \\ 
\fs  \end{pmatrix}$, we claim that $\depth(\fs) \geq k$. Indeed, assume that $\depth(\fs) < k$. Since $\fs$ is $(k - 1)$-admissible, one verifies that $\fs$ is $k$-admissible, which is a contradiction.

Assume that $ \begin{pmatrix}
 \fve  \\ 
\fs  \end{pmatrix} = \begin{pmatrix}
 \varepsilon_1 & \cdots & \varepsilon_n  \\ 
s_1 & \cdots & s_n  \end{pmatrix}$ where $\fs$ is not $k$-admissible and $\depth(\fs) \geq k$. In order to prove that $(H_k)$ holds for the array $\begin{pmatrix}
 \fve  \\ 
\fs  \end{pmatrix}$, we proceed by induction on $s_1 + \cdots + s_k$. The case $s_1 + \cdots + s_k = 1$ is a simple check.  Assume that $(H_k)$ holds when $s_1 + \cdots + s_k < s$. We need to show that $(H_k)$ holds when $s_1 + \cdots + s_k = s$. To do so, we give the following algorithm:

\textit{Algorithm: } We begin with an array $\begin{pmatrix}
 \fve  \\ 
\fs  \end{pmatrix}$ where $\fs$ is not $k$-admissible, $\depth(\fs) \geq k$ and $s_1 + \cdots + s_k = s$.

\textit{Step 1:} From Proposition \ref{polydecom}, we can decompose $\Li \begin{pmatrix}
 \fve  \\
 \fs  \end{pmatrix} $ as follows:
    \begin{equation} \label{eq: polyalgpartlem}
    \Li \begin{pmatrix}
 \fve  \\
\fs  \end{pmatrix}
    = \underbrace{ - \Li \begin{pmatrix}
 \fve'  \\
\fs'  \end{pmatrix} }_\text{type 1} + \underbrace{\sum\limits_i b_i\Li \begin{pmatrix}
 \fe_i'  \\
\mathfrak{t}'_i  \end{pmatrix} }_\text{type 2} + \underbrace{\sum\limits_i c_i\Li \begin{pmatrix}
 \fm_i  \\
\mathfrak{u}_i  \end{pmatrix} }_\text{type 3}  ,
    \end{equation}
    where $ b_i, c_i \in A$. Note that the term of type $1$ disappears when $\Init(\fs) = \fs$ and $s_n = q$. Moreover, for all arrays $ \begin{pmatrix}
 \fe  \\
\mathfrak{t}  \end{pmatrix} $ appearing on the right hand side of \eqref{eq: polyalgpartlem}, we have $\depth(\mathfrak{t}) \geq \depth(\fs) \geq k$ and $t_1 + \cdots + t_k \leq s_1 + \cdots + s_k = s$.

\textit{Step 2:} For all arrays $ \begin{pmatrix}
 \fe  \\
\mathfrak{t}  \end{pmatrix} $ appearing on the right hand side of \eqref{eq: polyalgpartlem}, if $\mathfrak{t}$ is either $k$-admissible or $\mathfrak{t}$ satisfies the condition $t_1 + \cdots + t_k < s$, then we deduce from the induction hypothesis that $(H_k)$ holds for the array $\begin{pmatrix}
 \fve  \\
\fs  \end{pmatrix}$, and hence we stop the algorithm. Otherwise, there exists an array $\begin{pmatrix}
 \fve_1  \\
\fs_1  \end{pmatrix}$ where $\fs_1$ is not $k$-admissible, $\depth(\fs_1) \geq k$ and $s_{11} + \cdots + s_{1k} = s$. For such an array, we repeat the algorithm for $\begin{pmatrix}
 \fve_1  \\
\fs_1  \end{pmatrix}$. It should be remarked that the array $\begin{pmatrix}
 \fve_1  \\
\fs_1  \end{pmatrix}$ is of type $3$ with respect to $\begin{pmatrix}
 \fve  \\
\fs  \end{pmatrix}$. Indeed, if the array $\begin{pmatrix}
 \fve_1  \\
\fs_1  \end{pmatrix}$ is of type $1$ or type $2$, then we deduce from Proposition \ref{polydecom} that $s_{11} + \cdots + s_{1k} < s_{1} + \cdots + s_{k} = s$, which is a contradiction.

It remains to show that the above algorithm stops after a finite number of steps. Set $\begin{pmatrix}
 \fve_0 \\
\fs_0  \end{pmatrix} = \begin{pmatrix}
 \fve  \\
\fs  \end{pmatrix}$. Assume that the above algorithm does not stop. Then there exists a sequence of arrays $\begin{pmatrix}
 \fve_0 \\
\fs_0  \end{pmatrix}, \begin{pmatrix}
 \fve_1  \\
\fs_1  \end{pmatrix}, \begin{pmatrix}
 \fve_2  \\
\fs_2  \end{pmatrix}, \dotsc$ such that for all $i \geq 0$, $\fs_i$ is not $k$-admissible, $\depth(\fs_i) \geq k$ and $\begin{pmatrix}
 \fve_{i + 1}  \\
\fs_{i + 1}  \end{pmatrix}$ is of type $3$ with respect to $\begin{pmatrix}
 \fve_i  \\
\fs_i  \end{pmatrix}$. From Proposition \ref{polydecom}, we obtain an infinite sequence
\begin{align*}
    \Init(\fs_0) \prec \Init(\fs_1) \prec \Init(\fs_2) \prec \cdots
\end{align*}
which is strictly increasing. For all $i \geq 0$, since $\fs_i$ is not $k$-admissible and $\depth(\fs_i) \geq k$, we have $\depth(\Init(\fs_i)) \leq k$, hence $\Init(\fs_i) \preceq q^{\{k\}}$. This shows that $\Init(\fs_i) = \Init(\fs_{i + 1})$ for all $i$ sufficiently large, which is a contradiction. We deduce that the algorithm stops after a finite number of steps. We finish the proof. 
\end{proof}

\subsubsection{A set of generators $\mathcal{AT}_w$ for ACMPL's}

We recall that $\mathcal{AL}_w$ is the $K$-vector space generated by ACMPL's of weight $w$. We denote by $\mathcal{AT}_w$ the set of all ACMPL's $\Li \begin{pmatrix}
 \fve  \\
\fs  \end{pmatrix}  =  \Li \begin{pmatrix}
 \varepsilon_1 & \dots & \varepsilon_n \\
s_1 & \dots & s_n \end{pmatrix} $ of weight $w$ such that $s_1, \dots, s_{n-1} \leq q$ and $s_n < q$.

We put $t(w)=|\mathcal{AT}_w|$. Then one verifies that
\begin{equation*}
    t(w) = \begin{cases}
(q - 1) q^{w-1}& \text{if } 1 \leq w < q, \\
            (q - 1) (q^{w-1} - 1) &\text{if } w = q,
		 \end{cases}
\end{equation*}
and for $w>q$, $t(w)=(q-1)\sum \limits_{i = 1}^q t(w-i)$.

We are ready to state a weak version of Brown's theorem for ACMPL's.

\begin{proposition} \label{prop: weak Brown}
The set of all elements $\Li \begin{pmatrix}
 \fve  \\
\fs  \end{pmatrix} $ such that $\Li \begin{pmatrix}
 \fve  \\
\fs  \end{pmatrix}  \in \mathcal{AT}_w$ forms a set of generators for $\mathcal{AL}_w$.
\end{proposition}
\begin{proof}
The result follows immediately from Proposition \ref{polyalgpartlem} in the case of $k = w$.
\end{proof}

\subsection{A strong version of Brown's theorem for ACMPL's} \label{sec:strong Brown}

\subsubsection{Another set of generators $\mathcal{AS}_w$ for ACMPL's} \label{sec:AS_w}

We consider the set $\mathcal{J}_w$ consisting of positive tuples $\fs = (s_1, \dots, s_n)$ of weight $w$ such that $s_1, \dots, s_{n-1} \leq q$ and $s_n <q$, together with the set  $\mathcal{J}'_w$ consisting of positive tuples $\fs = (s_1, \dots, s_n)$ of weight $w$ such that $ q \nmid s_i$ for all $i$. Then there is a bijection
\begin{equation*}
    \iota \colon \mathcal{J}'_w \longrightarrow \mathcal{J}_w
\end{equation*}
given as follows: for each tuple $\fs = (s_1, \dots, s_n) \in \mathcal{J}'_w$, since $q \nmid s_i$, we can write $s_i = h_i q + r_i $ where $0 < r_i < q$ and $h_i \in \mathbb{Z}^{\ge0}$. The image $\iota(\mathfrak s)$ is the tuple
\begin{equation*}
    \iota(\mathfrak s) = (\underbrace{q, \dots, q}_{\text{$h_1$ times}}, r_1 , \dots, \underbrace{q, \dots, q}_{\text{$h_n$ times}}, r_n).
\end{equation*}
Let $\mathcal{AS}_w$ denote the set of ACMPL's $\Li \begin{pmatrix}
 \fve  \\
\fs  \end{pmatrix} $ such that $\mathfrak s \in \mathcal{J}'_w$. We note that in general, $\mathcal{AS}_w$ is strictly smaller than $\mathcal{AT}_w$. The only exceptions are when $q=2$ or $w \leq q$.

\subsubsection{Cardinality of $\mathcal{AS}_w$.}

We now compute $s(w)=|\mathcal{AS}_w|$. To do so we denote by $\mathcal{AJ}_w$ the set consisting of arrays $ \begin{pmatrix}
 \varepsilon_1 & \dots & \varepsilon_n \\
s_1 & \dots & s_n \end{pmatrix} $ of weight $w$ such that $q \nmid s_i$ for all $i$ and by $\mathcal{AJ}^1_w$ the set consisting of arrays $ \begin{pmatrix}
 \varepsilon_1 & \dots & \varepsilon_n \\
s_1 & \dots & s_n \end{pmatrix} $ of weight~$w$ such that $s_1, \dots, s_{n-1} \leq q, s_n < q$ and $\varepsilon_i = 1$ whenever $s_i = q$ for $1 \leq i \leq n$. We construct a map
\begin{equation*}
    \varphi \colon \mathcal{AJ}_w \longrightarrow \mathcal{AJ}^1_w
\end{equation*}
as follows: for each array $ \begin{pmatrix}
 \fve\\
\fs \end{pmatrix}  =  \begin{pmatrix}
 \varepsilon_1 & \dots & \varepsilon_n \\
s_1 & \dots & s_n \end{pmatrix}  \in \mathcal{AJ}_w$, since $q \nmid s_i$, we can write $s_i = (h_i-1) q + r_i $ where $0 < r_i < q$ and $h_i \in \mathbb N$. The image $\varphi  \begin{pmatrix}
 \fve\\
\fs \end{pmatrix} $ is the array
\begin{equation*}
    \varphi  \begin{pmatrix}
 \fve\\
\fs \end{pmatrix}  = \bigg( \underbrace{\begin{pmatrix}
     1 & \dots & 1\\
     q & \dots & q
     \end{pmatrix}}_\text{$h_1-1$ times}
     \begin{pmatrix}
     \varepsilon_1 \\
     r_1
     \end{pmatrix} \dots
   \underbrace{\begin{pmatrix}
     1 & \dots & 1\\
     q & \dots & q
     \end{pmatrix}}_\text{$h_n-1$ times}
      \begin{pmatrix}
     \varepsilon_n  \\
     r_n
     \end{pmatrix} \bigg) .
\end{equation*}
It is easily seen that $\varphi$ is a bijection, hence $|\mathcal{AS}_w| =|\mathcal{AJ}_w| = |\mathcal{AJ}^1_w|$.
Thus $s(w) = |\mathcal{AJ}^1_w|$. One verifies that
\begin{equation*}
    s(w) = \begin{cases}
(q - 1) q^{w-1}& \text{if } 1 \leq w < q, \\
            (q - 1) (q^{w-1} - 1) &\text{if } w = q,
		 \end{cases}
\end{equation*}
and for $w>q$,
	\[ s(w)=(q-1)\sum \limits_{i = 1}^{q-1} s(w-i) + s(w - q). \]

\subsubsection{} We state a strong version of Brown's theorem for ACMPL's.

\begin{theorem} \label{thm: strong Brown}
The set $\mathcal{AS}_w$ forms a set of generators for $\mathcal{AL}_w$. In particular,
	\[ \dim_K \mathcal{AL}_w \leq s(w). \]
\end{theorem}

\begin{proof}
We recall that $\mathcal{AT}_w$ is the set of all ACMPL's $\Li \begin{pmatrix}
 \fve  \\
\fs  \end{pmatrix}$ with $\begin{pmatrix}
 \fve  \\
\fs  \end{pmatrix} \in \mathcal{AJ}_w$

Let $\textup{Li} \begin{pmatrix}
 \fve  \\
\fs  \end{pmatrix}  =  \textup{Li}\begin{pmatrix}
 \varepsilon_1 & \dots & \varepsilon_n \\
s_1 & \dots & s_n \end{pmatrix} \in \mathcal{AT}_w$. Then $\begin{pmatrix}
 \fve  \\
\fs  \end{pmatrix} \in \mathcal{AJ}_w$, which implies $s_1, \dotsc, s_{n - 1} \leq q$ and $s_n < q$. We express
$\begin{pmatrix}
 \fve  \\
\fs  \end{pmatrix}$ in the following form
\begin{align*}
     \bigg( \underbrace{\begin{pmatrix}
     \varepsilon_1 & \dots & \varepsilon_{h_1 - 1}\\
     q & \dots & q
     \end{pmatrix}}_\text{$h_1 - 1$ times}
     \begin{pmatrix}
     \varepsilon_{h_1} \\
     r_1
     \end{pmatrix} \dots
   \underbrace{\begin{pmatrix}
     \varepsilon_{h_1 + \cdots + h_{\ell - 1} + 1} & \dots & \varepsilon_{h_1 + \cdots + h_{\ell - 1} + (h_\ell - 1)}\\
     q & \dots & q
     \end{pmatrix}}_\text{$h_\ell - 1$  times}
      \begin{pmatrix}
     \varepsilon_{h_1 + \cdots + h_{\ell - 1} + h_\ell}  \\
     r_\ell
     \end{pmatrix} \bigg),
\end{align*}
where $h_1, \dotsc, h_\ell \geq 1$, $h_1 + \cdots  + h_\ell = n$ and $0 <  r_1, \dotsc, r_\ell < q$. Then we set 
\begin{align*}
    \begin{pmatrix}
 \fve'  \\
\fs'  \end{pmatrix}  =  \begin{pmatrix}
 \varepsilon'_1 & \dots & \varepsilon'_\ell \\
s'_1 & \dots & s'_\ell \end{pmatrix},
\end{align*}
where $\varepsilon'_i = \varepsilon_{h_1 + \cdots + h_{i - 1} + 1} \cdots \varepsilon_{h_1 + \cdots + h_{i - 1} + h_i}$ and $s'_i = (h_i - 1)q + r_i$ for $1 \leq i \leq \ell$. We note that $\iota(\fs')=\fs$. 

From Proposition \ref{polydecom} and Proposition \ref{prop: weak Brown}, we can decompose $\textup{Li}\begin{pmatrix}
 \fve'  \\
\fs'  \end{pmatrix}$ as follows:
\begin{equation*}
    \Li \begin{pmatrix}
 \fve'  \\
\fs'  \end{pmatrix}  = \sum  a_{\fe,\mathfrak t}^{\fve',\fs'} \Li \begin{pmatrix}
 \fe  \\
\mathfrak{t}  \end{pmatrix} ,
\end{equation*}
where $ \begin{pmatrix}
 \fe  \\
\mathfrak{t}  \end{pmatrix} $ ranges over all elements of $\mathcal{AJ}_w$ and $a_{\fe,\mathfrak t}^{\fve',\fs'} \in A$ satisfying
\begin{equation*}
 a_{\fe,\mathfrak t}^{\fve',\fs'} \equiv \begin{cases}
			\pm 1  \ (\text{mod } D_1) & \text{if }  \begin{pmatrix}
 \fe  \\
\mathfrak{t}  \end{pmatrix}  =  \begin{pmatrix}
 \fve  \\
\fs  \end{pmatrix},\\
            0 \ \ (\text{mod } D_1) & \text{otherwise}.
		 \end{cases}
\end{equation*}
Note that $\Li \begin{pmatrix}
 \fve'  \\
\fs'  \end{pmatrix}  \in \mathcal{AS}_w$. Thus the transition matrix from the set consisting of such $\Li \begin{pmatrix}
 \fve'  \\
\fs'  \end{pmatrix} $ as above (we allow repeated elements) to the set consisting of $\Li \begin{pmatrix}
 \fve  \\
\fs  \end{pmatrix} $ with $ \begin{pmatrix}
 \fve  \\
\fs  \end{pmatrix}  \in \mathcal{AJ}_w$ is invertible. It then follows again from Proposition \ref{prop: weak Brown} that $\mathcal{AS}_w$ is a set of generators for $\mathcal{AL}_w$, as desired. \end{proof}


\section{Dual $t$-motives and linear independence} \label{sec: dual motives}

We continue with the notation given in the Introduction. Further, letting $t$ be another independent variable, we denote by $\bT$ the Tate algebra in the variable $t$ with coefficients in $\C_\infty$ equipped with the Gauss norm $\lVert . \rVert_\infty$, and by $\bL$ the fraction field of $\bT$.

We  denote by $\mathcal E$ the ring of series $\sum_{n \geq 0} a_nt^n \in \overline K[[t]]$ such that $\lim_{n \to +\infty} \sqrt[n]{|a_n|_\infty}=0$ and $[K_\infty(a_0,a_1,\ldots):K_\infty]<\infty$. Then any $f \in \mathcal E$ is an entire function. 

For $a \in A=\Fq[\theta]$, we set $a(t):=a \rvert_{\theta=t} \in \Fq[t]$.

\subsection{Dual $t$-motives} 

We recall the notion of dual $t$-motives due to Anderson  (see \cite[\S 4]{BP20} and \cite[\S 5]{HJ20} for more details). We refer the reader to \cite{And86} for the related notion of $t$-motives.

For $i \in \mathbb Z$ we consider the $i$-fold twisting of $\C_\infty((t))$ defined by
\begin{align*}
\C_\infty((t)) & \rightarrow \C_\infty((t)) \\
f=\sum_j a_j t^j & \mapsto f^{(i)}:=\sum_j a_j^{q^i} t^j.
\end{align*}
We extend $i$-fold twisting to matrices with entries in $\C_\infty((t))$ by twisting entry-wise.

Let $\overline K[t,\sigma]$ be the non-commutative $\overline K[t]$-algebra generated by
the new variable $\sigma$ subject to the relation $\sigma f=f^{(-1)}\sigma$ for all $f \in \overline K[t]$.

\begin{definition}
An effective dual $t$-motive is a $\overline K[t,\sigma]$-module $\mathcal M'$ which is free and finitely generated over $\overline K[t]$ such that for $\ell\gg 0$ we have
	\[(t-\theta)^\ell(\mathcal M'/\sigma \mathcal M') = \{0\}.\]
\end{definition}

We mention that effective dual $t$-motives are called Frobenius modules in \cite{CPY19,CH21,Har21,KL16}. Note that Hartl and Juschka \cite[\S 4]{HJ20} introduced a more general notion of dual $t$-motives. In particular, effective dual $t$-motives are always dual $t$-motives.

Throughout this paper we will always work with effective dual $t$-motives. Therefore, we will sometimes drop the word ``effective" where there is no confusion.

Let $\mathcal M$ and $\mathcal M'$ be two effective dual $t$-motives. Then a morphism of effective dual $t$-motives $\mathcal M \to \mathcal M'$ is just a homomorphism of left $\overline K[t,\sigma]$-modules. We denote by $\cF$ the category of effective dual $t$-motives equipped with the trivial object $\mathbf{1}$.

We say that an object $\mathcal M$ of $\cF$ is given by a matrix $\Phi \in \Mat_r(\overline K[t])$ if $\mathcal M$ is a $\overline K[t]$-module free of rank $r$ and the action of $\sigma$ is represented by the matrix $\Phi$ on a given  $\overline K[t]$-basis for $\mathcal M$. We say that an object $\mathcal M$ of $\cF$ is uniformizable or rigid analytically trivial if there exists a matrix $\Psi \in \text{GL}_r(\bT)$ satisfying $\Psi^{(-1)}=\Phi \Psi$. The matrix $\Psi$ is called a rigid analytic trivialization of $\mathcal M$.

We now recall the Anderson-Brownawell-Papanikolas criterion which is crucial in the sequel (see \cite[Theorem 3.1.1]{ABP04}).

\begin{theorem}[Anderson-Brownawell-Papanikolas] \label{thm:ABP}
Let $\Phi \in \Mat_\ell(\overline K[t])$ be a matrix such that $\det \Phi=c(t-\theta)^s$ for some $c \in \overline K^\times$ and $s \in \mathbb Z^{\geq 0}$. Let $\psi \in \Mat_{\ell \times 1}(\mathcal E)$ be a vector  satisfying $\psi^{(-1)}=\Phi \psi$ and $\rho \in \Mat_{1 \times \ell}(\overline K)$ such that $\rho \psi(\theta)=0$. Then there exists a vector $P \in \Mat_{1 \times \ell}(\overline K[t])$ such that
	\[ P \psi=0 \quad \text{and} \quad P(\theta)=\rho. \]
\end{theorem}

\subsection{Some constructions of dual $t$-motives} \label{subsec:dual motives}

\subsubsection{General case}

We briefly review some constructions of dual $t$-motives introduced in \cite{CPY19} (see also \cite{Cha14,CH21,Har21}).  Let $\fs=(s_1,\ldots,s_r) \in \mathbb N^r$ be a tuple of positive integers and $\fQ=(Q_1,\dots,Q_r) \in \overline K[t]^r$ satisfying the condition
\begin{equation} \label{eq: condition for Q}
(\lVert Q_1 \rVert_\infty / |\theta|_\infty^{\frac{qs_1}{q-1}})^{q^{i_1}} \ldots (\lVert Q_r \rVert_\infty / |\theta|_\infty^{\frac{qs_r}{q-1}})^{q^{i_r}} \to 0
\end{equation}
as $0  \leq i_r < \dots < i_1$ and $ i_1 \to \infty$.

We consider the dual $t$-motives $\mathcal M_{\fs,\fQ}$ and $\mathcal M_{\fs,\fQ}'$ attached to $(\fs,\fQ)$ given by the matrices
\begin{align*}
\Phi_{\fs,\fQ} &=
\begin{pmatrix}
(t-\theta)^{s_1+\dots+s_r} & 0 & 0 & \dots & 0 \\
Q_1^{(-1)} (t-\theta)^{s_1+\dots+s_r} & (t-\theta)^{s_2+\dots+s_r} & 0 & \dots & 0 \\
0 & Q_2^{(-1)} (t-\theta)^{s_2+\dots+s_r} & \ddots & & \vdots \\
\vdots & & \ddots & (t-\theta)^{s_r} & 0 \\
0 & \dots & 0 & Q_r^{(-1)} (t-\theta)^{s_r} & 1
\end{pmatrix} \\
&\in \Mat_{r+1}(\overline K[t]),
\end{align*}
and $\Phi'_{\fs,\fQ} \in \Mat_r(\overline K[t])$ is the upper left $r \times r$ sub-matrix of $\Phi_{\fs,\fQ}$.

Throughout this paper, we work with the Carlitz period $\widetilde \pi$ which is a fundamental period of the Carlitz module (see \cite{Gos96, Tha04}). We fix a choice of $(q-1)$st root of $(-\theta)$ and set
	\[ \Omega(t):=(-\theta)^{-q/(q-1)} \prod_{i \geq 1} \left(1-\frac{t}{\theta^{q^i}} \right) \in \bT^\times \]
so that
	\[ \Omega^{(-1)}=(t-\theta)\Omega \quad \text{ and } \quad \frac{1}{\Omega(\theta)}=\widetilde \pi. \]
Given $(\fs,\fQ)$ as above, Chang introduced the following series (see \cite[Lemma 5.3.1]{Cha14} and also \cite[Eq. (2.3.2)]{CPY19})
\begin{align} \label{eq: series L}
\frakL(\fs;\fQ)=\frakL(s_1,\dots,s_r;Q_1,\dots,Q_r):=\sum_{i_1 > \dots > i_r \geq 0} (\Omega^{s_r} Q_r)^{(i_r)} \dots (\Omega^{s_1} Q_1)^{(i_1)}.
\end{align}

It is proved that $\frakL(\fs,\fQ) \in \mathcal E$ (see \cite[Lemma 5.3.1]{Cha14}). Here we recall that $\mathcal E$ denotes the ring of series $\sum_{n \geq 0} a_nt^n \in \overline K[[t]]$ such that $\lim_{n \to +\infty} \sqrt[n]{|a_n|_\infty}=0$ and $[K_\infty(a_0,a_1,\ldots):K_\infty]<\infty$. In the sequel, we will use the following crucial property of this series (see \cite[Lemma 5.3.5]{Cha14} and \cite[Proposition 2.3.3]{CPY19}): for all $j \in \mathbb Z^{\geq 0}$, we have
\begin{equation} \label{eq: power seriesChang}
\frakL(\fs;\fQ) \left(\theta^{q^j} \right)=\left(\frakL(\fs;\fQ)(\theta)\right)^{q^j}.
\end{equation}

Then the matrix given by	
\begin{align*}
\Psi_{\fs,\fQ} &=
\begin{pmatrix}
\Omega^{s_1+\dots+s_r} & 0 & 0 & \dots & 0 \\
\frakL(s_1;Q_1) \Omega^{s_2+\dots+s_r} & \Omega^{s_2+\dots+s_r} & 0 & \dots & 0 \\
\vdots & \frakL(s_2;Q_2) \Omega^{s_3+\dots+s_r} & \ddots & & \vdots \\
\vdots & & \ddots & \ddots & \vdots \\
\frakL(s_1,\dots,s_{r-1};Q_1,\dots,Q_{r-1}) \Omega^{s_r}  & \frakL(s_2,\dots,s_{r-1};Q_2,\dots,Q_{r-1}) \Omega^{s_r} & \dots & \Omega^{s_r}& 0 \\
\frakL(s_1,\dots,s_r;Q_1,\dots,Q_r)  & \frakL(s_2,\dots,s_r;Q_2,\dots,Q_r) & \dots & \frakL(s_r;Q_r) & 1
\end{pmatrix} \\
&\in \text{GL}_{r+1}(\bT)
\end{align*}
satisfies
	\[ \Psi_{\fs,\fQ}^{(-1)}=\Phi_{\fs,\fQ} \Psi_{\fs,\fQ}. \]
Thus $\Psi_{\fs,\fQ}$ is a rigid analytic trivialization associated to the dual $t$-motive $\mathcal M_{\fs,\fQ}$.

We also denote by $\Psi_{\fs,\fQ}'$ the upper $r \times r$ sub-matrix of $\Psi_{\fs,\fQ}$. It is clear that $\Psi_\fs'$ is a rigid analytic trivialization associated to the dual $t$-motive $\mathcal M_{\fs,\fQ}'$.

Further, combined with Eq. \eqref{eq: power seriesChang}, the above construction of dual $t$-motives implies that $\widetilde \pi^w \frak L(\fs;\fQ)(\theta)$ where $w=s_1+\dots+s_r$ has the MZ (multizeta) property in the sense of \cite[Definition 3.4.1]{Cha14}. By \cite[Proposition 4.3.1]{Cha14}, we get

\begin{proposition} \label{prop: MZ property}
Let $(\fs_i;\fQ_i)$ as before for $1 \leq i \leq m$. We suppose that all the tuples of positive integers $\fs_i$ have the same weight, say $w$. Then the following assertions are equivalent:
\begin{itemize}
\item[i)] $\frak L(\fs_1;\fQ_1)(\theta),\dots,\frak L(\fs_m;\fQ_m)(\theta)$ are $K$-linearly independent.

\item[ii)] $\frak L(\fs_1;\fQ_1)(\theta),\dots,\frak L(\fs_m;\fQ_m)(\theta)$ are $\overline K$-linearly independent.
\end{itemize}
\end{proposition}

We end this section by mentioning that Chang \cite{Cha14} also proved analogue of Goncharov's conjecture in this setting.

\subsubsection{Dual $t$-motives connected to MZV's and AMZV's} \label{sec:MZV motives} 

Following Anderson and Thakur \cite{AT09} we introduce dual $t$-motives connected to MZV's and AMZV's. We briefly review Anderson-Thakur polynomials introduced in \cite{AT90}. For $k \geq 0$, we set $[k]:=\theta^{q^k}-\theta$ and $D_k:= \prod^k_{\ell=1} [\ell]^{q^{k-\ell}}$. For $n \in \N$ we write $n-1 = \sum_{j \geq 0} n_j q^j$ with $0 \leq n_j \leq q-1$ and define
	\[ \Gamma_n:=\prod_{j \geq 0} D_j^{n_j}. \]

We set $\gamma_0(t) :=1$ and $\gamma_j(t) :=\prod^j_{\ell=1} (\theta^{q^j}-t^{q^\ell})$ for $j\geq 1$. Then Anderson-Thakur polynomials $\alpha_n(t) \in A[t]$ are given by the generating series
	\[ \sum_{n \geq 1} \frac{\alpha_n(t)}{\Gamma_n} x^n:=x\left( 1-\sum_{j \geq 0} \frac{\gamma_j(t)}{D_j} x^{q^j} \right)^{-1}. \]
Finally, we define $H_n(t)$ by switching $\theta$ and $t$
\begin{align*}
H_n(t)=\alpha_n(t) \big|_{t=\theta, \, \theta=t}.
\end{align*}
By \cite[Eq. (3.7.3)]{AT90} we get
\begin{equation} \label{eq: deg theta Hn}
\deg_\theta H_n \leq \frac{(n-1) q}{q-1}<\frac{nq}{q-1}.
\end{equation}

Let $\fs=(s_1,\ldots,s_r) \in \mathbb N^r$ be a tuple and $\mathfrak \epsilon=(\epsilon_1,\ldots,\epsilon_r) \in (\F_q^\times)^r$. Recall that $\overline{\mathbb F}_q$ denotes the algebraic closure of $\mathbb F_q$ in $\overline{K}$. For all $1 \leq i \leq r$ we fix a fixed $(q-1)$-th root $\gamma_i \in \overline{\mathbb F}_q$ of $\epsilon_i \in \mathbb F_q^\times$ and set $Q_{s_i,\epsilon_i}:=\gamma_i H_{s_i}$. Then we set $\fQ_{\fs,\fe}:=(Q_{s_1,\epsilon_1},\dots,Q_{s_r,\epsilon_r})$ and put $\frak L(\fs;\fe):=\frak L(\fs;\fQ_{\fs,\fe})$. By \eqref{eq: deg theta Hn} we know that $\lVert H_n \rVert_\infty < |\theta|_\infty ^{\tfrac{n q}{q-1}}$ for all $n \in \N$, thus $\fQ_{\fs,\fe}$ satisfies Condition \eqref{eq: condition for Q}. Thus we can define the dual $t$-motives $\mathcal M_{\fs,\fe}=\mathcal M_{\fs,\fQ_{\fs,\fe}}$ and $\mathcal M_{\fs,\fe}'=\mathcal M_{\fs,\fQ_{\fs,\fe}}'$ attached to $\fs$ whose matrices and rigid analytic trivializations will be denoted by $(\Phi_{\fs,\fe},\Psi_{\fs,\fe})$ and $(\Phi_{\fs,\fe}',\Psi_{\fs,\fe}')$, respectively. These dual $t$-motives are connected to MZV's and AMZV's by the following result (see \cite[Proposition 2.12]{CH21} for more details):
\begin{equation} \label{eq: MZV}
\frak L(\fs;\fe)(\theta)=\frac{\gamma_1 \dots \gamma_r \Gamma_{s_1} \dots \Gamma_{s_r} \zeta_A \begin{pmatrix}
\fe \\ \fs
\end{pmatrix}}{\widetilde \pi^{w(\fs)}}.
\end{equation}
By a result of Thakur \cite{Tha09}, one can show (see \cite[Theorem 2.1]{Har21}) that $\zeta_A \begin{pmatrix}
\fe \\ \fs
\end{pmatrix} \neq 0$. Thus $\frak L(\fs;\fe)(\theta) \neq 0$.

\subsubsection{Dual $t$-motives connected to CMPL's and ACMPL's} \label{sec:CMPL motives} 

We keep the notation as above. Let $\fs=(s_1,\ldots,s_r) \in \mathbb N^r$ be a tuple and $\fe=(\epsilon_1,\ldots,\epsilon_r) \in (\F_q^\times)^r$.  For all $1 \leq i \leq r$ we have a fixed $(q-1)$-th root $\gamma_i$ of $\epsilon_i \in \mathbb F_q^\times$ and set $Q_{s_i,\epsilon_i}':=\gamma_i$. Then we set $\fQ_{\fs,\fe}':=(Q_{s_1,\epsilon_1}',\dots,Q_{s_r,\epsilon_r}')$ and put
\begin{align} \label{eq: series Li}
\frakLi(\fs;\fe)=\frak L(\fs;\fQ_{\fs,\fe}')=\sum_{i_1 > \dots > i_r \geq 0} (\gamma_{i_r} \Omega^{s_r})^{(i_r)} \dots (\gamma_{i_1} \Omega^{s_1})^{(i_1)}.
\end{align}

Thus we can define the dual $t$-motives $\mathcal N_{\fs,\fe}=\mathcal N_{\fs,\fQ_{\fs,\fe}'}$ and $\mathcal N_{\fs,\fe}'=\mathcal N_{\fs,\fQ_{\fs,\fe}'}'$ attached to $(\fs,\fe)$. These dual $t$-motives are connected to CMPL's and ACMPL's by the following result (see \cite[Lemma 5.3.5]{Cha14} and \cite[Prop. 2.3.3]{CPY19}):
\begin{equation} \label{eq: ACMPL}
\frakLi(\fs;\fe)(\theta)=\frac{\gamma_1 \dots \gamma_r \Li \begin{pmatrix}
\fe \\ \fs
\end{pmatrix}}{\widetilde \pi^{w(\fs)}}.
\end{equation}

\subsection{A result for linear independence}

\subsubsection{Setup} 

Let $w \in \N$ be a positive integer. Let $\{(\fs_i;\fQ_i)\}_{1 \leq i \leq n}$ be a collection of pairs satisfying Condition \eqref{eq: condition for Q} such that $\fs_i$ always has weight $w$. We write $\fs_i=(s_{i1},\dots,s_{i \ell_i}) \in \mathbb N^{\ell_i}$ and $\fQ_i=(Q_{i1},\dots,Q_{i\ell_i}) \in (\Fq^\times)^{\ell_i}$ so that $s_{i1}+\dots+s_{i \ell_i}=w$. We introduce the set of tuples
	\[ I(\fs_i;\fQ_i):=\{\emptyset,(s_{i1};Q_{i1}),\dots,(s_{i1},\dots,s_{i (\ell_i-1)};Q_{i1},\dots,Q_{i(\ell_i-1)})\},\]
and set
	\[ I:=\cup_i I(\fs_i;\fQ_i). \]

\subsubsection{Linear independence} 

We are now ready to state the main result of this section.

\begin{theorem} \label{theorem: linear independence}
We keep the above notation. We suppose further that $\{(\fs_i;\fQ_i)\}_{1 \leq i \leq n}$ satisfies the following conditions:
\begin{itemize}
\item[(LW)] For any weight $w'<w$, the values $\frak L(\frak t;\fQ)(\theta)$ with $(\frak t;\fQ) \in I$ and $w(\frak t)=w'$ are all $K$-linearly independent. In particular, $\frak L(\frak t;\fQ)(\theta)$ is always nonzero.

\item[(LD)] There exist $a \in A$ and $a_i \in A$ for $1 \leq i \leq n$ which are not all zero such that
\begin{equation*}
a+\sum_{i=1}^n a_i \frak L(\fs_i;\fQ_i)(\theta)=0.
\end{equation*}
\end{itemize}

For all $(\frak t;\fQ) \in I$, we set the following series in $t$
\begin{equation} \label{eq:f}
f_{\mathfrak t;\fQ}:= \sum_i a_i(t) \mathfrak L(s_{i(k+1)},\dots,s_{i \ell_i};Q_{i(k+1)},\dots,Q_{i \ell_i}),
\end{equation}
where the sum runs through the set of indices $i$ such that $(\mathfrak t;\fQ)=(s_{i1},\dots,s_{i k};Q_{i1},\dots,Q_{i k})$ for some $0 \leq k \leq \ell_i-1$. 

Then for all $(\mathfrak t;\fQ) \in I$, $f_{\mathfrak t;\fQ}(\theta)$ belongs to $K$.
\end{theorem}

\begin{remark}
1) Here we note that  LW stands for Lower Weights and LD for Linear Dependence.

2) With the above notation we have
\[ f_{\emptyset}= \sum_i a_i(t) \mathfrak L(\fs_i;\fQ_i). \]

2) In fact, we improve \cite[Theorem B]{ND21} in two directions. First, we remove the restriction to Anderson-Thakur polynomials and tuples $\fs_i$. Second, and more importantly, we allow an additional term $a$, which is crucial in the sequel. More precisely, in the case of MZV's, while \cite[Theorem B]{ND21} investigates linear relations between MZV's of weight $w$, Theorem \ref{theorem: linear independence}  investigates linear relations between MZV's of weight $w$ and suitable powers $\widetilde \pi^w$ of the Carlitz period.
\end{remark}

\begin{proof}
The proof will be divided into two steps.

\medskip
\noindent {\bf Step 1.} We first construct a dual $t$-motive to which we will apply the Anderson-Brownawell-Papanikolas criterion. We recall $a_i(t):=a_i \rvert_{\theta=t} \in \Fq[t]$.

For each pair $(\fs_i;\fQ_i)$ we have attached to it a matrix $\Phi_{\fs_i,\fQ_i}$. For $\fs_i=(s_{i1},\dots,s_{i \ell_i}) \in \mathbb N^{\ell_i}$ and $\fQ_i=(Q_{i1},\dots,Q_{i\ell_i}) \in (\Fq^\times)^{\ell_i}$ we recall
	\[ I(\fs_i;\fQ_i)=\{\emptyset,(s_{i1};Q_{i1}),\dots,(s_{i1},\dots,s_{i (\ell_i-1)};Q_{i1},\dots,Q_{(\ell_i-1)})\}, \]
and $I:=\cup_i I(\fs_i;\fQ_i)$.

We now construct a new matrix $\Phi'$ indexed by elements of $I$, say 
\[
\Phi'=\left(\Phi'_{(\mathfrak t;\fQ),(\mathfrak t';\fQ')}\right)_{(\mathfrak t;\fQ),(\mathfrak t';\fQ') \in I} \in \Mat_{|I|}(\overline K[t]). 
\]
For the row which corresponds to the empty pair $\emptyset$ we put
\begin{align*}
\Phi'_{\emptyset,(\mathfrak t';\fQ')}=
\begin{cases}
(t-\theta)^w & \text{if } (\mathfrak t';\fQ')=\emptyset, \\
0 & \text{otherwise}.
\end{cases}
\end{align*}
For the row indexed by $(\mathfrak t;\fQ)=(s_{i1},\dots,s_{i j};Q_{i1},\dots,Q_{ij})$ for some $i$ and $1 \leq j \leq \ell_i-1$ we put
\begin{align*}
\Phi'_{(\mathfrak t;\fQ),(\mathfrak t';\fQ')}=
\begin{cases}
(t-\theta)^{w-w(\mathfrak t')} & \text{if } (\mathfrak t';\fQ')=(\mathfrak t;\fQ), \\
Q_{ij}^{(-1)} (t-\theta)^{w-w(\mathfrak t')} & \text{if } (\mathfrak t';\fQ')=(s_{i1},\dots,s_{i (j-1)};Q_{i1},\dots,Q_{i (j-1)}), \\
0 & \text{otherwise}.
\end{cases}
\end{align*}
Note that $\Phi_{\fs_i,\fQ_i}'=\left(\Phi'_{(\mathfrak t;\fQ),(\mathfrak t';\fQ')}\right)_{(\mathfrak t;\fQ),(\mathfrak t';\fQ') \in I(\fs_i;\fQ_i)}$ for all $i$.

We define $\Phi \in \Mat_{|I|+1}(\overline K[t])$ by
\begin{align*}
\Phi=\begin{pmatrix}
\Phi' & 0  \\
\bv & 1
\end{pmatrix} \in \Mat_{|I|+1}(\overline K[t]), \quad \bv=(v_{\mathfrak t,\fQ})_{(\mathfrak t;\fQ) \in I} \in \Mat_{1 \times|I|}(\overline K[t]),
\end{align*}
where
\[
v_{\mathfrak t;\fQ}=\sum_i a_i(t) Q_{i \ell_i}^{(-1)} (t-\theta)^{w-w(\mathfrak t)} .
\]
Here the sum runs through the set of indices $i$ such that $(\mathfrak t;\fQ)=(s_{i1},\dots,s_{i (\ell_i-1)};Q_{i1},\dots,Q_{i (\ell_i-1)})$ and the empty sum is defined to be zero.

We now introduce a rigid analytic trivialization matrix $\Psi$ for $\Phi$. We define $\Psi'=\left(\Psi'_{(\mathfrak t;\fQ),(\mathfrak t';\fQ')}\right)_{(\mathfrak t;\fQ),(\mathfrak t';\fQ') \in I} \in  \text{GL}_{|I|}(\bT)$ as follows. For the row which corresponds to the empty pair $\emptyset$ we define
\begin{align*}
\Psi'_{\emptyset,(\mathfrak t';\fQ')}=
\begin{cases}
\Omega^w & \text{if } (\mathfrak t';\fQ')=\emptyset, \\
0 & \text{otherwise}.
\end{cases}
\end{align*}
For the row indexed by $(\mathfrak t;\fQ)=(s_{i1},\dots,s_{i j};Q_{i1},\dots,Q_{i j})$ for some $i$ and $1 \leq j \leq \ell_i-1$ we put
\begin{align*}
&\Psi'_{(\mathfrak t;\fQ),(\mathfrak t';\fQ')}= \\
&
\begin{cases}
\mathfrak L(\mathfrak t;\fQ) \Omega^{w-w(\mathfrak t)} & \text{if } (\mathfrak t';\fQ')=\emptyset, \\
\mathfrak L(s_{i(k+1)},\dots,s_{ij};Q_{i(k+1)},\dots,Q_{ij}) \Omega^{w-w(\mathfrak t)} & \text{if } (\mathfrak t';\fQ')=(s_{i1},\dots,s_{i k};Q_{i1},\dots,Q_{i k}) \text{ for some } 1 \leq k \leq j, \\
0 & \text{otherwise}.
\end{cases}
\end{align*}
Note that $\Psi_{\fs_i,\fQ_i}'=\left(\Psi'_{(\mathfrak t;\fQ),(\mathfrak t';\fQ')}\right)_{(\mathfrak t;\fQ),(\mathfrak t';\fQ') \in I(\fs_i;\fQ_i)}$ for all $i$.

We define $\Psi \in \text{GL}_{|I|+1}(\bT)$ by
\begin{align*}
\Psi=\begin{pmatrix}
\Psi' & 0  \\
\bff & 1
\end{pmatrix} \in \text{GL}_{|I|+1}(\bT), \quad \bff=(f_{\mathfrak t;\fQ})_{\mathfrak t \in I} \in \Mat_{1 \times|I|}(\bT).
\end{align*}
Here we recall (see Eq. \eqref{eq:f})
\begin{equation*}
f_{\mathfrak t;\fQ}= \sum_i a_i(t) \mathfrak L(s_{i(k+1)},\dots,s_{i \ell_i};Q_{i(k+1)},\dots,Q_{i \ell_i})
\end{equation*}
where the sum runs through the set of indices $i$ such that $(\mathfrak t;\fQ)=(s_{i1},\dots,s_{i k};Q_{i1},\dots,Q_{i k})$ for some $0 \leq k \leq \ell_i-1$. In particular, $f_{\emptyset}= \sum_i a_i(t) \mathfrak L(\fs_i;\fQ_i)$.

By construction and by \S \ref{subsec:dual motives}, we get $\Psi^{(-1)}=\Phi \Psi$, that means $\Psi$ is a rigid analytic trivialization for $\Phi$.

\medskip
\noindent {\bf Step 2.} Next we apply the Anderson-Brownawell-Papanikolas criterion (see Theorem \ref{thm:ABP}) to prove Theorem \ref{theorem: linear independence}.

In fact, we define
\begin{align*}
\widetilde \Phi=\begin{pmatrix}
1 & 0  \\
0 & \Phi
\end{pmatrix} \in  \Mat_{|I|+2}(\overline K[t])
\end{align*}
and consider the vector constructed from the first column vector of $\Psi$
\begin{align*}
\widetilde \psi=\begin{pmatrix}
1 \\
\Psi_{(\mathfrak t;\fQ),\emptyset}' \\
f_\emptyset
\end{pmatrix}_{(\mathfrak t;\fQ) \in I}.
\end{align*}
Then we have $\widetilde \psi\twistinv=\widetilde \Phi \widetilde \psi$.

We also observe that for all $(\mathfrak t;\fQ) \in I$ we have $\Psi_{(\mathfrak t;\fQ),\emptyset}'=\mathfrak L(\mathfrak t;\fQ) \Omega^{w-w(\mathfrak t)}$. Further,
\begin{align*}
a+f_\emptyset(\theta)=a+\sum_i a_i \mathfrak L(\fs_i;\fQ_i)(\theta)=0.
\end{align*}
By Theorem \ref{thm:ABP} with $\rho=(a,0,\dots,0,1)$ we deduce that there exists $\bh=(g_0,g_{\mathfrak t,\fQ},g) \in \Mat_{1 \times (|I|+2)}(\overline K[t])$ such that $\bh \psi=0$, and that
$g_{\mathfrak t,\fQ}(\theta)=0$ for $(\mathfrak t,\fQ) \in I$, $g_0(\theta)=a$ and $g(\theta)=1 \neq 0$. If we put $ \bg:=(1/g)\bh \in \Mat_{1 \times (|I|+2)}(\overline K(t))$, then all the entries of $ \bg$ are regular at $t=\theta$.

Now we have
\begin{align} \label{eq: reduction}
( \bg- \bg\invtwist \widetilde \Phi) \widetilde \psi= \bg \widetilde \psi-( \bg  \widetilde \psi)\invtwist=0.
\end{align}
We write $ \bg- \bg\invtwist \widetilde \Phi=(B_0,B_{\mathfrak t},0)_{\mathfrak t \in I}$. We claim that $B_0=0$ and $B_{\mathfrak t,\fQ}=0$ for all $(\mathfrak t;\fQ) \in I$. In fact, expanding \eqref{eq: reduction} we obtain
\begin{equation} \label{eq: B}
B_0+\sum_{\mathfrak t \in I} B_{\mathfrak t,\fQ} \mathfrak L(\mathfrak t;\fQ) \Omega^{w-w(\mathfrak t)}=0.
\end{equation}

By \eqref{eq: power seriesChang} we see that for $(\mathfrak t;\fQ) \in I$ and $j \in \N$,
\begin{equation*}
\frakL(\mathfrak t;\fQ)(\theta^{q^j})=(\frakL(\mathfrak t;\fQ)(\theta))^{q^j}
\end{equation*}
which is nonzero by Condition $(LW)$.

First, as the function $\Omega$ has a simple zero at $t=\theta^{q^k}$ for $k \in \N$, specializing \eqref{eq: B} at $t=\theta^{q^j}$ yields $B_0(\theta^{q^j})=0$ for $j \geq 1$. Since $B_0$ belongs to $\overline K(t)$, it follows that $B_0=0$.

Next, we put $w_0:=\max_{(\mathfrak t;\fQ) \in I} w(\mathfrak t)$ and denote by $I(w_0)$ the set of $(\mathfrak t;\fQ) \in I$ such that $w(\mathfrak t)=w_0$. Then dividing \eqref{eq: B} by $\Omega^{w-w_0}$ yields
\begin{equation} \label{eq: B1}
\sum_{(\mathfrak t;\fQ) \in I} B_{\mathfrak t,\fQ} \mathfrak L(\mathfrak t;\fQ) \Omega^{w_0-w(\mathfrak t)}=\sum_{(\mathfrak t;\fQ) \in I(w_0)} B_{\mathfrak t,\fQ} \mathfrak L(\mathfrak t;\fQ)+\sum_{(\mathfrak t;\fQ) \in I \setminus I(w_0)} B_{\mathfrak t,\fQ} \mathfrak L(\mathfrak t;\fQ) \Omega^{w_0-w(\mathfrak t)}=0.
\end{equation}
Since each $B_{\mathfrak t,\fQ}$ belongs to $\overline K(t)$, they are defined at $t=\theta^{q^j}$ for $j \gg 1$. Note that the function $\Omega$ has a simple zero at $t=\theta^{q^k}$ for $k \in \N$. Specializing \eqref{eq: B1} at $t=\theta^{q^j}$ and using \eqref{eq: power seriesChang} yields
	\[ \sum_{(\mathfrak t;\fQ) \in I(w_0)} B_{\mathfrak t,\fQ}(\theta^{q^j}) (\mathfrak L(\mathfrak t;\fQ)(\theta))^{q^j}=0 \]
for $j \gg 1$.

We claim that $B_{\mathfrak t,\fQ}(\theta^{q^j})=0$ for $j \gg 1$ and for all $(\mathfrak t;\fQ) \in I(w_0)$. Otherwise, we get a non-trivial $\overline K$-linear relation between $\mathfrak L(\mathfrak t;\fQ)(\theta)$ with $(\frak t;\fQ) \in I$ of weight $w_0$. By Proposition \ref{prop: MZ property} we deduce a non-trivial $K$-linear relation between $\mathfrak L(\mathfrak t;\fQ)(\theta)$ with $(\frak t;\fQ) \in I(w_0)$, which contradicts with Condition $(LW)$.
Now we know that $B_{\mathfrak t,\fQ}(\theta^{q^j})=0$ for $j \gg 1$ and for all $(\mathfrak t;\fQ) \in I(w_0)$. Since each $B_{\mathfrak t,\fQ}$ belongs to $\overline K(t)$, it follows that
$B_{\mathfrak t,\fQ}=0$ for all $(\mathfrak t;\fQ) \in I(w_0)$.

Next, we put $w_1:=\max_{(\mathfrak t;\fQ) \in I \setminus I(w_0)} w(\mathfrak t)$ and denote by $I(w_1)$ the set of $(\mathfrak t;\fQ) \in I$ such that $w(\mathfrak t)=w_1$. Dividing \eqref{eq: B} by $\Omega^{w-w_1}$ and specializing at $t=\theta^{q^j}$ yields
	\[ \sum_{(\mathfrak t;\fQ) \in I(w_1)} B_{\mathfrak t,\fQ}(\theta^{q^j}) (\mathfrak L(\mathfrak t;\fQ)(\theta))^{q^j}=0 \]
for $j \gg 1$. Since $w_1<w$, by Proposition \ref{prop: MZ property} and Condition $(LW)$ again we deduce that $B_{\mathfrak t,\fQ}(\theta^{q^j})=0$ for $j \gg 1$ and for all $(\mathfrak t;\fQ) \in I(w_1)$. Since each $B_{\mathfrak t,\fQ}$ belongs to $\overline K(t)$, it follows that
$B_{\mathfrak t,\fQ}=0$ for all $(\mathfrak t;\fQ) \in I(w_1)$. Repeating the previous arguments we deduce that $B_{\mathfrak t,\fQ}=0$ for all $(\mathfrak t;\fQ) \in I$ as required.

We have proved that $ \bg- \bg\invtwist \widetilde \Phi=0$. Thus
\begin{align*}
\begin{pmatrix}
1 & 0 & 0 \\
0 & \text{Id} & 0 \\
g_0/g & (g_{\mathfrak t,\fQ}/g)_{(\mathfrak t;\fQ) \in I}  & 1
\end{pmatrix}^{(-1)}
\begin{pmatrix}
1 & 0  \\
0 & \Phi
\end{pmatrix}
= \begin{pmatrix}
1 & 0 & 0 \\
0 & \Phi' & 0 \\
0 & 0 & 1
\end{pmatrix}
\begin{pmatrix}
1 & 0 & 0 \\
0 & \text{Id} & 0 \\
g_0/g & (g_{\mathfrak t,\fQ}/g)_{(\mathfrak t;\fQ) \in I}  & 1
\end{pmatrix}.
\end{align*}
By \cite[Prop. 2.2.1]{CPY19} we see that the common denominator $b$ of $g_0/g$ and $g_{\mathfrak t,\fQ}/g$ for $(\mathfrak t,\fQ) \in I$ belongs to $\Fq[t] \setminus \{0\}$. If we put $\delta_0=bg_0/g$ and $\delta_{\mathfrak t,\fQ}=b g_{\mathfrak t,\fQ}/g$ for $(\mathfrak t,\fQ) \in I$ which belong to $\overline K[t]$ and $\delta:=(\delta_{\mathfrak t,\fQ})_{\mathfrak t \in I} \in \Mat_{1 \times |I|}(\overline K[t])$, then $\delta_0^{(-1)}=\delta_0$ and
\begin{align} \label{eq:equation for delta}
\begin{pmatrix}
\text{Id} & 0 \\
\delta & 1
\end{pmatrix}^{(-1)}
\begin{pmatrix}
\Phi' & 0  \\
b \bv & 1
\end{pmatrix}
= \begin{pmatrix}
 \Phi' & 0 \\
0 & 1
\end{pmatrix}
\begin{pmatrix}
\text{Id} & 0 \\
\delta & 1
\end{pmatrix}.
\end{align}

If we put $X:=\begin{pmatrix}
\text{Id} & 0 \\
\delta & 1
\end{pmatrix}
\begin{pmatrix}
\Psi' & 0  \\
b \bff & 1
\end{pmatrix}$, then $X^{(-1)}=\begin{pmatrix}
 \Phi' & 0 \\
0 & 1
\end{pmatrix} X$. By \cite[\S 4.1.6]{Pap08} there exist $\nu_{\mathfrak t,\fQ} \in \Fq(t)$ for $(\mathfrak t,\fQ) \in I$ such that if we set $\nu=(\nu_{\mathfrak t,\fQ})_{(\mathfrak t,\fQ) \in I} \in \Mat_{1 \times |I|}(\Fq(t))$,
\begin{align*}
X=\begin{pmatrix}
\Psi' & 0 \\
0 & 1
\end{pmatrix}
\begin{pmatrix}
\text{Id} & 0 \\
\nu & 1
\end{pmatrix}.
\end{align*}
Thus the equation $\begin{pmatrix}
\text{Id} & 0 \\
\delta & 1
\end{pmatrix}
\begin{pmatrix}
\Psi' & 0  \\
b \bff & 1
\end{pmatrix}=\begin{pmatrix}
\Psi' & 0 \\
0 & 1
\end{pmatrix}
\begin{pmatrix}
\text{Id} & 0 \\
\nu & 1
\end{pmatrix}$ implies
\begin{equation} \label{eq:nu}
\delta \Psi'+b \bff=\nu.
\end{equation}
The left-hand side belongs to $\bT$, so does the right-hand side. Thus $\nu=(\nu_{\mathfrak t,\fQ})_{(\mathfrak t,\fQ) \in I} \in \Mat_{1 \times |I|}(\Fq[t])$. For any $j \in \N$, by specializing \eqref{eq:nu} at $t=\theta^{q^j}$ and using \eqref{eq: power seriesChang} and the fact that $\Omega$ has a simple zero at $t=\theta^{q^j}$ we deduce that
	\[ \bff(\theta)=\nu(\theta)/b(\theta). \]
Thus for all $(\mathfrak t,\fQ) \in I$, $f_{\mathfrak t;\fQ}(\theta)$ given as in \eqref{eq:f} belongs to $K$. 
\end{proof}


\section{Linear relations between ACMPL's} \label{sec:transcendental part}

In this section we use freely the notation of \S \ref{sec:algebraic part} and  \S \ref{sec:CMPL motives}.

\subsection{Preliminaries} 

We begin this section by proving several auxiliary lemmas which will be useful in the sequel. We recall that $\overline{\mathbb F}_q$ denotes the algebraic closure of $\mathbb F_q$ in $\overline{K}$.

\begin{lemma} \label{lem: gamma_i}
Let $\epsilon_i \in \Fq^\times$ be different elements. We denote by $\gamma_i \in \overline{\F}_q$ a $(q-1)$-th root of $\epsilon_i$. Then $\gamma_i$ are all $\Fq$-linearly independent.
\end{lemma}

\begin{proof}
We know that $\Fq^\times$ is cyclic as a multiplicative group. Let $\epsilon$ be a generating element of  $\Fq^\times$ so that $\Fq^\times=\langle \epsilon \rangle$. Let $\gamma$ be the associated $(q-1)$-th root of $\epsilon$. Then for all $1 \leq i \leq q-1$ it follows that $\gamma^i$ is a $(q-1)$-th root of $\epsilon^i$. Thus it suffices to show that the polynomial $P(X)=X^{q-1}-\epsilon$ is irreducible in $\Fq[X]$. Suppose that this is not the case, write $P(X)=P_1(X)P_2(X)$ with $1 \leq \deg P_1<q-1$. Since the roots of $P(X)$ are of the form $\alpha \gamma$ with $\alpha \in \Fq^\times$, those of $P_1(X)$ are also of this form. Looking at the constant term of $P_1(X)$, we deduce that $\gamma^{\deg P_1} \in \Fq^\times$. If we put $m=\text{gcd}(\deg P_1,q-1)$, then $1 \leq m<q-1$ and $\gamma^m \in \Fq^\times$. Letting $\beta:=\gamma^m \in \Fq^\times$, we get $\beta^{\frac{q-1}{m}}=\gamma^{q-1}=\epsilon$. Since $1 \leq m<q-1$, we get a contradiction with the fact that $\Fq^\times=\langle \epsilon \rangle$. The proof is finished.
\end{proof}

\begin{lemma} \label{lem: same character}
Let $ \Li \begin{pmatrix}
 \fe_i  \\
\fs_i \end{pmatrix} \in \mathcal{AL}_w$ and $a_i \in K$ satisfying
\begin{equation*}
\sum_i a_i \frakLi(\fs_i;\fe_i)(\theta)=0.
\end{equation*}
For $\epsilon \in \Fq^\times$ we denote by $I(\epsilon)=\{i:\, \chi(\fe_i)=\epsilon\}$ the set of indices whose corresponding character equals $\epsilon$. Then for all $\epsilon \in \Fq^\times$,
	\[ \sum_{i \in I(\epsilon)} a_i \frakLi(\fs_i;\fe_i)(\theta)=0. \]
\end{lemma}

\begin{proof}
We keep the notation of Lemma \ref{lem: gamma_i}. Suppose that we have a relation
	\[\sum_i \gamma_i a_i=0\]
with $a_i \in K_\infty$. By Lemma \ref{lem: gamma_i} and the fact that $K_\infty=\Fq((1/\theta))$, we deduce that $a_i=0$ for all $i$.

By \eqref{eq: ACMPL} the relation $\sum_i a_i \frakLi(\fs_i;\fe_i)(\theta)=0$ is equivalent to the following one
	\[ \sum_i a_i \gamma_{i1} \dots \gamma_{i\ell_i}  \Li \begin{pmatrix}
 \fe_i  \\
\fs_i  \end{pmatrix}=0. \]
By the previous discussion, for all $\epsilon \in \Fq^\times$,
	\[ \sum_{i \in I(\epsilon)} a_i \gamma_{i1} \dots \gamma_{i\ell_i}  \Li \begin{pmatrix}
 \fe_i  \\
\fs_i  \end{pmatrix}=0. \]
By \eqref{eq: ACMPL} again we deduce the desired relation	
	\[ \sum_{i \in I(\epsilon)} a_i \frakLi(\fs_i;\fe_i)(\theta)=0. \]
\end{proof}

\begin{lemma} \label{lem: KuanLin}
Let $m \in \mathbb N$, $\varepsilon \in \Fq^\times$, $\delta \in \overline K[t]$ and $F(t,\theta) \in \overline{\F}_q[t,\theta]$ (resp. $F(t,\theta) \in \F_q[t,\theta]$) satisfying
	\[ \varepsilon\delta = \delta^{(-1)}(t-\theta)^m+F^{(-1)}(t,\theta). \]
Then $\delta \in \overline{\F}_q[t,\theta]$ (resp. $\delta \in \F_q[t,\theta]$) and
	\[ \deg_\theta \delta \leq \max\left\{\frac{qm}{q-1},\frac{\deg_\theta F(t,\theta)}{q}\right\}. \]
\end{lemma}

\begin{proof}
The proof follows the same line as that of \cite[Theorem 2]{KL16} where it is shown that if $F(t,\theta) \in \F_q[t,\theta]$ and $\varepsilon=1$, then $\delta \in \F_q[t,\theta]$. We write down the proof for the case $F(t,\theta) \in \overline{\F}_q[t,\theta]$ for the convenience of the reader.

By twisting once the equality $\varepsilon\delta = \delta^{(-1)}(t-\theta)^m+F^{(-1)}(t,\theta)$ and the fact that $\varepsilon^q=\varepsilon$, we get
	\[ \varepsilon\delta^{(1)} = \delta(t-\theta^q)^m+F(t,\theta). \]
We put $n=\deg_t \delta$ and express
	\[ \delta=a_n t^n + \dots +a_1t+a_0 \in \overline K[t] \]
with $a_0,\dots,a_n \in \overline K$. For $i<0$ we put $a_i=0$.

Since $\deg_t \delta^{(1)}=\deg_t \delta=n < \delta(t-\theta^q)^m=n+m$, it follows that $\deg_t F(t,\theta)=n+m$. Thus we write $F(t,\theta)=b_{n+m} t^{n+m}+\dots+b_1 t+b_0$ with $b_0,\dots,b_{n+m} \in \overline{\F}_q[\theta]$. Plugging into the previous equation, we obtain
	\[ \varepsilon(a_n^q t^n + \dots +a_0^q) = (a_n t^n + \dots +a_0)(t-\theta^q)^m+b_{n+m} t^{n+m}+\dots+b_0. \]

Comparing the coefficients $t^j$ for $n+1 \leq j \leq n+m$ yields
	\[ a_{j-m}+\sum_{i=j-m+1}^{n} {m \choose j-i} (-\theta^q)^{m-j+i} a_i+b_j=0. \]
Since $b_j \in \overline{\F}_q[\theta]$ for all $n+1 \leq j \leq n+m$, we can show by descending induction that $a_j \in \overline{\F}_q[\theta]$ for all $n+1-m \leq j \leq n$.

If $n+1-m \leq 0$, then we are done. Otherwise, comparing the coefficients $t^j$ for $m \leq j \leq n$ yields
	\[ a_{j-m}+\sum_{i=j-m+1}^{n} {m \choose j-i} (-\theta^q)^{m-j+i} a_i+b_j-\varepsilon a_j^q=0. \]
Since $b_j \in \overline{\F}_q[\theta]$ for all $m \leq j \leq n$ and $a_j \in \overline{\F}_q[\theta]$ for all $n+1-m \leq j \leq n$, we can show by descending induction that $a_j \in \overline{\F}_q[\theta]$ for all $0 \leq j \leq n-m$. We conclude that $\delta \in \overline{\F}_q[t,\theta]$.

We now show that $\deg_\theta \delta \leq \max\{\frac{qm}{q-1},\frac{\deg_\theta F(t,\theta)}{q}\}$. Otherwise, suppose that $\deg_\theta \delta > \max\{\frac{qm}{q-1},\frac{\deg_\theta F(t,\theta)}{q}\}$. Then $\deg_\theta \delta^{(1)}=q \deg_\theta \delta$. It implies that $\deg_\theta \delta^{(1)} > \deg_\theta (\delta(t-\theta^q)^m)=\deg_\theta \delta+qm$ and $\deg_\theta \delta^{(1)} >\deg_\theta F(t,\theta)$. Hence we get
	\[\deg_\theta (\varepsilon \delta^{(1)})= \deg_\theta \delta^{(1)} > \deg_\theta(\delta(t-\theta^q)^m+F(t,\theta)), \]
which is a contradiction.
\end{proof}

\subsection{Linear relations: statement of the main result}

\begin{theorem} \label{thm: trans ACMPL}
Let $w \in \N$. We recall that the set  $\mathcal{J}'_w$ consists of positive tuples $\fs = (s_1, \dots, s_n)$ of weight $w$ such that $ q \nmid s_i$ for all $i$. Suppose that we have a non-trivial relation
\begin{equation*}
a+\sum_{\fs_i \in \mathcal{J}'_w} a_i \frakLi(\fs_i;\fe_i)(\theta)=0, \quad \text{for $a, a_i \in K$.}
\end{equation*}
Then $q-1 \mid w$ and $a \neq 0$.

Further, if $q-1 \mid w$, then there is a unique relation
\begin{equation*}
1+\sum_{\fs_i \in \mathcal{J}'_w} a_i \frakLi(\fs_i;\fe_i)(\theta)=0, \quad \text{for $a_i \in K$.}
\end{equation*}
Also, for indices $(\fs_i; \fe_i)$ with nontrivial coefficient $a_i$, we have $\fe_i = (1, \dots, 1)$.

In particular, the ACMPL's in $\mathcal{AS}_w$ are linearly independent over $K$.
\end{theorem}

\begin{remark} \label{rem:remark on Brown}
We emphasize that although Theorem \ref{thm: trans ACMPL} is a purely transcendental result, it is crucial that we need the full strength of algebraic theory for ACMPL's (i.e., Theorem \ref{thm: strong Brown}) to conclude (see the last step of the proof).
\end{remark}

As a direct consequence of Theorem \ref{thm: trans ACMPL}, we obtain:
\begin{theorem} \label{thm:ACMPL}
Let $w \in \N$. Then the ACMPL's in $\mathcal{AS}_w$ form a basis for $\mathcal{AL}_w$. In particular,
	\[ \dim_K \mathcal{AL}_w=s(w). \]
\end{theorem}

\begin{proof}
By Theorem \ref{thm: trans ACMPL} the ACMPL's in $\mathcal{AS}_w$ are all linearly independent over $K$. Then by Theorem \ref{thm: strong Brown} we deduce that the ACMPL's in $\mathcal{AS}_w$ form a basis for $\mathcal{AL}_w$. Hence $\dim_K \mathcal{AL}_w=|\mathcal{AS}_w|=s(w)$ as required.
\end{proof}

\subsection{Proof of Theorem \ref{thm: trans ACMPL}} \ppar

We outline the ideas of the proof. Starting from such a nontrivial relation, we apply the Anderson-Brownawell-Papanikolas criterion in \cite{ABP04} and reduce to the solution of a system of $\sigma$-linear equations. In contrast to \cite[\S 4 and \S 5]{ND21}, this system has a unique solution when $q-1$ divides $w$. We first show that for such a weight $w$ up to a scalar in $K^\times$ there is at most one linear relation between ACMPL's in $\mathcal{AS}_w$ and $\widetilde \pi^w$. Second we show a linear relation between ACMPL's in $\mathcal{AS}_w$ and $\widetilde \pi^w$ where the coefficient of $\widetilde \pi^w$ is nonzero. For this we use Brown's theorem for AMCPL's, i.e., Theorem \ref{thm: strong Brown}.

We are back to the proof of Theorem \ref{thm: trans ACMPL}. We claim that if $q-1 \nmid w$, then any linear relation
	\[ a+\sum_{\fs_i \in \mathcal{J}'_w} a_i \frakLi(\fs_i;\fe_i)(\theta)=0 \]
with $a,a_i \in K$ implies that $a=0$. In fact, if we recall that $\overline{\mathbb F}_q$ denotes the algebraic closure of $\mathbb F_q$ in $\overline{K}$, then the claim follows from Eq \eqref{eq: ACMPL} and that $\widetilde{\pi}^w\not\in \overline{\mathbb F}_q \left(( \frac{1}{\theta}\right))$ since $q-1 \nmid w$.

The proof is by induction on the weight $w \in \N$. For $w=1$, we distinguish two cases:
\begin{itemize}
\item  If $q>2$, then by the previous remark it suffices to show that if
	\[a+\sum_i a_i \frakLi(1;\epsilon_i)(\theta)=0, \]
then $a_i=0$ for all $i$. In fact, it follows immediately from Lemma \ref{lem: same character}. 

\item If $q=2$, then $w=q-1=1$. Then the theorem holds from the facts that there is only one index $(\fs_1;\fe_1)=(1,1)$ and that $\Li(1)=\zeta_A(1)=-D^{-1}_1 \widetilde \pi$.
\end{itemize}

Suppose that Theorem \ref{thm: trans ACMPL} holds for all $w'<w$. We now prove that it holds for $w$. Suppose that we have a linear relation
\begin{equation} \label{eq: non trivial relation MZV}
a+\sum_i a_i \frakLi(\fs_i;\fe_i)(\theta)=0.
\end{equation}

By Lemma \ref{lem: same character} and its proof we can suppose further that $\fe_i$ has the same character, i.e., there exists $\epsilon \in \Fq^\times$ such that for all $i$,
\begin{equation} \label{eq: same character}
\chi(\fe_i)=\epsilon_{i1} \dots \epsilon_{i\ell_i}=\epsilon.
\end{equation}

We now apply Theorem \ref{theorem: linear independence} to our setting of ACMPL's. We recall that by Eq. \eqref{eq: series Li}, 
\[
\frakLi(\fs;\fe)=\frak L(\fs;\fQ_{\fs,\fe}'),
\]
and also
\begin{align} \label{eq: I recall}
I(\fs_i; \fe_i) &= \{\emptyset, (s_{i1}; \epsilon_{i1}), \dots, (s_{i1}, \dots, s_{i(\ell_{i}-1)}; \epsilon_{i1}, \dots, \epsilon_{i(\ell_i-1)})\}, \notag \\
I &= \cup_i I(\fs_i;\fe_i).
\end{align}
We know that the hypothesis are verified:
\begin{itemize}
\item[(LW)] By the induction hypothesis, for any weight $w'<w$, the values $\frakLi(\frak t;\fe)(\theta)$ with $(\frak t;\fe) \in I$ and $w(\frak t)=w'$ are all $K$-linearly independent.

\item[(LD)] By \eqref{eq: non trivial relation MZV}, there exist $a \in A$ and $a_i \in A$ for $1 \leq i \leq n$ which are not all zero such that
\begin{equation*}
a+\sum_{i=1}^n a_i \frakLi(\fs_i;\fe_i)(\theta)=0.
\end{equation*}
\end{itemize}
Thus Theorem \ref{theorem: linear independence} implies that for all $(\mathfrak t;\fe) \in I$, $f_{\mathfrak t;\fe}(\theta)$ belongs to $K$ where $f_{\mathfrak t;\fe}$ is given by
\begin{equation*}
f_{\mathfrak t;\fe}:= \sum_i a_i(t) \frakLi(s_{i(k+1)},\dots,s_{i \ell_i};\epsilon_{i(k+1)},\dots,\epsilon_{i \ell_i}).
\end{equation*}
Here, the sum runs through the set of indices $i$ such that $(\mathfrak t;\fe)=(s_{i1},\dots,s_{i k};\epsilon_{i1},\dots,\epsilon_{i k})$ for some $0 \leq k \leq \ell_i-1$.

We derive a direct consequence of the previous rationality result. Let $(\mathfrak t;\fe) \in I$ and $\mathfrak t \neq \emptyset$. Then $(\mathfrak t;\fe)=(s_{i1},\dots,s_{i k};\epsilon_{i1},\dots,\epsilon_{i k})$ for some $i$ and $1 \leq k \leq \ell_i-1$. We denote by $J(\mathfrak t;\fe)$ the set of all such $i$. We know that there exists $a_{\mathfrak t;\fe} \in K$ such that
\begin{equation*}
a_{\mathfrak t;\fe}+f_{\mathfrak t;\fe}(\theta)=0,
\end{equation*}
or equivalently,
\begin{equation*}
a_{\mathfrak t;\fe}+\sum_{i \in J(\frak t;\fe)} a_i \frakLi(s_{i(k+1)},\dots,s_{i \ell_i};\epsilon_{i(k+1)},\dots,\epsilon_{i \ell_i})(\theta)=0.
\end{equation*}
The ACMPL's appearing in the above equality belong to $\mathcal{AS}_{w-w(\mathfrak t)}$. By the induction hypothesis, we can suppose that $\epsilon_{i(k+1)}=\dots=\epsilon_{i \ell_i}=1$. Further, if $q-1 \nmid w-w(\frak t)$, then $a_i(t)=0$ for all $i \in J(\frak t;\fe)$. Therefore,  letting $(\fs_i;\fe_i)=(s_{i1},\dots,s_{i \ell_i};\epsilon_{i1},\dots,\epsilon_{i \ell_i})$ we can suppose that $s_{i2},\dots,s_{i \ell_i}$ are all divisible by $q-1$ and $\epsilon_{i2}=\dots=\epsilon_{i \ell_i}=1$. In particular, for all $i$, $\epsilon_{i1}=\chi(\fe_i)=\epsilon$.

Now we want to solve \eqref{eq:equation for delta}. Further, in this system we can assume that the corresponding element $b \in \Fq[t] \setminus \{0\}$ equals $1$. We define
	\[ J:=I \cup \{(\fs_i;\fe_i)\} \]
where $I$ is given as in \eqref{eq: I recall}. For $(\mathfrak t;\fe) \in J$ we denote by $J_0(\frak t;\fe)$ consisting of $(\frak t';\fe') \in I$ such that there exist $i$ and $0 \leq j <\ell_i$ so that $(\frak t;\fe)=(s_{i1},s_{i2},\dots,s_{ij};\epsilon,1,\dots,1)$ and $(\frak t';\fe')=(s_{i1},s_{i2},\dots,s_{i(j+1)};\epsilon,1,\dots,1)$. In particular, for $(\frak t;\fe)=(\fs_i;\fe_i)$, $J_0(\frak t;\fe)$ is the empty set. For $(\mathfrak t;\fe) \in J \setminus \{\emptyset\}$, we also put
	\[ m_{\frak t}:=\frac{w-w(\frak t)}{q-1} \in \mathbb Z^{\geq 0}. \]
Then it is clear that \eqref{eq:equation for delta} is equivalent finding $(\delta_{\mathfrak t;\fe})_{(\mathfrak t;\fe) \in J} \in \Mat_{1 \times |J|}(\overline K[t])$ such that
\begin{equation} \label{eq:delta}
\delta_{\mathfrak t;\fe}=\delta_{\mathfrak t;\fe}^{(-1)} (t-\theta)^{w-w(\frak t)}+\sum_{(\mathfrak t';\fe') \in J_0(\frak t;\fe)} \delta_{\mathfrak t';\fe'}^{(-1)} (t-\theta)^{w-w(\frak t)}, \quad \text{for all } (\frak t;\fe) \in J \setminus \{\emptyset\},
\end{equation}
and
\begin{equation} \label{eq:delta emptyset}
\delta_{\mathfrak t;\fe}=\delta_{\mathfrak t;\fe}^{(-1)} (t-\theta)^{w-w(\frak t)}+\sum_{(\mathfrak t';\fe') \in J_0(\frak t;\fe)} \delta_{\mathfrak t';\fe'}^{(-1)} \gamma^{(-1)} (t-\theta)^{w-w(\frak t)}, \quad \text{for } (\frak t;\fe)=\emptyset.
\end{equation}
Here $\gamma^{q-1}=\epsilon$. In fact, for $(\frak t;\fe)=(\fs_i;\fe_i)$, the corresponding equation becomes $\delta_{\fs_i;\fe_i}=\delta_{\fs_i;\fe_i}^{(-1)}$. Thus $\delta_{\fs_i;\fe_i}=a_i(t) \in \Fq[t]$.

Letting $y$ be a variable, we denote by $v_y$ the valuation associated to the place $y$ of the field $\Fq(y)$. We put
	\[ T:=t-t^q, \quad X:=t^q-\theta^q. \]

We claim that
\begin{itemize}
\item[1)] For all $(\mathfrak t;\fe) \in J \setminus \{\emptyset\}$, the polynomial $\delta_{\frak t;\fe}$ is of the form
	\[ \delta_{\frak t;\fe}=f_{\frak t} \left(X^{m_{\frak t}}+\sum_{i=0}^{m_{\frak t}-1} P_{\frak t,i}(T) X^i \right)\]
where
\begin{itemize}
\item $f_{\frak t} \in \Fq[t]$,
\item for all $0 \leq i \leq m_{\frak t}-1$, $P_{\frak t,i}(y)$ belongs to $\Fq(y)$ with $v_y(P_{\frak t,i}) \geq 1$.
\end{itemize}

\item[2)] For all $\mathfrak t \in J \setminus \{\emptyset\}$ and all $\frak t' \in J_0(\frak t)$, there exists $P_{\frak t,\frak t'} \in \Fq(y)$ such that
	\[ f_{\frak t'}=f_{\frak t} P_{\frak t,\frak t'}(T). \]
In particular, if $f_{\frak t}=0$, then $f_{\frak t'}=0$.
\end{itemize}

The proof is by induction on $m_{\frak t}$. We start with $m_{\frak t}=0$. Then $\frak t=\fs_i$ and $\fe=\fe_i$ for some $i$. We have observed that $\delta_{\fs_i;\fe_i}=a_i(t) \in \Fq[t]$ and the assertion follows.

Suppose that the claim holds for all $(\frak t;\fe) \in J \setminus \{\emptyset\}$ with $m_{\frak t}<m$. We now prove the claim for all $(\frak t;\fe) \in J \setminus \{\emptyset\}$ with $m_{\frak t}=m$. In fact, we fix such $\frak t$ and want to find $\delta_{\frak t;\fe} \in \overline K[t]$ such that
\begin{equation} \label{eq:delta1}
 \delta_{\mathfrak t;\fe}=\delta_{\mathfrak t;\fe}^{(-1)} (t-\theta)^{(q-1)m}+\sum_{(\mathfrak t';\fe') \in J_0(\frak t;\fe)} \delta_{\mathfrak t';\fe'}^{(-1)} (t-\theta)^{(q-1)m}.
\end{equation}
By the induction hypothesis, for all $(\frak t';\fe') \in J_0(\frak t;\fe)$, we know that
	\[ \delta_{\frak t';\fe'}=f_{\frak t'} \left(X^{m_{\frak t'}}+\sum_{i=0}^{m_{\frak t'}-1} P_{\frak t',i}(T) X^i \right)\]
where
\begin{itemize}
\item $f_{\frak t'} \in \Fq[t]$,
\item for all $0 \leq i \leq m_{\frak t'}-1$, $P_{\frak t',i}(y) \in \Fq(y)$ with $v_y(P_{\frak t,i}) \geq 1$.
\end{itemize}
For $(\frak t';\fe') \in J_0(\frak t;\fe)$, we write $\frak t'=(\frak t,(m-k)(q-1))$ with $0 \leq k < m$ and $k \not\equiv m \pmod{q}$, in particular $m_{\frak t'}=k$. We put $f_k=f_{\frak t'}$ and $P_{\frak t',i}=P_{k,i}$ so that
\begin{equation} \label{eq:delta2}
\delta_{\frak t';\fe'}=f_k \left(X^k+\sum_{i=0}^{k-1} P_{k,i}(T) X^i \right) \in \Fq[t,\theta^q].
\end{equation}

By Lemma \ref{lem: KuanLin}, $\delta_{\frak t;\fe}$ belongs to $K[t]$, and
$\deg_\theta \delta_{\frak t;\fe} \leq mq$. Further, since $\delta_{\frak t;\fe}$ is divisible by $(t-\theta)^{(q-1)m}$,  we write $\delta_{\frak t;\fe}=F (t-\theta)^{(q-1)m}$ with $F \in K[t]$ and $\deg_\theta F \leq m$. Dividing \eqref{eq:delta1} by $(t-\theta)^{(q-1)m}$ and twisting once yields
\begin{equation} \label{eq:F}
F^{(1)}=F (t-\theta)^{(q-1)m}+\sum_{(\mathfrak t';\fe') \in J_0(\frak t;\fe)} \delta_{\mathfrak t';\fe'}.
\end{equation}
As $\delta_{\frak t';\fe'} \in \Fq[t,\theta^q]$ for all $(\frak t';\fe') \in J_0(\frak t;\fe)$, it follows that $F (t-\theta)^{(q-1)m} \in \Fq[t,\theta^q]$. As $\deg_\theta F \leq m$, we get
	\[ F=\sum_{0 \leq i \leq m/q} f_{m-iq} (t-\theta)^{m-iq}, \quad \text{for $f_{m-iq} \in \Fq[t]$}. \]
Thus
\begin{align*}
& F (t-\theta)^{(q-1)m}=\sum_{0 \leq i \leq m/q} f_{m-iq} (t-\theta)^{mq-iq}=\sum_{0 \leq i \leq m/q} f_{m-iq} X^{m-i}, \\
& F^{(1)}=\sum_{0 \leq i \leq m/q} f_{m-iq} (t-\theta^q)^{m-iq}=\sum_{0 \leq i \leq m/q} f_{m-iq} (T+X)^{m-iq}.
\end{align*}
Putting these and \eqref{eq:delta2} into \eqref{eq:F} gets
\begin{align*}
& \sum_{0 \leq i \leq m/q} f_{m-iq} (T+X)^{m-iq}=\sum_{0 \leq i \leq m/q} f_{m-iq} X^{m-i}+\sum_{\substack{0 \leq k<m \\ k \not\equiv m \pmod{q}}} f_k \left(X^k+\sum_{i=0}^{k-1} P_{k,i}(T) X^i \right).
\end{align*}
Comparing the coefficients of powers of $X$ yields the following linear system in the variables $f_0,\dots,f_{m-1}$:
\begin{equation*}
B_{\big| y=T} \begin{pmatrix}
f_{m-1} \\
\vdots \\
f_0
\end{pmatrix}=f_m \begin{pmatrix}
Q_{m-1} \\
\vdots \\
Q_0
\end{pmatrix}_{\big| y=T}.
\end{equation*}
Here for $0 \leq i \leq m-1$, $Q_i={m \choose i} y^{m-i} \in y\Fq[y]$ and $B=(B_{ij})_{0 \leq i,j \leq m-1} \in \Mat_m(\Fq(y))$ such that
\begin{itemize}
\item $v_y(B_{ij}) \geq 1$ if $i>j$,
\item $v_y(B_{ij}) \geq 0$ if $i<j$,
\item $v_y(B_{ii})=0$ as $B_{ii}=\pm 1$.
\end{itemize}
The above properties follow from the fact that $P_{k,i} \in \Fq(y)$ and $v_y(P_{k,i}) \geq 1$.
Thus $v_y(\det B)=0$ so that $\det B \neq 0$. It follows that for all $0 \leq i \leq m-1$, $f_i=f_m P_i(T)$ with $P_i \in \Fq(y)$ and $v_y(P_i) \geq 1$ and we are done.

To conclude, we have to solve \eqref{eq:delta} for $(\frak t;\fe)=\emptyset$. We have some extra work as we have a factor $\gamma^{(-1)}$ on the right hand side of \eqref{eq:delta emptyset}. We use $\gamma^{(-1)}=\gamma/\epsilon$ and put $\delta:=\delta_{\emptyset,\emptyset}/\gamma \in \overline K[t]$. Then we have to solve
\begin{equation} \label{eq:delta emptyset2}
\epsilon \delta=\delta^{(-1)} (t-\theta)^w+\sum_{(\mathfrak t';\fe') \in J_0(\emptyset)} \delta_{\mathfrak t';\fe'}^{(-1)} (t-\theta)^w.
\end{equation}

We distinguish two cases.
\subsubsection{Case 1: $q-1 \nmid w$, says $w=m(q-1)+r$ with $0<r<q-1$} \ppar

We know that for all $(\frak t';\fe') \in J_0(\emptyset)$, says $\frak t'=((m-k)(q-1)+r)$ with $0 \leq k \leq m$  and $k \not\equiv m-r \pmod{q}$,
\begin{equation} \label{eq:delta4}
\delta_{\frak t';\fe'}=f_k \left(X^k+\sum_{i=0}^{k-1} P_{k,i}(T) X^i \right) \in \Fq[t,\theta^q]
\end{equation}
where
\begin{itemize}
\item $f_k \in \Fq[t]$,
\item for all $0 \leq i \leq k-1$, $P_{k,i}(y)$ belongs to $\Fq(y)$ with $v_y(P_{k,i}) \geq 1$.
\end{itemize}

By Lemma \ref{lem: KuanLin}, $\delta$ belongs to $K[t]$. We claim that $\deg_\theta \delta \leq mq$. Otherwise, we have $\deg_\theta \delta_\emptyset > mq$. Twisting \eqref{eq:delta emptyset2} once gets
\begin{equation*}
\epsilon \delta^{(1)}=\delta (t-\theta^q)^w+\sum_{(\mathfrak t';\fe') \in J_0(\emptyset)} \delta_{\mathfrak t';\fe'} (t-\theta^q)^w.
\end{equation*}
As $\deg_\theta \delta > mq$, we compare the degrees of $\theta$ on both sides and obtain
	\[ q\deg_\theta \delta=\deg_\theta \delta+wq. \]
Thus $q-1 \mid w$, which is a contradiction. We conclude that $\deg_\theta \delta \leq mq$.

From \eqref{eq:delta emptyset2} we see that $\delta$ is divisible by $(t-\theta)^w$. Thus we write $\delta=F (t-\theta)^w$ with $F \in K[t]$ and $\deg_\theta F \leq mq-w=m-r$. Dividing \eqref{eq:delta emptyset2} by $(t-\theta)^w$ and twisting once yields
\begin{equation} \label{eq:F1}
\epsilon F^{(1)}=F (t-\theta)^w+\sum_{(\mathfrak t';\fe') \in J_0(\emptyset)} \delta_{\mathfrak t'}.
\end{equation}
Since $\delta_{\frak t';\fe'} \in \Fq[t,\theta^q]$  for all $(\frak t';\fe') \in J_0(\emptyset)$, it follows that $F (t-\theta)^w \in \Fq[t,\theta^q]$. As $\deg_\theta F \leq m-r$, we write
	\[ F=\sum_{0 \leq i \leq (m-r)/q} f_{m-r-iq} (t-\theta)^{m-r-iq}, \quad \text{for $f_{m-r-iq} \in \Fq[t]$}. \]
It follows that
\begin{align*}
& F (t-\theta)^w=\sum_{0 \leq i \leq (m-r)/q} f_{m-r-iq} (t-\theta)^{mq-iq}=\sum_{0 \leq i \leq (m-r)/q} f_{m-r-iq} X^{m-i}, \\
& F^{(1)}=\sum_{0 \leq i \leq (m-r)/q} f_{m-r-iq} (t-\theta^q)^{m-r-iq}=\sum_{0 \leq i \leq (m-r)/q} f_{m-r-iq} (T+X)^{m-r-iq}.
\end{align*}
Putting these and \eqref{eq:delta4} into \eqref{eq:F1} yields
\begin{align*}
& \epsilon \sum_{0 \leq i \leq (m-r)/q} f_{m-r-iq} (T+X)^{m-r-iq} \\
&=\sum_{0 \leq i \leq (m-r)/q} f_{m-r-iq} X^{m-i}+\sum_{\substack{0 \leq k \leq m \\ k \not\equiv m-r \pmod{q}}} f_k \left(X^k+\sum_{i=0}^{k-1} P_{k,i}(T) X^i \right).
\end{align*}
Comparing the coefficients of powers of $X$ yields the following linear system in the variables $f_0,\dots,f_m$:
\begin{equation*}
B_{\big| y=T} \begin{pmatrix}
f_m \\
\vdots \\
f_0
\end{pmatrix}=0.
\end{equation*}
Here $B=(B_{ij})_{0 \leq i,j \leq m} \in \Mat_{m+1}(\Fq(y))$ such that
\begin{itemize}
\item $v_y(B_{ij}) \geq 1$ if $i>j$,
\item $v_y(B_{ij}) \geq 0$ if $i<j$,
\item $v_y(B_{ii})=0$ as $B_{ii} \in \Fq^\times$.
\end{itemize}
The above properties follow from the fact that $P_{k,i} \in \Fq(y)$ and $v_y(P_{k,i}) \geq 1$.
Thus $v_y(\det B)=0$. Hence $f_0=\dots=f_m=0$. It follows that $\delta_\emptyset=0$ as $\delta=0$ and $\delta_{\frak t';\fe'}=0$ for all $(\frak t';\fe') \in J_0(\emptyset)$. We conclude that $\delta_{\frak t;\fe}=0$ for all $(\frak t;\fe) \in J$. In particular, for all $i$, $a_i(t)=\delta_{\fs_i;\fe_i}=0$, which is a contradiction. Thus this case can never happen.

\subsubsection{Case 2: $q-1 \mid w$, says $w=m(q-1)$} \ppar

By similar arguments as above, we show that $\delta=F (t-\theta)^{(q-1)m}$ with $F \in K[t]$ of the form
	\[ F=\sum_{0 \leq i \leq m/q} f_{m-iq} (t-\theta)^{m-iq}, \quad \text{for $f_{m-iq} \in \Fq[t]$}. \]
Thus
\begin{align*}
& F (t-\theta)^{(q-1)m}=\sum_{0 \leq i \leq m/q} f_{m-iq} (t-\theta)^{mq-iq}=\sum_{0 \leq i \leq m/q} f_{m-iq} X^{m-i}, \\
& F^{(1)}=\sum_{0 \leq i \leq m/q} f_{m-iq} (t-\theta^q)^{m-iq}=\sum_{0 \leq i \leq m/q} f_{m-iq} (T+X)^{m-iq}.
\end{align*}
Putting these and \eqref{eq:delta2} into \eqref{eq:delta emptyset2} gets
\begin{align*}
& \epsilon \sum_{0 \leq i \leq m/q} f_{m-iq} (T+X)^{m-iq}\\
&=\sum_{0 \leq i \leq m/q} f_{m-iq} X^{m-i}+\sum_{\substack{0 \leq k<m \\ k \not\equiv m \pmod{q}}} f_k \left(X^k+\sum_{i=0}^{k-1} P_{k,i}(T) X^i \right).
\end{align*}

Comparing the coefficients of powers of $X$ yields
	\[ \epsilon f_m=f_m \]
and the following linear system in the variables $f_0,\dots,f_{m-1}$:
\begin{equation*}
B_{\big| y=T} \begin{pmatrix}
f_{m-1} \\
\vdots \\
f_0
\end{pmatrix}=f_m \begin{pmatrix}
Q_{m-1} \\
\vdots \\
Q_0
\end{pmatrix}_{\big| y=T}.
\end{equation*}
Here for $0 \leq i \leq m-1$, $Q_i={m \choose i} y^{m-i} \in y\Fq[y]$ and $B=(B_{ij})_{0 \leq i,j \leq m-1} \in \Mat_m(\Fq(y))$ such that
\begin{itemize}
\item $v_y(B_{ij}) \geq 1$ if $i>j$,
\item $v_y(B_{ij}) \geq 0$ if $i<j$,
\item $v_y(B_{ii})=0$ as $B_{ii} \in \Fq^\times$.
\end{itemize}
The above properties follow from the fact that $P_{k,i} \in \Fq(y)$ and $v_y(P_{k,i}) \geq 1$.
Thus $v_y(\det B)=0$ so that $\det B \neq 0$.

We distinguish two subcases.

\medskip
\noindent {\bf Subcase 1: $\epsilon \neq 1$.} \ppar
It follows that $f_m=0$. Then $f_0=\dots=f_{m-1}=0$. Thus $\delta_{\frak t;\fe}=0$ for all $(\frak t;\fe) \in J$. In particular, for all $i$, $a_i(t)=\delta_{\fs_i;\fe_i}=0$. This is a contradiction and we conclude that this case can never happen.

\medskip
\noindent {\bf Subcase 2: $\epsilon=1$.} \ppar

It follows that $\gamma \in \Fq^\times$ and thus
\begin{itemize}
\item[1)] The polynomial $\delta_\emptyset=\delta \gamma$ is of the form
	\[ \delta_\emptyset=f_\emptyset \left(X^m+\sum_{i=0}^{m-1} P_{\emptyset,i}(T) X^i \right)\]
with
\begin{itemize}
\item $f_\emptyset \in \Fq[t]$,
\item for all $0 \leq i \leq m-1$, $P_{\emptyset,i}(y) \in  \Fq(y)$ with $v_y(P_{\emptyset,i}) \geq 1$.
\end{itemize}

\item[2)] For all $(\frak t';\fe') \in J_0(\emptyset)$, there exists $P_{\emptyset,\frak t'} \in \Fq(y)$ such that
	\[ f_{\frak t'}=f_\emptyset P_{\emptyset,\frak t'}(T). \]
\end{itemize}

Hence there exists a unique solution $(\delta_{\mathfrak t;\fe})_{(\mathfrak t;\fe) \in J} \in \Mat_{1 \times |J|}(K[t])$ of \eqref{eq:delta} up to a factor in $\Fq(t)$. Recall that for all $i$, $a_i(t)=\delta_{\fs_i;\fe_i}$. Therefore, up to a scalar in $K^\times$, there exists at most one non-trivial relation
	\[ a \widetilde \pi^w+\sum_i a_i \Li \begin{pmatrix}
 \fve_i  \\
\fs_i  \end{pmatrix}=0 \]
with $a_i \in K$ and $\Li \begin{pmatrix}
 \fve_i  \\
\fs_i  \end{pmatrix} \in \mathcal{AS}_w$. Further, we must have $\fve_i=(1,\dots,1)$ for all $i$.

To conclude, it suffices to exhibit such a relation with $a \neq 0$. In fact, we recall $w=(q-1)m$ and then express $\Li(q-1)^m=\Li \begin{pmatrix}
 1 \\
q-1  \end{pmatrix}^m$ as a $K$-linear combination of ACMPL's of weight $w$.  By Theorem \ref{thm: strong Brown}, we can write
	\[ \Li(q-1)^m=\Li \begin{pmatrix}
 1 \\
q-1  \end{pmatrix}^m=\sum_i a_i \Li \begin{pmatrix}
 \fve_i  \\
\fs_i  \end{pmatrix}, \quad \text{where $a_i \in K$, $\Li \begin{pmatrix}
 \fve_i  \\
\fs_i  \end{pmatrix} \in \mathcal{AS}_w.$} \]
We note that $\Li(q-1)=\zeta_A(q-1)=-D^{-1}_1 \widetilde \pi^{q-1}$. Thus
	\[ (-D_1)^{-m} \widetilde \pi^w-\sum_i a_i \Li \begin{pmatrix}
 \fve_i  \\
\fs_i  \end{pmatrix}=0, \]
which is the desired relation. 


\section{Applications on AMZV's and Zagier-Hoffman's conjectures in positive characteristic} \label{sec:applications}

In this section we give two applications of the study of ACMPL's.

First we use Theorem \ref{thm:ACMPL} to prove Theorem \ref{thm: ZagierHoffman AMZV} which calculates the dimensions of the vector space $\mathcal{AZ}_w$ of alternating multiple zeta values in positive characteristic (AMZV's) of fixed weight introduced by Harada \cite{Har21}. Consequently we determine all linear relations for AMZV's. To do so we develop an algebraic theory to obtain a weak version of Brown's theorem for AMZV's. Then we deduce that $\mathcal{AZ}_w$ and $\mathcal{AL}_w$ are equal and conclude. In contrast to the setting of MZV's, although the results are clean, we are unable to obtain either sharp upper bounds or sharp lower bounds for $\mathcal{AZ}_w$ for general $w$ without the theory of ACMPL's.

Second we restrict our attention to MZV's and determine all linear relations between MZV's. In particular, we obtain a proof of Zagier-Hoffman's conjectures in positive characteristic in full generality (i.e., Theorem \ref{thm: ZagierHoffman}) and generalize the work of one of the authors \cite{ND21}.

\subsection{Linear relations between AMZV's} \label{sec: application AMZV}

\subsubsection{Preliminaries}

For $d \in \mathbb{Z}$ and for $\fs=(s_1,\dots,s_n) \in \N^n$, recalling $S_d(\fs)$ and $S_{<d}(\fs)$ given in \S \ref{sec: definition Sd}, and further letting $\begin{pmatrix}
 \fve  \\
\fs  \end{pmatrix}  =  \begin{pmatrix}
 \varepsilon_1 & \dots & \varepsilon_n \\
s_1 & \dots & s_n \end{pmatrix} $ be an array, we recall (see \S \ref{sec: definition Sd})
\begin{equation*}
S_d \begin{pmatrix}
\fve \\ \fs
\end{pmatrix}  = \sum\limits_{\substack{a_1, \dots, a_n \in A_{+} \\ d = \deg a_1> \dots > \deg a_n\geq 0}} \dfrac{\varepsilon_1^{\deg a_1} \dots \varepsilon_n^{\deg a_n }}{a_1^{s_1} \dots a_n^{s_n}} \in K
\end{equation*}
and
\begin{equation*}
    S_{<d} \begin{pmatrix}
 \fve  \\
\fs  \end{pmatrix}  = \sum\limits_{\substack{a_1, \dots, a_n \in A_{+} \\ d > \deg a_1> \dots > \deg a_n\geq 0}} \dfrac{\varepsilon_1^{\deg a_1} \dots \varepsilon_n^{\deg a_n }}{a_1^{s_1} \dots a_n^{s_n}} \in K.
\end{equation*}
One verifies easily the following formulas:
\begin{align*}
& S_{<d} \begin{pmatrix}
 1& \dots & 1 \\
s_1 & \dots & s_n \end{pmatrix}  = S_{<d}(s_1, \dots, s_n),\quad S_d \begin{pmatrix}
 1 &\dots & 1 \\
s_1 & \dots & s_n \end{pmatrix}  = S_{d}(s_1, \dots, s_n),\\
& S_{d} \begin{pmatrix}
 \varepsilon \\
s  \end{pmatrix}  = \varepsilon^d S_d(s),\quad S_{d} \begin{pmatrix}
 \fve  \\
\fs  \end{pmatrix}  = S_{d} \begin{pmatrix}
 \varepsilon_1  \\
s_1  \end{pmatrix} S_{<d} \begin{pmatrix}
 \fve_{-}  \\
\fs_{-}  \end{pmatrix}.
\end{align*}

Harada \cite{Har21} introduced the alternating multiple zeta value (AMZV) as follows;
\begin{equation*}
    \zeta_A \begin{pmatrix}
 \fve  \\
\fs  \end{pmatrix}  = \sum \limits_{d \geq 0} S_d \begin{pmatrix}
 \fve  \\
\fs  \end{pmatrix}  = \sum\limits_{\substack{a_1, \dots, a_n \in A_{+} \\ \deg a_1> \dots > \deg a_n\geq 0}} \dfrac{\varepsilon_1^{\deg a_1} \dots \varepsilon_n^{\deg a_n }}{a_1^{s_1} \dots a_n^{s_n}}  \in K_{\infty}.
\end{equation*}

Using Chen's formula (see \cite{Che15}), Harada proved that for $s, t \in \mathbb{N}$ and $\varepsilon, \epsilon \in \mathbb{F}_q^{\times}$, we have
\begin{equation} \label{eq: sum AMZV}
S_d \begin{pmatrix}
 \varepsilon \\
s  \end{pmatrix} S_d \begin{pmatrix}
 \epsilon  \\
t  \end{pmatrix}  = S_d \begin{pmatrix}
 \varepsilon\epsilon  \\
s+t  \end{pmatrix}  + \sum \limits_i \Delta^i_{s,t} S_d \begin{pmatrix}
 \varepsilon\epsilon  & 1 \\
s+t-i & i \end{pmatrix} ,
\end{equation}
where
\begin{equation} \label{eq:Delta Chen}
    \Delta^i_{s,t} = \begin{cases}
			(-1)^{s-1} {i - 1  \choose s - 1} +  (-1)^{t-1} {i-1 \choose t-1} & \quad  \text{if } q - 1 \mid i \text{ and } 0 < i < s + t, \\
            0 & \quad \text{otherwise.}
		 \end{cases}
\end{equation}

\begin{remark} \label{rmk: reduce sum}
    When $s + t \leq q$, we deduce from the above formulas that 
    \begin{align*}
        S_d \begin{pmatrix}
 \varepsilon \\
s  \end{pmatrix} S_d \begin{pmatrix}
 \epsilon  \\
t  \end{pmatrix}  = S_d \begin{pmatrix}
 \varepsilon\epsilon  \\
s+t  \end{pmatrix}.
    \end{align*}
\end{remark}

He then proved similar results for products of AMZV's (see \cite{Har21}):
\begin{proposition} \label{sums}
Let $ \begin{pmatrix}
 \fve  \\
\fs  \end{pmatrix} $, $ \begin{pmatrix}
 \fe  \\
\mathfrak{t}  \end{pmatrix} $ be two arrays. Then
\begin{enumerate}
    \item There exist $f_i \in \mathbb{F}_q$ and arrays $ \begin{pmatrix}
 \fm_i  \\
\mathfrak{u}_i  \end{pmatrix} $ with $ \begin{pmatrix}
 \fm_i  \\
\mathfrak{u}_i  \end{pmatrix}  \leq  \begin{pmatrix}
 \fve  \\
\fs  \end{pmatrix}  +  \begin{pmatrix}
 \fe  \\
\mathfrak{t}  \end{pmatrix}  $ and $\depth(\mathfrak{u}_i) \leq \depth(\fs) + \depth(\mathfrak{t})$ for all $i$  such that
    \begin{equation*}
        S_d \begin{pmatrix}
 \fve  \\
\fs  \end{pmatrix} S_d \begin{pmatrix}
 \fe  \\
\mathfrak{t}  \end{pmatrix}  = \sum \limits_i f_i S_d \begin{pmatrix}
 \fm_i  \\
\mathfrak{u}_i  \end{pmatrix}  \quad \text{for all } d \in \mathbb{Z}.
    \end{equation*}
    \item There exist $f'_i \in \mathbb{F}_q$ and arrays $ \begin{pmatrix}
 \fm'_i  \\
\mathfrak{u}'_i  \end{pmatrix} $ with $ \begin{pmatrix}
 \fm'_i  \\
\mathfrak{u}'_i  \end{pmatrix}  \leq  \begin{pmatrix}
 \fve  \\
\fs  \end{pmatrix}  +  \begin{pmatrix}
 \fe  \\
\mathfrak{t}  \end{pmatrix}  $ and $\depth(\mathfrak{u}'_i) \leq \depth(\fs) + \depth(\mathfrak{t})$ for all $i$  such that
    \begin{equation*}
        S_{<d} \begin{pmatrix}
 \fve  \\
\fs  \end{pmatrix} S_{<d} \begin{pmatrix}
 \fe  \\
\mathfrak{t}  \end{pmatrix}  = \sum \limits_i f'_i S_{<d} \begin{pmatrix}
 \fm'_i  \\
\mathfrak{u}'_i  \end{pmatrix}  \quad \text{for all } d \in \mathbb{Z}.
    \end{equation*}
        \item There exist $f''_i \in \mathbb{F}_q$ and arrays $ \begin{pmatrix}
 \fm''_i  \\
\mathfrak{u}''_i  \end{pmatrix} $ with $ \begin{pmatrix}
 \fm''_i  \\
\mathfrak{u}''_i  \end{pmatrix}  \leq  \begin{pmatrix}
 \fve  \\
\fs  \end{pmatrix}  +  \begin{pmatrix}
 \fe  \\
\mathfrak{t}  \end{pmatrix}  $ and $\depth(\mathfrak{u}''_i) \leq \depth(\fs) + \depth(\mathfrak{t})$ for all $i$  such that
    \begin{equation*}
        S_d \begin{pmatrix}
 \fve  \\
\fs  \end{pmatrix} S_{<d} \begin{pmatrix}
 \fe  \\
\mathfrak{t}  \end{pmatrix}  = \sum \limits_i f''_i S_d \begin{pmatrix}
 \fm''_i  \\
\mathfrak{u}''_i  \end{pmatrix}  \quad \text{for all } d \in \mathbb{Z}.
    \end{equation*}
\end{enumerate}
\end{proposition}

We denote by $\mathcal{AZ}$ the $K$-vector space generated by the AMZV's and $\mathcal{AZ}_w$ the $K$-vector space generated by the AMZV's of weight $w$. It follows from Proposition~\ref{sums} that $\mathcal{AZ}$ is a $K$-algebra.

\subsubsection{Algebraic theory for AMZV's} \label{sec: alg theory AMZV} 

We can extend an algebraic theory for AMZV's which follow the same line as that in \S \ref{sec:algebraic part}. 

\begin{definition}
A binary relation is a $K$-linear combination of the form
\begin{equation*}
    \sum \limits_i a_i S_d \begin{pmatrix}
 \fve_i  \\
\fs_i  \end{pmatrix}  + \sum \limits_i b_i S_{d+1} \begin{pmatrix}
 \fe_i  \\
\mathfrak{t}_i  \end{pmatrix}  =0 \quad \text{for all } d \in \mathbb{Z},
\end{equation*}
where $a_i,b_i \in K$ and $ \begin{pmatrix}
 \fve_i  \\
\fs_i  \end{pmatrix} ,  \begin{pmatrix}
 \fe_i  \\
\mathfrak{t}_i  \end{pmatrix} $ are arrays of the same weight.

A binary relation is called a fixed relation if $b_i = 0$ for all $i$.
\end{definition}

We denote by $\mathfrak{R}_{w}$ the set of all binary relations of weight $w$. From Lemma \ref{agree} and the relation $R_{\varepsilon}$ defined in \S \ref{sec:Todd}, we obtain the following binary relation
\begin{equation*}
 R_{\varepsilon} \colon \quad  S_d \begin{pmatrix}
 \varepsilon\\
q  \end{pmatrix}  + \varepsilon^{-1}D_1 S_{d+1} \begin{pmatrix}
 \varepsilon& 1 \\
1 & q-1  \end{pmatrix}  =0,
\end{equation*}
where $D_1 = \theta^q - \theta$.

For later definitions, let $R \in \mathfrak{R}_w$ be a binary relation of the form
\begin{equation} \label{eq: Rd}
    R(d) \colon \quad \sum \limits_i a_i S_d \begin{pmatrix}
 \fve_i  \\
\fs_i  \end{pmatrix}  + \sum \limits_i b_i S_{d+1} \begin{pmatrix}
 \fe_i  \\
\mathfrak{t}_i  \end{pmatrix}  =0,
\end{equation}
where $a_i,b_i \in K$ and $ \begin{pmatrix}
 \fve_i  \\
\fs_i  \end{pmatrix} ,  \begin{pmatrix}
 \fe_i  \\
\mathfrak{t}_i  \end{pmatrix} $ are arrays of the same weight. We now define some operators on $K$-vector spaces of binary relations.

First, we define operators $\mathcal B^*$. Let $ \begin{pmatrix}
 \sigma  \\
v  \end{pmatrix} $ be an array. We introduce
\begin{equation*}
    \mathcal B^*_{\sigma,v} \colon \mathfrak{R}_{w} \longrightarrow \mathfrak{R}_{w+v}
\end{equation*}
as follows: for each $R \in \mathfrak{R}_{w}$ as given in \eqref{eq: Rd},
the image $\mathcal B^*_{\sigma,v}(R) = S_d \begin{pmatrix}
 \sigma  \\
v  \end{pmatrix} \sum_{j < d} R(j)$ is a fixed relation of the form
\begin{align*}
    0 &= S_d \begin{pmatrix}
 \sigma  \\
v  \end{pmatrix}   \left(\sum \limits_ia_i S_{<d} \begin{pmatrix}
 \fve_i  \\
\fs_i  \end{pmatrix}  + \sum \limits_i  b_i S_{<d+1} \begin{pmatrix}
 \fe_i  \\
\mathfrak{t}_i  \end{pmatrix} \right)  \\
    &= \sum \limits_i a_i S_d \begin{pmatrix}
 \sigma  \\
v  \end{pmatrix} S_{<d} \begin{pmatrix}
 \fve_i  \\
\fs_i  \end{pmatrix}  + \sum \limits_i  b_i S_d \begin{pmatrix}
 \sigma  \\
v  \end{pmatrix}  S_{<d} \begin{pmatrix}
 \fe_i  \\
\mathfrak{t}_i  \end{pmatrix}  + \sum \limits_i  b_i S_d \begin{pmatrix}
 \sigma  \\
v  \end{pmatrix}  S_{d} \begin{pmatrix}
 \fe_i  \\
\mathfrak{t}_i  \end{pmatrix} \\
    &= \sum \limits_i a_i S_d \begin{pmatrix}
\sigma & \fve_i  \\
v& \fs_i  \end{pmatrix}  + \sum \limits_i  b_i S_d \begin{pmatrix}
\sigma & \fe_i  \\
v& \mathfrak{t}_i  \end{pmatrix}  + \sum \limits_i  b_i \sum \limits_j f_{i,j} S_d \begin{pmatrix}
 \fm_{i,j}  \\
\mathfrak{u}_{i,j}  \end{pmatrix} .
\end{align*}
The last equality follows from Proposition \ref{sums}.

Let $ \begin{pmatrix}
 \Sigma  \\
V  \end{pmatrix}  =  \begin{pmatrix}
 \sigma_1 & \dots & \sigma_n \\
v_1 & \dots & v_n \end{pmatrix} $ be an array. We define an operator $\mathcal{B}^*_{\Sigma,V}(R) $ by
\begin{equation*}
    \mathcal B^*_{\Sigma,V}(R) := \mathcal B^*_{\sigma_1,v_1} \circ \dots \circ \mathcal B^*_{\sigma_n,v_n}(R).
\end{equation*}

\begin{lemma} \label{polybesao AMZV}
Let $ \begin{pmatrix}
 \Sigma  \\
V  \end{pmatrix}  =  \begin{pmatrix}
 \sigma_1 & \dots & \sigma_n \\
v_1 & \dots & v_n \end{pmatrix} $ be an array. Under the notations of \eqref{eq: Rd}, suppose that for all $i$, $v_n + t_{i1} \leq q$ where $\mathfrak{t}_{i} = (t_{i1}, \mathfrak{t}_{i-})$ . Then $\mathcal B^*_{\Sigma,V}(R)$ is of the form
\begin{align*}
 \sum \limits_i a_i S_d \begin{pmatrix}
\Sigma & \fve_i  \\
V& \fs_i  \end{pmatrix} & + \sum \limits_i  b_i S_d \begin{pmatrix}
\Sigma & \fe_i  \\
V& \mathfrak{t}_i  \end{pmatrix}
\\ & + \sum \limits_i  b_i S_d  \begin{pmatrix}
 \sigma_1 & \dots & \sigma_{n-1} & \sigma_n \epsilon_{i1} &\fe_{i-}  \\
v_1 & \dots & v_{n-1} & v_n+ t_{i1} & \mathfrak{t}_{i-}  \end{pmatrix}  = 0 .
\end{align*}
\end{lemma}

\begin{proof}
From the definition, we have $\mathcal{B}^*_{\sigma_n,v_n}(R)$ is of the form
\begin{align*}
    \sum \limits_i a_i S_d \begin{pmatrix}
\sigma_n & \fve_i  \\
v_n& \fs_i  \end{pmatrix}  + \sum \limits_i  b_i S_d \begin{pmatrix}
\sigma_n & \fe_i  \\
v_n& \mathfrak{t}_i  \end{pmatrix}  + \sum \limits_i  b_i S_d \begin{pmatrix}
 \sigma_n  \\
v_n  \end{pmatrix}  S_{d} \begin{pmatrix}
 \fe_i  \\
\mathfrak{t}_i  \end{pmatrix} = 0.
\end{align*}
For all $i$, since $v_n + t_{i1} \leq q$, it follows from Remark \ref{rmk: reduce sum} that 
\begin{align*}
    S_d \begin{pmatrix}
 \sigma_n  \\
v_n  \end{pmatrix}  S_{d} \begin{pmatrix}
 \fe_i  \\
\mathfrak{t}_i  \end{pmatrix} = S_d \begin{pmatrix}
 \sigma_n  \\
v_n  \end{pmatrix} S_d \begin{pmatrix}
 \epsilon_{i1}  \\
t_{i1}  \end{pmatrix}  S_{<d} \begin{pmatrix}
 {\fe_i}_-  \\
{\mathfrak{t}_i}_-  \end{pmatrix}
&= S_d \begin{pmatrix}
 \sigma_n \epsilon_{i1} \\
v_n +  t_{i1} \end{pmatrix} S_{<d} \begin{pmatrix}
 {\fe_i}_-  \\
{\mathfrak{t}_i}_-  \end{pmatrix}\\
&= S_d \begin{pmatrix}
 \sigma_n \epsilon_{i1} & {\fe_i}_- \\
v_n +  t_{i1} & {\mathfrak{t}_i}_- \end{pmatrix},
\end{align*} 
hence $\mathcal{B}^*_{\sigma_n,v_n}(R)$ is of the form
\begin{align*}
    \sum \limits_i a_i S_d \begin{pmatrix}
\sigma_n & \fve_i  \\
v_n& \fs_i  \end{pmatrix}  + \sum \limits_i  b_i S_d \begin{pmatrix}
\sigma_n & \fe_i  \\
v_n& \mathfrak{t}_i  \end{pmatrix}  + \sum \limits_i  b_i S_d \begin{pmatrix}
 \sigma_n \epsilon_{i1} & {\fe_i}_- \\
v_n +  t_{i1} & {\mathfrak{t}_i}_- \end{pmatrix} = 0.
\end{align*}
Apply the operator $\mathcal B^*_{\sigma_1,v_1} \circ \dots \circ \mathcal B^*_{\sigma_{n - 1},v_{n - 1}}$ to $\mathcal{B}^*_{\sigma_n,v_n}(R)$, the result then follows from the definition.
\end{proof}

Second, we define operators $\mathcal C$. Let $ \begin{pmatrix}
 \Sigma  \\
V  \end{pmatrix} $ be an array of weight $v$. We introduce
\begin{equation*}
     \mathcal C_{\Sigma,V}(R) \colon \mathfrak{R}_{w} \longrightarrow \mathfrak{R}_{w+v}
\end{equation*}
as follows: for each $R \in \mathfrak{R}_{w}$ as given in \eqref{eq: Rd},
the image $\mathcal C_{\Sigma,V}(R) = R(d) S_{<d+1} \begin{pmatrix}
 \Sigma  \\
V  \end{pmatrix} $ is a binary relation of the form
\begin{align*}
    0 &= \left( \sum \limits_i a_i S_d \begin{pmatrix}
 \fve_i  \\
\fs_i  \end{pmatrix}  + \sum \limits_i b_i S_{d+1} \begin{pmatrix}
 \fe_i  \\
\mathfrak{t}_i  \end{pmatrix} \right) S_{<d+1} \begin{pmatrix}
 \Sigma  \\
V  \end{pmatrix}   \\
    &= \sum \limits_i a_i S_d \begin{pmatrix}
 \fve_i  \\
\fs_i  \end{pmatrix} S_{d} \begin{pmatrix}
 \Sigma  \\
V  \end{pmatrix}  + \sum \limits_i a_i S_d \begin{pmatrix}
 \fve_i  \\
\fs_i  \end{pmatrix} S_{<d} \begin{pmatrix}
 \Sigma  \\
V  \end{pmatrix}  + \sum \limits_i b_i S_{d+1} \begin{pmatrix}
 \fe_i  \\
\mathfrak{t}_i  \end{pmatrix} S_{<d+1} \begin{pmatrix}
 \Sigma  \\
V  \end{pmatrix} \\
    &= \sum \limits_i f_i S_d \begin{pmatrix}
 \fm_i  \\
\mathfrak{u}_i  \end{pmatrix}  + \sum \limits_i f'_i S_{d+1} \begin{pmatrix}
 \fm'_i  \\
\mathfrak{u}'_i  \end{pmatrix} .
\end{align*}
The last equality follows from Proposition \ref{sums}.

In particular, the following proposition gives the form of $\mathcal C_{\Sigma,V}(R_{\varepsilon})$.

\begin{proposition} \label{polycer1 AMZV}
Let $ \begin{pmatrix}
 \Sigma  \\ V  \end{pmatrix}$ be an array with $V = (v_1,V_{-})$ and $\Sigma = (\sigma_1, \Sigma_{-})$. Then $\mathcal C_{\Sigma,V}(R_{\varepsilon})$ is of the form
\begin{equation*}
  S_d \begin{pmatrix}
 \varepsilon\sigma_1 & \Sigma_{-} \\
q + v_1 & V_{-} \end{pmatrix}  +  \sum \limits_i a_i S_d \begin{pmatrix}
\fve_i  \\
\fs_i  \end{pmatrix}  + \sum \limits_i b_i S_{d+1} \begin{pmatrix}
\varepsilon& \fe_i  \\
1 & \mathfrak{t}_i  \end{pmatrix}  =0,
\end{equation*}
where $a_i, b_i \in K$ and $\begin{pmatrix}
\fve_i  \\
\fs_i  \end{pmatrix},  \begin{pmatrix}
 \fe_i  \\
\mathfrak{t}_i  \end{pmatrix} $ are arrays satisfying
\begin{itemize}
    \item $\begin{pmatrix}
\fve_i  \\
\fs_i  \end{pmatrix} \leq \begin{pmatrix}
\varepsilon  \\
q \end{pmatrix} + \begin{pmatrix}
\Sigma \\
V  \end{pmatrix}$ and $s_{i1} < q + v_1$ for all $i$;
\item $ \begin{pmatrix}
 \fe_i  \\
\mathfrak{t}_i  \end{pmatrix}  \leq  \begin{pmatrix}
 1  \\
q - 1 \end{pmatrix}  +  \begin{pmatrix}
 \Sigma  \\
V  \end{pmatrix} $ for all $i$.
\end{itemize} 
\end{proposition}

\begin{proof}
From the definition,  $\mathcal C_{\Sigma,V}(R_{\varepsilon})$ is of the form
\begin{equation*}
    S_d \begin{pmatrix}
\varepsilon \\
q  \end{pmatrix} S_{d} \begin{pmatrix}
 \Sigma  \\
V  \end{pmatrix}  + S_d \begin{pmatrix}
 \varepsilon \\
q  \end{pmatrix} S_{<d} \begin{pmatrix}
 \Sigma  \\
V  \end{pmatrix}  + \varepsilon^{-1} D_1 S_{d+1} \begin{pmatrix}
\varepsilon&  1  \\
1 & q-1 \end{pmatrix} S_{<d+1} \begin{pmatrix}
 \Sigma  \\
V  \end{pmatrix}  = 0.
\end{equation*}
It follows from \eqref{eq: sum AMZV} and Proposition \ref{sums} that
\begin{align*}
        S_d \begin{pmatrix}
\varepsilon \\
q  \end{pmatrix} S_{d} \begin{pmatrix}
 \Sigma  \\
V  \end{pmatrix}  + S_d \begin{pmatrix}
 \varepsilon \\
q  \end{pmatrix} S_{<d} \begin{pmatrix}
 \Sigma  \\
V  \end{pmatrix}  &=   S_d \begin{pmatrix}
 \varepsilon\sigma_1 & \Sigma_{-} \\
q + v_1 & V_{-} \end{pmatrix}  +  \sum \limits_i a_i S_d \begin{pmatrix}
\fve_i  \\
\fs_i  \end{pmatrix},  \\
  \varepsilon^{-1}D_1 S_{d+1} \begin{pmatrix}
\varepsilon&  1  \\
1 & q-1 \end{pmatrix} S_{<d+1} \begin{pmatrix}
 \Sigma  \\
V  \end{pmatrix}   &= \sum \limits_i b_i S_{d+1} \begin{pmatrix}
 \varepsilon& \fe_i  \\
1 & \mathfrak{t}_i  \end{pmatrix} ,
\end{align*}
where $a_i, b_i \in K$ and $\begin{pmatrix}
\fve_i  \\
\fs_i  \end{pmatrix},  \begin{pmatrix}
 \fe_i  \\
\mathfrak{t}_i  \end{pmatrix} $ are arrays satisfying
\begin{itemize}
    \item $\begin{pmatrix}
\fve_i  \\
\fs_i  \end{pmatrix} \leq \begin{pmatrix}
\varepsilon  \\
q \end{pmatrix} + \begin{pmatrix}
\Sigma \\
V  \end{pmatrix}$ and $s_{i1} < q + v_1$ for all $i$;
\item $ \begin{pmatrix}
 \fe_i  \\
\mathfrak{t}_i  \end{pmatrix}  \leq  \begin{pmatrix}
 1  \\
q - 1 \end{pmatrix}  +  \begin{pmatrix}
 \Sigma  \\
V  \end{pmatrix} $ for all $i$.
\end{itemize} 
This proves the proposition.
\end{proof}

Finally, we define operators $\mathcal{BC}$. Let $\varepsilon \in \mathbb{F}_q^{\times}$. We introduce
\begin{equation*}
   \mathcal{BC}_{\varepsilon,q} \colon \mathfrak{R}_{w} \longrightarrow \mathfrak{R}_{w+q}
\end{equation*}
as follows: for each $R \in \mathfrak{R}_{w}$ as given in \eqref{eq: Rd}, the image $\mathcal{BC}_{\varepsilon,q}(R)$ is a binary relation given by
\begin{align*}
    \mathcal{BC}_{\varepsilon,q}(R) = \mathcal B^*_{\varepsilon,q}(R) - \sum\limits_i b_i \mathcal C_{\fe_i,\mathfrak{t}_i} (R_{\varepsilon}).
\end{align*}

Let us clarify the definition of $\mathcal{BC}_{\varepsilon,q}$. We know that $\mathcal B^*_{\varepsilon,q}(R)$ is of the form
\begin{equation*}
    \sum \limits_i a_i S_d \begin{pmatrix}
\varepsilon& \fve_i  \\
q& \fs_i  \end{pmatrix}  + \sum \limits_i  b_i S_d \begin{pmatrix}
\varepsilon& \fe_i  \\
q& \mathfrak{t}_i  \end{pmatrix}  + \sum \limits_i  b_i S_d \begin{pmatrix}
 \varepsilon  \\
q  \end{pmatrix}  S_{d} \begin{pmatrix}
 \fe_i  \\
\mathfrak{t}_i  \end{pmatrix} = 0.
\end{equation*}
Moreover, $\mathcal C_{\fe_i,\mathfrak{t}_i} (R_{\varepsilon})$ is of the form
\begin{equation*}
S_d \begin{pmatrix}
\varepsilon& \fe_i  \\
q& \mathfrak{t}_i  \end{pmatrix}  +  S_d \begin{pmatrix}
 \varepsilon  \\
q  \end{pmatrix}  S_{d} \begin{pmatrix}
 \fe_i  \\
\mathfrak{t}_i  \end{pmatrix}  + \varepsilon^{-1}D_1 S_{d+1} \begin{pmatrix}
 \varepsilon \\
1  \end{pmatrix} S_{<d+1} \begin{pmatrix}
1  \\
q-1  \end{pmatrix} S_{<d+1} \begin{pmatrix}
\fe_i  \\
\mathfrak{t}_i  \end{pmatrix} = 0.
\end{equation*}
Combining with Proposition \ref{sums}, we have that $\mathcal{BC}_{\varepsilon,q}(R)$ is of the form
\begin{equation*}
   \sum \limits_i a_i S_d \begin{pmatrix}
\varepsilon& \fve_i  \\
q& \fs_i  \end{pmatrix}  + \sum \limits_{i,j} b_{ij} S_{d+1} \begin{pmatrix}
\varepsilon& \fe_{ij}  \\
1& \mathfrak{t}_{ij}  \end{pmatrix} =0,
\end{equation*}
where $b_{ij} \in K$ and $ \begin{pmatrix}
\fe_{ij}  \\
\mathfrak{t}_{ij}  \end{pmatrix} $ are arrays satisfying $ \begin{pmatrix}
\fe_{ij}  \\
\mathfrak{t}_{ij}  \end{pmatrix}  \leq  \begin{pmatrix}
1  \\
q-1  \end{pmatrix}  +  \begin{pmatrix}
\fe_{i}  \\
\mathfrak{t}_{i}  \end{pmatrix} $ for all $j$.

\begin{proposition}\label{polydecom AMZV}
1) Let $ \begin{pmatrix}
 \fve  \\
\fs  \end{pmatrix}  =  \begin{pmatrix}
 \varepsilon_1 & \dots & \varepsilon_n \\
s_1 & \dots & s_n \end{pmatrix} $ be an array such that $\Init(\fs) = (s_1, \dots, s_{k-1})$ for some $1 \leq k \leq n$.  Then $\zeta_A \begin{pmatrix}
 \fve  \\
\fs  \end{pmatrix} $ can be decomposed as follows:
    \begin{equation*}
    \zeta_A \begin{pmatrix}
 \fve  \\
\fs  \end{pmatrix}
    = \underbrace{\sum\limits_i a_i \zeta_A \begin{pmatrix}
 \fve'_i  \\
\fs'_i  \end{pmatrix} }_\text{type 1} + \underbrace{\sum\limits_i b_i\zeta_A \begin{pmatrix}
 \fe_i'  \\
\mathfrak{t}'_i  \end{pmatrix} }_\text{type 2} + \underbrace{\sum\limits_i c_i\zeta_A \begin{pmatrix}
 \fm_i  \\
\mathfrak{u}_i  \end{pmatrix} }_\text{type 3}  ,
    \end{equation*}
    where $a_i, b_i, c_i \in K$ such that for all $i$, the following properties are satisfied:
    \begin{itemize}
        \item For all arrays $ \begin{pmatrix}
 \fe  \\
\mathfrak{t}  \end{pmatrix} $ appearing on the right hand side,
        \begin{equation*}
            \depth(\mathfrak{t}) \geq \depth(\fs) \quad \text{and} \quad T_k(\mathfrak{t}) \leq T_k(\fs).
     \end{equation*}
     \item For the array $ \begin{pmatrix}
 \fve'  \\
\fs'  \end{pmatrix} $ of type $1$ with respect to $ \begin{pmatrix}
 \fve  \\
\fs  \end{pmatrix} $, we have $\Init(\fs) \preceq \Init(\fs')$ and $s'_k < s_k$.

\item For the array $ \begin{pmatrix}
 \fe'  \\
\mathfrak{t}'  \end{pmatrix} $ of type $2$ with respect to $ \begin{pmatrix}
 \fve  \\
\fs  \end{pmatrix} $, for all $k \leq \ell \leq n$,
        \begin{equation*}
t'_{1} +  \dots + t'_\ell < s_1 +  \dots + s_\ell.
     \end{equation*}

        \item For the array $ \begin{pmatrix}
 \fm  \\
\mathfrak{u}  \end{pmatrix} $ of type $3$ with respect to $ \begin{pmatrix}
 \fve  \\
\fs  \end{pmatrix} $, we have $\Init(\fs) \prec\Init(\mathfrak{u})$.
    \end{itemize}

\noindent 2) Let $ \begin{pmatrix}
 \fve  \\
\fs  \end{pmatrix}  =  \begin{pmatrix}
 \varepsilon_1 & \dots & \varepsilon_k \\
s_1 & \dots & s_k \end{pmatrix} $ be an array such that $\Init(\fs) = \fs$ and $s_k = q$. Then $\zeta_A \begin{pmatrix}
 \fve  \\
\fs  \end{pmatrix} $ can be decomposed as follows:
    \begin{equation*}
    \zeta_A \begin{pmatrix}
 \fve  \\
\fs  \end{pmatrix}
    =   \underbrace{\sum\limits_i b_i\zeta_A \begin{pmatrix}
 \fe'_i  \\
\mathfrak{t}'_i  \end{pmatrix} }_\text{type 2} + \underbrace{\sum\limits_i c_i\zeta_A \begin{pmatrix}
 \fm_i  \\
\mathfrak{u}_i  \end{pmatrix} }_\text{type 3}  ,
    \end{equation*}
    where $ b_i, c_i \in K$ such that for all $i$, the following properties are satisfied:
    \begin{itemize}
        \item For all arrays $ \begin{pmatrix}
 \fe  \\
\mathfrak{t}  \end{pmatrix} $ appearing on the right hand side,
        \begin{equation*}
            \depth(\mathfrak{t}) \geq \depth(\fs) \quad \text{and} \quad T_k(\mathfrak{t}) \leq T_k(\fs).
     \end{equation*}

\item For the array $ \begin{pmatrix}
 \fe'  \\
\mathfrak{t}'  \end{pmatrix} $ of type $2$ with respect to $ \begin{pmatrix}
 \fve  \\
\fs  \end{pmatrix} $,
        \begin{equation*}
t'_{1} +  \dots + t'_k < s_1 +  \dots + s_k.
     \end{equation*}

        \item For the array $ \begin{pmatrix}
 \fm  \\
\mathfrak{u}  \end{pmatrix} $ of type $3$ with respect to $ \begin{pmatrix}
 \fve  \\
\fs  \end{pmatrix} $, we have $\Init(\fs) \prec\Init(\mathfrak{u})$.
\end{itemize}
\end{proposition}

\begin{proof}
For Part 1, since $\Init(\fs) = (s_1, \dots, s_{k-1})$, we get $s_k > q$. Set $ \begin{pmatrix}
\Sigma  \\
V  \end{pmatrix}  =  \begin{pmatrix}
1 & \varepsilon_{k+1} &\dots & \varepsilon_n \\
s_k - q & s_{k+1} &\dots & s_n \end{pmatrix} $. By Proposition \ref{polycer1 AMZV}, $\mathcal{C}_{\Sigma,V}(R_{\varepsilon_k})$ is of the form
\begin{equation} \label{polyr1 AMZV}
  S_d \begin{pmatrix}
 \varepsilon_{k} & \dots & \varepsilon_n \\
s_{k} & \dots & s_n \end{pmatrix}  +  \sum \limits_i a_i S_d \begin{pmatrix}
\fve_i  \\
\fs_i  \end{pmatrix}  + \sum \limits_i b_i S_{d+1} \begin{pmatrix}
\varepsilon_k & \fe_i  \\
1 & \mathfrak{t}_i  \end{pmatrix}  =0,
\end{equation}
where $a_i, b_i \in K$ and $\begin{pmatrix}
\fve_i  \\
\fs_i  \end{pmatrix},  \begin{pmatrix}
 \fe_i  \\
\mathfrak{t}_i  \end{pmatrix} $ are arrays satisfying
\begin{align} \label{eq: part 1 compa 0 AMZV}
    \begin{pmatrix}
\fve_i  \\
\fs_i  \end{pmatrix} &\leq \begin{pmatrix}
\varepsilon_k  \\
q \end{pmatrix} + \begin{pmatrix}
\Sigma \\
V  \end{pmatrix} = \begin{pmatrix}
 \varepsilon_{k} & \dots & \varepsilon_n \\
s_{k} & \dots & s_n \end{pmatrix} \quad \text{and} \quad s_{i1} < q + v_1 = s_k;\\ \notag
\begin{pmatrix}
 \fe_i  \\
\mathfrak{t}_i  \end{pmatrix}  &\leq  \begin{pmatrix}
 1  \\
q - 1 \end{pmatrix}  +  \begin{pmatrix}
 \Sigma  \\
V  \end{pmatrix} = \begin{pmatrix}
1 & \varepsilon_{k + 1} & \dots & \varepsilon_n \\
s_k - 1 & s_{k + 1} & \dots & s_n \end{pmatrix}.
\end{align}

For $m \in \mathbb{N}$, we recall that $q^{\{m\}}$ is the sequence of length $m$ with all terms equal to $q$. Setting $s_0 = 0$, we may assume that there exists a maximal index $j$ with $0 \leq j \leq k-1$ such that $s_j < q$, hence $\Init(\fs) = (s_1, \dots, s_j, q^{\{k-j-1\}})$.

Then the operator $ \mathcal{BC}_{\varepsilon_{j+1},q} \circ \dots \circ \mathcal{BC}_{\varepsilon_{k-1},q}$ applied to the relation \eqref{polyr1 AMZV} gives
\begin{align} \label{polyr1 AMZV 2}
 S_d \begin{pmatrix}
 \varepsilon_{j+1} & \dots & \varepsilon_{k-1} & \varepsilon_{k} & \dots & \varepsilon_n \\
q & \dots & q & s_{k} & \dots & s_n \end{pmatrix}  & +  \sum \limits_i a_i S_d \begin{pmatrix}
\varepsilon_{j+1} & \dots & \varepsilon_{k-1} & \fve_i \\
q & \dots & q & \fs_i  \end{pmatrix}  \\ \notag
&
+ \sum \limits_i b_{i_1 \dots i_{k-j}} S_{d+1} \begin{pmatrix}
\varepsilon_{j+1} & \fe_{i_1 \dots i_{k-j}}  \\
1 & \mathfrak{t}_{i_1 \dots i_{k-j}}  \end{pmatrix}  =0,
\end{align}
where $b_{i_1 \dots i_{k-j}} \in K$ and for $2 \leq l \leq k-j$, $ \begin{pmatrix}
 \fe_{i_1 \dots i_l}  \\
 \mathfrak{t}_{i_1 \dots i_l}  \end{pmatrix} $ are arrays satisfying
\begin{equation*}
     \begin{pmatrix}
 \fe_{i_1 \dots i_l}  \\
 \mathfrak{t}_{i_1 \dots i_l}  \end{pmatrix}  \leq  \begin{pmatrix}
 1  \\
q - 1 \end{pmatrix}  +  \begin{pmatrix}
\varepsilon_{k-l+2} & \fe_{i_1 \dots i_{l-1}}  \\
1 & \mathfrak{t}_{i_1 \dots i_{l-1}}  \end{pmatrix}  =  \begin{pmatrix}
\varepsilon_{k-l+2} & \fe_{i_1 \dots i_{l-1}}  \\
q & \mathfrak{t}_{i_1 \dots i_{l-1}}  \end{pmatrix} .
\end{equation*}
Thus 
\begin{equation} \label{eq: part 1 compa AMZV}
     \begin{pmatrix}
 \fe_{i_1 \dots i_{k-j}}  \\
 \mathfrak{t}_{i_1 \dots i_{k-j}}  \end{pmatrix}  \leq  \begin{pmatrix}
 \varepsilon_{j+2}& \dots & \varepsilon_k  & 1 & \varepsilon_{k+1}& \dots & \varepsilon_n \\
q& \dots & q & s_{k} - 1 & s_{k+1} &\dots & s_n \end{pmatrix} .
\end{equation}

We let $ \begin{pmatrix}
\Sigma'  \\
V'  \end{pmatrix}  =  \begin{pmatrix}
 \varepsilon_{1} &\dots & \varepsilon_j \\
 s_1 &\dots & s_j \end{pmatrix} $ and we apply $\mathcal{B}^*_{\Sigma',V'}$ to \eqref{polyr1 AMZV 2}. Since $s_j < q$, i.e., $s_j + 1 \leq q$, it follows from Lemma \ref{polybesao AMZV} that 
\begin{align*}
  S_d \begin{pmatrix}
\fve \\ \fs \end{pmatrix}  &+  \sum \limits_i a_i S_d \begin{pmatrix}
\varepsilon_1 & \dots & \varepsilon_j & \varepsilon_{j+1} & \dots & \varepsilon_{k-1} & \fve_i \\
s_1 & \dots & s_j & q & \dots & q & \fs_i  \end{pmatrix}  \\
& + \sum \limits_i b_{i_1 \dots i_{k-j}} S_d \begin{pmatrix}
\varepsilon_1 & \dots & \varepsilon_j & \varepsilon_{j+1} & \fe_{i_1 \dots i_{k-j}}  \\
s_1 & \dots & s_j & 1 & \mathfrak{t}_{i_1 \dots i_{k-j}}  \end{pmatrix}  \\
& + \sum \limits_i b_{i_1 \dots i_{k-j}} S_d \begin{pmatrix}
\varepsilon_1 & \dots & \varepsilon_{j-1} & \varepsilon_j \varepsilon_{j+1} & \fe_{i_1 \dots i_{k-j}}  \\
s_1 & \dots & s_{j-1} & s_j+1 & \mathfrak{t}_{i_1 \dots i_{k-j}}  \end{pmatrix}  =0.
\end{align*}
Hence
\begin{align*}
  \zeta_A \begin{pmatrix}
\fve \\ \fs \end{pmatrix}  &+  \sum \limits_i a_i \zeta_A \begin{pmatrix}
\varepsilon_1 & \dots & \varepsilon_j & \varepsilon_{j+1} & \dots & \varepsilon_{k-1} & \fve_i \\
s_1 & \dots & s_j & q & \dots & q & \fs_i  \end{pmatrix}  \\
& + \sum \limits_i b_{i_1 \dots i_{k-j}} \zeta_A \begin{pmatrix}
\varepsilon_1 & \dots & \varepsilon_j & \varepsilon_{j+1} & \fe_{i_1 \dots i_{k-j}}  \\
s_1 & \dots & s_j & 1 & \mathfrak{t}_{i_1 \dots i_{k-j}}  \end{pmatrix}  \\
& + \sum \limits_i b_{i_1 \dots i_{k-j}} \zeta_A \begin{pmatrix}
\varepsilon_1 & \dots & \varepsilon_{j-1} & \varepsilon_j \varepsilon_{j+1} & \fe_{i_1 \dots i_{k-j}}  \\
s_1 & \dots & s_{j-1} & s_j+1 & \mathfrak{t}_{i_1 \dots i_{k-j}}  \end{pmatrix}  =0.
\end{align*}
In other words, we have
\begin{align} \label{eq: part 1 Li AMZV}
  \zeta_A \begin{pmatrix}
\fve \\ \fs \end{pmatrix}  = &-  \sum \limits_i a_i \zeta_A \begin{pmatrix}
\varepsilon_1 & \dots & \varepsilon_j & \varepsilon_{j+1} & \dots & \varepsilon_{k-1} & \fve_i \\
s_1 & \dots & s_j & q & \dots & q & \fs_i  \end{pmatrix}  \\ \notag
& - \sum \limits_i b_{i_1 \dots i_{k-j}} \zeta_A \begin{pmatrix}
\varepsilon_1 & \dots & \varepsilon_j & \varepsilon_{j+1} & \fe_{i_1 \dots i_{k-j}}  \\
s_1 & \dots & s_j & 1 & \mathfrak{t}_{i_1 \dots i_{k-j}}  \end{pmatrix}  \\ \notag
& - \sum \limits_i b_{i_1 \dots i_{k-j}} \zeta_A \begin{pmatrix}
\varepsilon_1 & \dots & \varepsilon_{j-1} & \varepsilon_j \varepsilon_{j+1} & \fe_{i_1 \dots i_{k-j}}  \\
s_1 & \dots & s_{j-1} & s_j+1 & \mathfrak{t}_{i_1 \dots i_{k-j}}  \end{pmatrix}.
\end{align}
The first term, the second term, and the third term on the right hand side of \eqref{eq: part 1 Li AMZV} are referred to as type 1, type 2, and type 3 respectively.

We now verify the conditions of arrays of type 1, type 2, and type 3 with respect to $ \begin{pmatrix}
\fve  \\
\fs  \end{pmatrix} $. We first note that $\fs = (s_1, \dots, s_j, q^{\{k-j-1\}}, s_k, \dotsc, s_n)$.

\textit{Type 1: } For $\begin{pmatrix}
 \fve'  \\
\fs'  \end{pmatrix} =
\begin{pmatrix}
\varepsilon_1 & \dots & \varepsilon_j & \varepsilon_{j+1} & \dots & \varepsilon_{k-1} & \fve_i \\
s_1 & \dots & s_j & q & \dots & q & \fs_i  \end{pmatrix}$,  it follows from \eqref{eq: part 1 compa 0 AMZV} and Remark \ref{rmk: compa depth} that 
\begin{align*}
    \depth(\fs') = j + (k - j - 1) + \depth(\fs_i) \geq j + (k - j - 1) + (n - k + 1) = \depth(\fs).
\end{align*}
For $k \leq \ell \leq n$, it follows from \eqref{eq: part 1 compa 0 AMZV} that 
\begin{align*}
    s'_1 + \cdots + s'_{\ell} \leq  s_1 + \cdots + s_{\ell}. 
\end{align*}
Since $w(\fs') = w(\fs)$, one deduces that $T_k(\fs') \leq T_k(\fs)$. Moreover, on verifies that $\Init(\fs) = (s_1, \dotsc, s_{k - 1}) \preceq \Init(\fs')$ and $s'_k = s_{i1} < s_k$ by \eqref{eq: part 1 compa 0 AMZV}.

\textit{Type 2:} For $ \begin{pmatrix}
 \fe'  \\
\mathfrak{t}'  \end{pmatrix} = \begin{pmatrix}
\varepsilon_1 & \dots & \varepsilon_j & \varepsilon_{j+1} & \fe_{i_1 \dots i_{k-j}}  \\
s_1 & \dots & s_j & 1 & \mathfrak{t}_{i_1 \dots i_{k-j}}  \end{pmatrix}$, it follows from \eqref{eq: part 1 compa AMZV} and Remark \ref{rmk: compa depth} that 
\begin{align*}
    \depth(\mathfrak{t}')  &= j + 1 + \depth(\mathfrak{t}_{i_1 \dots i_{k-j}})\\
    &\geq j + 1 + (k - j - 2) + 1 + (n - k + 1) = n + 1  > \depth(\fs).
\end{align*}
For $\ell = k$, since $s_k > q$, it follows from \eqref{eq: part 1 compa AMZV} that
\begin{align*}
    t'_1 + \cdots + t'_k \leq  s_1 + \cdots + s_j + 1 + q(k - j - 1) < s_1 + \cdots + s_{k}.
\end{align*}
For $k < \ell \leq n$, it follows from \eqref{eq: part 1 compa AMZV} that
\begin{align*}
    t'_1 + \cdots + t'_{\ell} \leq s_1 + \cdots + s_{\ell - 1} < s_1 + \cdots + s_{\ell}.
\end{align*}
Since $w(\mathfrak{t}') = w(\fs)$, one deduces that $T_k(\mathfrak{t}') \leq T_k(\fs)$.

\textit{Type 3:} For $ \begin{pmatrix}
 \fm  \\
\mathfrak{u}  \end{pmatrix} = \begin{pmatrix}
\varepsilon_1 & \dots & \varepsilon_{j-1} & \varepsilon_j \varepsilon_{j+1} & \fe_{i_1 \dots i_{k-j}}  \\
s_1 & \dots & s_{j-1} & s_j+1 & \mathfrak{t}_{i_1 \dots i_{k-j}}  \end{pmatrix}$, it follows from \eqref{eq: part 1 compa AMZV} and Remark \ref{rmk: compa depth} that 
\begin{align*}
    \depth(\mathfrak{u})  &= j + \depth(\mathfrak{t}_{i_1 \dots i_{k-j}}) \geq j + (k - j - 2) + 1 + (n - k + 1)  = \depth(\fs).
\end{align*}
For $k \leq \ell \leq n$, it follows from \eqref{eq: part 1 compa AMZV} that
\begin{align*}
    u_1 + \cdots + u_{\ell} \leq s_1 + \cdots + s_{\ell}.
\end{align*}
Since $w(\mathfrak{u}) = w(\fs)$, one deduces that $T_k(\mathfrak{u}) \leq T_k(\fs)$. Moreover, we have $\Init(\fs) \prec (s_1, \dotsc, s_{j - 1}, s_j + 1) \preceq \Init(\mathfrak{u})$. We have proved Part 1.

For Part $2$, we recall the relation $R_{\varepsilon_k}$ given by
\begin{equation*}
 R_{\varepsilon_k} \colon \quad  S_d \begin{pmatrix}
 \varepsilon_k\\
q  \end{pmatrix}  + \varepsilon_k^{-1}D_1 S_{d+1} \begin{pmatrix}
 \varepsilon_k& 1 \\
1 & q-1  \end{pmatrix}  =0.
\end{equation*}
Setting $s_0 = 0$, we may assume that there exists a maximal index $j$ with $0 \leq j \leq k-1$ such that $s_j < q$. We note that $s_k = q$, hence $\Init(\fs) = \fs = (s_1, \dots, s_j, q^{\{k-j-1\}}, q)$. Then  $ \mathcal{BC}_{\varepsilon_{j+1},q} \circ \dots \circ \mathcal{BC}_{\varepsilon_{k-1},q}(R_{\varepsilon_k})$ is of the form
\begin{align} \label{eq: Part 2 Li AMZV}
 S_d &\begin{pmatrix}
 \varepsilon_{j+1} & \dots & \varepsilon_{k} \\
q & \dots & q \end{pmatrix} + \sum \limits_i b_{i_1 \dots i_{k-j}} S_{d+1} \begin{pmatrix}
\varepsilon_{j+1} & \fe_{i_1 \dots i_{k-j}}  \\
1 & \mathfrak{t}_{i_1 \dots i_{k-j}}  \end{pmatrix}  =0,
\end{align}
where $b_{i_1 \dots i_{k-j}} \in K$ and for $2 \leq l \leq k-j$, $ \begin{pmatrix}
 \fe_{i_1 \dots i_l}  \\
 \mathfrak{t}_{i_1 \dots i_l}  \end{pmatrix} $ are arrays satisfying
\begin{equation*}
     \begin{pmatrix}
 \fe_{i_1 \dots i_l}  \\
 \mathfrak{t}_{i_1 \dots i_l}  \end{pmatrix}  \leq  \begin{pmatrix}
 1  \\
q - 1 \end{pmatrix}  +  \begin{pmatrix}
\varepsilon_{k-l+2} & \fe_{i_1 \dots i_{l-1}}  \\
1 & \mathfrak{t}_{i_1 \dots i_{l-1}}  \end{pmatrix}  =  \begin{pmatrix}
\varepsilon_{k-l+2} & \fe_{i_1 \dots i_{l-1}}  \\
q & \mathfrak{t}_{i_1 \dots i_{l-1}}  \end{pmatrix} .
\end{equation*}
Thus 
\begin{equation} \label{eq: part 2 compa AMZV}
     \begin{pmatrix}
 \fe_{i_1 \dots i_{k-j}}  \\
 \mathfrak{t}_{i_1 \dots i_{k-j}}  \end{pmatrix}  \leq  \begin{pmatrix}
 \varepsilon_{j+2}& \dots & \varepsilon_{k} & 1 \\
q& \dots & q &  q - 1 \end{pmatrix} .
\end{equation}

We let $ \begin{pmatrix}
\Sigma'  \\
V'  \end{pmatrix}  =  \begin{pmatrix}
 \varepsilon_{1} &\dots & \varepsilon_j \\
 s_1 &\dots & s_j \end{pmatrix} $ and we apply $\mathcal{B}^*_{\Sigma',V'}$ to \eqref{eq: Part 2 Li AMZV}. Since $s_j < q$, i.e., $s_j + 1 \leq q$, it follows from Lemma \ref{polybesao AMZV} that 
\begin{align*}
  S_d \begin{pmatrix}
\fve \\ \fs \end{pmatrix}  & + \sum \limits_i b_{i_1 \dots i_{k-j}} S_d \begin{pmatrix}
\varepsilon_1 & \dots & \varepsilon_j & \varepsilon_{j+1} & \fe_{i_1 \dots i_{k-j}}  \\
s_1 & \dots & s_j & 1 & \mathfrak{t}_{i_1 \dots i_{k-j}}  \end{pmatrix}  \\
& + \sum \limits_i b_{i_1 \dots i_{k-j}} S_d \begin{pmatrix}
\varepsilon_1 & \dots & \varepsilon_{j-1} & \varepsilon_j \varepsilon_{j+1} & \fe_{i_1 \dots i_{k-j}}  \\
s_1 & \dots & s_{j-1} & s_j+1 & \mathfrak{t}_{i_1 \dots i_{k-j}}  \end{pmatrix}  =0,
\end{align*}
hence 
\begin{align*}
  \zeta_A \begin{pmatrix}
\fve \\ \fs \end{pmatrix}  & + \sum \limits_i b_{i_1 \dots i_{k-j}} \zeta_A \begin{pmatrix}
\varepsilon_1 & \dots & \varepsilon_j & \varepsilon_{j+1} & \fe_{i_1 \dots i_{k-j}}  \\
s_1 & \dots & s_j & 1 & \mathfrak{t}_{i_1 \dots i_{k-j}}  \end{pmatrix}  \\
& + \sum \limits_i b_{i_1 \dots i_{k-j}} \zeta_A \begin{pmatrix}
\varepsilon_1 & \dots & \varepsilon_{j-1} & \varepsilon_j \varepsilon_{j+1} & \fe_{i_1 \dots i_{k-j}}  \\
s_1 & \dots & s_{j-1} & s_j+1 & \mathfrak{t}_{i_1 \dots i_{k-j}}  \end{pmatrix}  =0.
\end{align*}
In other words, we have
\begin{align} \label{eq: part 2 Li 2 AMZV}
  \zeta_A \begin{pmatrix}
\fve \\ \fs \end{pmatrix}  =& - \sum \limits_i b_{i_1 \dots i_{k-j}} \zeta_A \begin{pmatrix}
\varepsilon_1 & \dots & \varepsilon_j & \varepsilon_{j+1} & \fe_{i_1 \dots i_{k-j}}  \\
s_1 & \dots & s_j & 1 & \mathfrak{t}_{i_1 \dots i_{k-j}}  \end{pmatrix}  \\ \notag
& - \sum \limits_i b_{i_1 \dots i_{k-j}} \zeta_A \begin{pmatrix}
\varepsilon_1 & \dots & \varepsilon_{j-1} & \varepsilon_j \varepsilon_{j+1} & \fe_{i_1 \dots i_{k-j}}  \\
s_1 & \dots & s_{j-1} & s_j+1 & \mathfrak{t}_{i_1 \dots i_{k-j}}  \end{pmatrix}.
\end{align}
The first term and the second term on the right hand side of \eqref{eq: part 2 Li 2 AMZV} are referred to as type 2 and type 3 respectively.

We now verify the conditions of arrays of type 2 and type 3 with respect to $ \begin{pmatrix}
\fve  \\
\fs  \end{pmatrix} $.

\textit{Type 2:} For $ \begin{pmatrix}
 \fe'  \\
\mathfrak{t}'  \end{pmatrix} = \begin{pmatrix}
\varepsilon_1 & \dots & \varepsilon_j & \varepsilon_{j+1} & \fe_{i_1 \dots i_{k-j}}  \\
s_1 & \dots & s_j & 1 & \mathfrak{t}_{i_1 \dots i_{k-j}}  \end{pmatrix}$, it follows from \eqref{eq: part 2 compa AMZV} and Remark \ref{rmk: compa depth} that 
\begin{align*}
    \depth(\mathfrak{t}')  = j + 1 + \depth(\mathfrak{t}_{i_1 \dots i_{k-j}}) \geq j + 1 + (k - j)= k + 1  > \depth(\fs).
\end{align*}
It follows from \eqref{eq: part 2 compa AMZV} that
\begin{align*}
    t'_1 + \cdots + t'_k \leq  s_1 + \cdots + s_j + 1 + q(k - j - 1) < s_1 + \cdots + s_{k}.
\end{align*}
Since $w(\mathfrak{t}') = w(\fs)$, one deduces that $T_k(\mathfrak{t}') \leq T_k(\fs)$.

\textit{Type 3:} For $ \begin{pmatrix}
 \fm  \\
\mathfrak{u}  \end{pmatrix} = \begin{pmatrix}
\varepsilon_1 & \dots & \varepsilon_{j-1} & \varepsilon_j \varepsilon_{j+1} & \fe_{i_1 \dots i_{k-j}}  \\
s_1 & \dots & s_{j-1} & s_j+1 & \mathfrak{t}_{i_1 \dots i_{k-j}}  \end{pmatrix}$, it follows from \eqref{eq: part 2 compa AMZV} and Remark \ref{rmk: compa depth} that 
\begin{align*}
    \depth(\mathfrak{u})  = j + \depth(\mathfrak{t}_{i_1 \dots i_{k-j}}) \geq j + (k - j) = \depth(\fs).
\end{align*}
It follows from \eqref{eq: part 2 compa AMZV} that
\begin{align*}
    u_1 + \cdots + u_k \leq s_1 + \cdots + s_{j - 1} + (s_{j} + 1) + q(k - j - 1) + (q - 1) = s_1 + \cdots + s_k.
\end{align*}
Since $w(\mathfrak{u}) = w(\fs)$, one deduces that $T_k(\mathfrak{u}) \leq T_k(\fs)$. Moreover, we have $\Init(\fs) \prec (s_1, \dotsc, s_{j - 1}, s_j + 1) \preceq \Init(\mathfrak{u})$. We finish the proof.
\end{proof}

\begin{proposition} \label{polyalgpartlem AMZV}
For all $k \in \mathbb{N}$ and for all arrays $ \begin{pmatrix}
 \fve  \\
\fs  \end{pmatrix} $, $\zeta_A \begin{pmatrix}
 \fve  \\
\fs  \end{pmatrix} $ can be expressed as a $K$-linear combination of $\zeta_A \begin{pmatrix}
 \fe  \\
\mathfrak{t}  \end{pmatrix} $'s of the same weight such that $\mathfrak{t}$ is $k$-admissible.
\end{proposition}

\begin{proof}
The proof follows the same line as that of \cite[Proposition 3.2]{ND21}. We write down the proof for the reader's convenience.

We consider the following statement: 

$(H_k) \quad$ For all arrays $ \begin{pmatrix}
 \fve  \\
\fs  \end{pmatrix} $, we can express $\zeta_A \begin{pmatrix}
 \fve  \\ 
\fs  \end{pmatrix} $  as a $K$-linear combination of $\zeta_A \begin{pmatrix}
 \fe  \\
\mathfrak{t}  \end{pmatrix} $'s of the same weight such that $\mathfrak{t}$ is $k$-admissible.\\

We need to show that $(H_k)$ holds for all $k \in \mathbb{N}$. We proceed the proof by induction on $k$. For $k = 1$, we prove that $(H_1)$ holds by induction on the first component $s_1$ of $\fs$. If $s_1 \leq q$, then either $\fs$ is  $1$-admissible,  or $\begin{pmatrix}
 \fve  \\ 
\fs  \end{pmatrix} = \begin{pmatrix}
 \varepsilon \\ 
q \end{pmatrix}$. For the case $\begin{pmatrix}
 \fve  \\ 
\fs  \end{pmatrix} = \begin{pmatrix}
 \varepsilon \\ 
q \end{pmatrix}$, we deduce from the relation $R_{\varepsilon}$ that 
\begin{equation*}
\zeta_A \begin{pmatrix}
 \varepsilon\\
q  \end{pmatrix}  = - \varepsilon^{-1}D_1 \zeta_A \begin{pmatrix}
 \varepsilon& 1 \\
1 & q-1  \end{pmatrix},
\end{equation*}
hence $(H_1)$ holds in this case since $(1, q - 1)$ is $1$-admissible. If $s_1 > q$, we assume that $(H_1)$ holds for the array $\begin{pmatrix}
 \fve  \\ 
\fs  \end{pmatrix}$ where $s_1 < s$. We need to shows that $(H_1)$ holds for the array $\begin{pmatrix}
 \fve  \\ 
\fs  \end{pmatrix}$ where $s_1 = s$. Indeed, assume that  $ \begin{pmatrix}
 \fve  \\ 
\fs  \end{pmatrix} = \begin{pmatrix}
 \varepsilon_1 & \cdots & \varepsilon_n  \\ 
s_1 & \cdots & s_n  \end{pmatrix}$. Set $\begin{pmatrix}
 \Sigma \\ 
V  \end{pmatrix} = \begin{pmatrix}
 \varepsilon_1 & \varepsilon_2 & \cdots & \varepsilon_n  \\ 
s_1 - q & s_2 & \cdots & s_n  \end{pmatrix}$. From Proposition \ref{polycer1 AMZV}, we have $\mathcal C_{\Sigma,V}(R_{1})$ is of the form
\begin{equation*}
  S_d \begin{pmatrix}
 \varepsilon_1 & \cdots & \varepsilon_n \\
s_1 & \cdots & s_n \end{pmatrix}  +  \sum \limits_i a_i S_d \begin{pmatrix}
  \fve_i  \\
\fs_i \end{pmatrix}  + \sum \limits_i b_i S_{d+1} \begin{pmatrix}
1 & \fe_i  \\
1 & \mathfrak{t}_i  \end{pmatrix}  =0,
\end{equation*}
where $a_i, b_i \in K$ and $\begin{pmatrix}
  \fve_i  \\
\fs_i \end{pmatrix}$ are arrays satisfying $s_{i1} < q + v_1 = s_1 = s$ for all $i$. Thus we deduce that 
\begin{align*}
     \zeta_A \begin{pmatrix}
 \fve  \\ 
\fs  \end{pmatrix} =  -  \sum \limits_i a_i \zeta_A \begin{pmatrix}
  \fve_i  \\
\fs_i \end{pmatrix}  -  \sum \limits_i b_i \zeta_A \begin{pmatrix}
1 & \fe_i  \\
1 & \mathfrak{t}_i  \end{pmatrix}.
\end{align*}
It is clear that $(1, \mathfrak{t}_i)$ is $1$-admissible. Moreover, for all $i$, since $s_{i1} < s$, we deduce from the induction hypothesis that $(H_1)$ holds for $\begin{pmatrix}
 \fve_i  \\ 
\fs_i  \end{pmatrix}$, which proves that $(H_1)$ holds for $\begin{pmatrix}
 \fve  \\ 
\fs  \end{pmatrix}$. 

We next assume that $(H_{k - 1})$ holds. We need to show that $(H_k)$ holds. It suffices to consider the array $\begin{pmatrix}
 \fve  \\ 
\fs  \end{pmatrix}$ where $\fs$ is not $k$-admissible. Moreover, from the induction hypothesis of $(H_{k - 1})$, we may assume that $\fs$ is $(k - 1)$-admissible. For such array $\begin{pmatrix}
 \fve  \\ 
\fs  \end{pmatrix}$, we claim that $\depth(\fs) \geq k$. Indeed, assume that $\depth(\fs) < k$. Since $\fs$ is $(k - 1)$-admissible, one verifies that $\fs$ is $k$-admissible, which is a contradiction.

Assume that $ \begin{pmatrix}
 \fve  \\ 
\fs  \end{pmatrix} = \begin{pmatrix}
 \varepsilon_1 & \cdots & \varepsilon_n  \\ 
s_1 & \cdots & s_n  \end{pmatrix}$ where $\fs$ is not $k$-admissible and $\depth(\fs) \geq k$. In order to prove that $(H_k)$ holds for the array $\begin{pmatrix}
 \fve  \\ 
\fs  \end{pmatrix}$, we proceed by induction on $s_1 + \cdots + s_k$. The case $s_1 + \cdots + s_k = 1$ is clear.  Assume that $(H_k)$ holds when $s_1 + \cdots + s_k < s$. We need to show that $(H_k)$ holds when $s_1 + \cdots + s_k = s$. To do so, we give the following algorithm:

\textit{Algorithm: } We begin with an array $\begin{pmatrix}
 \fve  \\ 
\fs  \end{pmatrix}$ where $\fs$ is not $k$-admissible, $\depth(\fs) \geq k$ and $s_1 + \cdots + s_k = s$.

\textit{Step 1:} From Proposition \ref{polydecom AMZV}, we can decompose $\zeta_A \begin{pmatrix}
 \fve  \\
 \fs  \end{pmatrix} $ as follows:
    \begin{equation} \label{eq: polyalgpartlem AMZV}
    \zeta_A \begin{pmatrix}
 \fve  \\
\fs  \end{pmatrix}
    = \underbrace{\sum\limits_i a_i \zeta_A \begin{pmatrix}
 \fve'_i  \\
\fs'_i  \end{pmatrix} }_\text{type 1} + \underbrace{\sum\limits_i b_i\zeta_A \begin{pmatrix}
 \fe_i'  \\
\mathfrak{t}'_i  \end{pmatrix} }_\text{type 2} + \underbrace{\sum\limits_i c_i\zeta_A \begin{pmatrix}
 \fm_i  \\
\mathfrak{u}_i  \end{pmatrix} }_\text{type 3}  ,
    \end{equation}
    where $a_i,  b_i, c_i \in K$. Note that the term of type $1$ disappears when $\Init(\fs) = \fs$ and $s_n = q$. Moreover, for all arrays $ \begin{pmatrix}
 \fe  \\
\mathfrak{t}  \end{pmatrix} $ appearing on the right hand side of \eqref{eq: polyalgpartlem AMZV}, we have $\depth(\mathfrak{t}) \geq \depth(\fs) \geq k$ and $t_1 + \cdots + t_k \leq s_1 + \cdots + s_k = s$.

\textit{Step 2:} For all arrays $ \begin{pmatrix}
 \fe  \\
\mathfrak{t}  \end{pmatrix} $ appearing on the right hand side of \eqref{eq: polyalgpartlem AMZV}, if $\mathfrak{t}$ is either $k$-admissible or $\mathfrak{t}$ satisfies the condition $t_1 + \cdots + t_k < s$, then we deduce from the induction hypothesis that $(H_k)$ holds for the array $\begin{pmatrix}
 \fve  \\
\fs  \end{pmatrix}$, and hence we stop the algorithm. Otherwise, there exists an array $\begin{pmatrix}
 \fve_1  \\
\fs_1  \end{pmatrix}$ where $\fs_1$ is not $k$-admissible, $\depth(\fs_1) \geq k$ and $s_{11} + \cdots + s_{1k} = s$. For such an array, we repeat the algorithm for $\begin{pmatrix}
 \fve_1  \\
\fs_1  \end{pmatrix}$. It should be remarked that the array $\begin{pmatrix}
 \fve_1  \\
\fs_1  \end{pmatrix}$ is of type $1$ or type $3$ with respect to $\begin{pmatrix}
 \fve  \\
\fs  \end{pmatrix}$. Indeed, if the array $\begin{pmatrix}
 \fve_1  \\
\fs_1  \end{pmatrix}$ is of type $2$, then we deduce from Proposition \ref{polydecom AMZV} that $s_{11} + \cdots + s_{1k} < s_{1} + \cdots + s_{k} = s$, which is a contradiction.

It remains to show that the above algorithm stops after a finite number of steps. Set $\begin{pmatrix}
 \fve_0 \\
\fs_0  \end{pmatrix} = \begin{pmatrix}
 \fve  \\
\fs  \end{pmatrix}$. Assume that the above algorithm does not stop. Then there exists a sequence of arrays $\begin{pmatrix}
 \fve_0 \\
\fs_0  \end{pmatrix}, \begin{pmatrix}
 \fve_1  \\
\fs_1  \end{pmatrix}, \begin{pmatrix}
 \fve_2  \\
\fs_2  \end{pmatrix}, \dotsc$ such that for all $i \geq 0$, $\fs_i$ is not $k$-admissible, $\depth(\fs_i) \geq k$ and $\begin{pmatrix}
 \fve_{i + 1}  \\
\fs_{i + 1}  \end{pmatrix}$ is of type $1$ or type $3$ with respect to $\begin{pmatrix}
 \fve_i  \\
\fs_i  \end{pmatrix}$. From Proposition \ref{polydecom AMZV}, we obtain an infinite sequence
\begin{align*}
    \Init(\fs_0) \preceq \Init(\fs_1) \preceq \Init(\fs_2) \preceq \cdots
\end{align*}
which is increasing. For all $i \geq 0$, since $\fs_i$ is not $k$-admissible and $\depth(\fs_i) \geq k$, we have $\depth(\Init(\fs_i)) \leq k$, hence $\Init(\fs_i) \preceq q^{\{k\}}$. Thus there exists an index $i_0$ sufficiently large such that $\Init(\fs_i) = \Init(\fs_{i + 1})$ for all $i \geq i_0$. It follows from Proposition \ref{polydecom AMZV} that $\begin{pmatrix}
 \fve_{i + 1}  \\
\fs_{i + 1}  \end{pmatrix}$ is of type $1$ with respect to $\begin{pmatrix}
 \fve_i  \\
\fs_i  \end{pmatrix}$ for all $i \geq i_0$. Moreover, if we set $\ell = \depth(\Init(\fs_i)) + 1$, then $s_{(i + 1)\ell} < s_{i\ell}$ for all $i \geq i_0$, which is a contradiction. We deduce that the algorithm stops after a finite number of steps. 
\end{proof}

Consequently, we obtain a weak version of Brown's theorem for AMZV's as follows.
\begin{proposition} \label{prop: weak Brown AMZV}
The set of all elements $\zeta_A \begin{pmatrix}
 \fve  \\
\fs  \end{pmatrix}$ such that $\zeta_A \begin{pmatrix}
 \fve  \\
\fs  \end{pmatrix}  \in \mathcal{AT}_w$ forms a set of generators for $\mathcal{AZ}_w$. Here we recall that $\mathcal{AT}_w$ is the set of all AMZV's $\zeta_A \begin{pmatrix}
 \fve  \\
\fs  \end{pmatrix}  =  \Li \begin{pmatrix}
 \fve  \\
\fs  \end{pmatrix}$ of weight $w$ such that $ \begin{pmatrix}
 \fve  \\ 
\fs  \end{pmatrix} = \begin{pmatrix}
 \varepsilon_1 & \cdots & \varepsilon_n  \\ 
s_1 & \cdots & s_n  \end{pmatrix}$ with $s_1, \dots, s_{n-1} \leq q$ and $s_n < q$ introduced in the paragraph preceding Proposition \ref{prop: weak Brown}.
\end{proposition}

\begin{proof}
It follows from Proposition \ref{polyalgpartlem AMZV} in the case of $k = w$.
\end{proof}

\subsection{Proof of Theorem \ref{thm: ZagierHoffman AMZV}} \label{sec:proof Theorem A} 

As a direct consequence of Proposition \ref{prop: weak Brown} and Proposition \ref{prop: weak Brown AMZV}, we get

\begin{theorem} \label{thm:bridge}
The $K$-vector space $\mathcal{AZ}_w$ of AMZV's of weight $w$ and the $K$-vector space $\mathcal{AL}_w$ of ACMPL's of weight $w$ are the same.
\end{theorem}

By this identification we apply Theorem \ref{thm:ACMPL} to obtain Theorem \ref{thm: ZagierHoffman AMZV}.

\subsection{Zagier-Hoffman's conjectures in positive characteristic} \label{sec: application MZV}

\subsubsection{Known results}

We use freely the notation introduced in \S \ref{sec:ZagierHoffman}. We recall that for $w \in \N$, $\mathcal Z_w$ denotes the $K$-vector space spanned by the MZV's of weight~$w$ and $\mathcal T_w$ denotes the set of $\zeta_A(\fs)$ where $\fs=(s_1,\ldots,s_r) \in \N^r$ of weight $w$ with $1\leq s_i\leq q$ for $1\leq i\leq r-1$ and $s_r<q$.

Recall that the main results of \cite{ND21} state that
\begin{itemize}
\item For all $w \in \N$ we always have $\dim_K \mathcal Z_w \leq d(w)$  (see \cite[Theorem A]{ND21}).

\item For $w \leq 2q-2$ we have $\dim_K \mathcal Z_w \geq d(w)$  (see \cite[Theorem B]{ND21}). In particular, Conjecture \ref{conj: basis} holds for $w \leq 2q-2$ (see \cite[Theorem D]{ND21}).
\end{itemize}
However, as stated in \cite[Remark 6.3]{ND21} it would be very difficult to extend the method of \cite{ND21} for general weights.

As an application of our main results, we present a proof of Theorem \ref{thm: ZagierHoffman} which settles both Conjectures \ref{conj: dimension} and \ref{conj: basis}.

\subsubsection{Proof of Theorem \ref{thm: ZagierHoffman}} \label{sec:proof Theorem B}
As we have already known the sharp upper bound for $\mathcal Z_w$ (see \cite[Theorem A]{ND21}), Theorem \ref{thm: ZagierHoffman} follows immediately from the following proposition.

\begin{proposition} \label{prop:lower bound}
For all $w \in \N$ we have $\dim_K \mathcal Z_w \geq d(w)$.
\end{proposition}

\begin{proof}
We denote by $\mathcal S_w$ the set of CMPL's consisting of $\Li(s_1,\ldots,s_r)$ of weight $w$ with  $q \nmid s_i$ for all $i$. 

Then $\mathcal{S}_w$ can be considered as a subset of $\mathcal{AS}_w$ by assuming $\fe=(1,\ldots, 1)$. In fact, all algebraic relations in \S \ref{sec:algebraic part} hold for CMPL version, i.e., for $\Si_d(s_1,\dots,s_r)=\Si_d \begin{pmatrix}1 & \dots &1 \\ s_1& \dots & s_r\end{pmatrix}$ and $\Li(s_1,\dots,s_r)=\Li \begin{pmatrix}1 & \dots &1 \\ s_1& \dots & s_r\end{pmatrix}$. It follows that $\mathcal S_w$ is contained in $\mathcal Z_w$ by Theorem \ref{thm:bridge}. Further, by \S \ref{sec:AS_w}, $|\mathcal S_w|=d(w)$. By Theorem \ref{thm: trans ACMPL} we deduce that elements in $\mathcal S_w$ are all linearly independent over $K$. Therefore, $\dim_K \mathcal Z_w \geq |\mathcal S_w|=d(w)$.
\end{proof}

\subsection{Sharp bounds without ACMPL's} \label{sec:without ACMPL} 

To end this paper we mention that without ACMPL's it seems very hard to obtain for arbitrary weight $w$
\begin{itemize}
\item either the sharp upper bound $\dim_K \mathcal{AZ}_w \leq s(w)$,
\item or the sharp lower bound $\dim_K \mathcal{AZ}_w \geq s(w)$.
\end{itemize}

We can only do this for small weights with ad hoc arguments. We collect the results below, sketch some ideas for the proofs, and refer the reader to \cite{IKLNDP23} for full details.

\begin{proposition} \label{prop:Brown adhoc}
Let $w \leq 2q-2$. Then $\dim_K \mathcal{AZ}_w \leq s(w)$.
\end{proposition}

\begin{proof}
We denote by $\mathcal{AT}_w^1$ the subset of AMZV's $\zeta_A \begin{pmatrix}
	\fe \\ \fs
	\end{pmatrix}$ of $\mathcal{AT}_w$ such that $\epsilon_i=1$ whenever $s_i=q$ (i.e., the character corresponding to $q$ is always 1) and by $\langle \mathcal{AT}_w^1 \rangle$ the $K$-vector space spanned by the AMZV's in $\mathcal{AT}_w^1$. We see that $|\mathcal{AT}_w^1|=s(w)$. Thus it suffices to prove that $\langle \mathcal{AT}_w^1 \rangle=\mathcal{AZ}_w$.

Recall that for any $\epsilon \in \Fq^\times$, we recall the relation
\begin{equation*}
R_{\varepsilon} \colon \quad  S_d \begin{pmatrix}
 \varepsilon\\
q  \end{pmatrix}  + \varepsilon^{-1}D_1 S_{d+1} \begin{pmatrix}
 \varepsilon& 1 \\
1 & q-1  \end{pmatrix}  =0,
\end{equation*}
where $D_1 = \theta^q - \theta$.

We recall that the coefficients $\Delta^i_{s,t}$ are given in \eqref{eq:Delta Chen}. Let $U=(u_1,\dots,u_n)$ and $W=(w_1,\dots,w_r)$ be tuples of positive integers such that $u_n \leq q-1$ and $w_1,\dots,w_r \leq q$. Let $\fe=(\epsilon_1,\dots,\epsilon_n)\in(\Fq^\times)^n$ and $\fla=(\lambda_1,\dots,\lambda_r)\in (\Fq^\times)^r$. By direct calculations, we show that $\mathcal B_{\fe,U} \mathcal C_{\fla,W} (R_{\epsilon})$ can be written as
\begin{align*}
&S_d\begin{pmatrix}
\fe & \epsilon & \fla\\
U & q & W
\end{pmatrix}+S_d\begin{pmatrix}
\fe & \epsilon \lambda_1 & \fla_-\\
U & q + w_1 & W_-
\end{pmatrix}\\
& +\Delta^{q-1}_{q,w_1}S_d\begin{pmatrix}
\fe & \epsilon \lambda_1 & 1 & \fla_-\\
U & w_1 + 1 & q - 1 & W_-
\end{pmatrix}\\ &+\Delta^{q-1}_{q,w_1}S_d\begin{pmatrix}
\fe & \epsilon \lambda_1 & \lambda_2 &  \lambda_3 & \dots & \lambda_r\\
U & w_1 + 1 & q + w_2 - 1 & w_3 & \dots & w_r
\end{pmatrix}\\
& +\Delta^{q-1}_{q,w_1} \sum_{i=2}^{r-1}\prod_{j=2}^i(\Delta_{q-1,w_j}^{q-1}+1)S_d\begin{pmatrix}
\fe & \epsilon \lambda_1 & \lambda_2 \dots \lambda_i & \lambda_{i+1} &  \lambda_{i+2}  \dots  \lambda_r\\
U & w_1 + 1 & w_2  \dots  w_i& q + w_{i+1} - 1 & w_{i+2}  \dots  w_r
\end{pmatrix}\\
& +\Delta^{q-1}_{q,w_1} \sum_{i=2}^{r}\prod_{j=2}^i(\Delta^{q-1}_{q-1,w_j}+1)S_d \begin{pmatrix}
\fe & \epsilon \lambda_1 & \lambda_2 & \dots & \lambda_i & 1 &  \lambda_{i+1} & \dots & \lambda_r\\
U & w_1 + 1 & w_2 & \dots & w_i& q - 1 & w_{i+1} & \dots & w_r
\end{pmatrix} \\
&+\epsilon^{-1}D_1S_d\begin{pmatrix}
\fe & \epsilon & 1 & \fla\\ U & 1 & q - 1 & W
\end{pmatrix}+\epsilon^{-1}D_1S_d\begin{pmatrix}
\fe & \epsilon & \lambda_1 & \fla_-\\
U & 1 & q + w_1 - 1 & W_-
\end{pmatrix}\\
&+\epsilon^{-1}D_1S_d\begin{pmatrix}
\epsilon_1 & \dots & \epsilon_{n-1} & \epsilon_n\epsilon& 1 & \fla\\
u_1 & \dots & u_{n-1} & u_n + 1 & q - 1 & W
\end{pmatrix}\\
&+\epsilon^{-1}D_1S_d\begin{pmatrix}
\epsilon_1 & \dots & \epsilon_{n-1} & \epsilon_n\epsilon& \lambda_1 & \fla_-\\
u_1 & \dots & u_{n-1} & u_n + 1 & q + w_1 - 1 & W_-
\end{pmatrix}\\
&+\epsilon^{-1}D_1\sum_{i=1}^{r-1}\prod_{j=1}^i(\Delta_{q-1,w_j}^{q-1}+1)S_d\begin{pmatrix}
\fe & \epsilon & \lambda_1 & \dots & \lambda_i & \lambda_{i+1} & \lambda_{i+2} & \dots & \lambda_r\\
U & 1 & w_1 &\dots & w_i & q + w_{i+1} - 1 & w_{i+2} &\dots & w_r
\end{pmatrix}\\
&+\epsilon^{-1}D_1\sum_{i=1}^{r}\prod_{j=1}^i(\Delta^{q-1}_{q-1,w_j}+1)S_d
\begin{pmatrix}
\fe&\epsilon&\lambda_1\dots\lambda_i&1&\lambda_{i+1}\dots\lambda_r\\
U&1&w_1\dots w_i&q-1&w_{i+1}\dots w_r
\end{pmatrix}\\
&+\epsilon^{-1}D_1\sum_{i=1}^{r}\prod_{j=1}^i(\Delta^{q-1}_{q-1,w_j}+1)S_d\begin{pmatrix}
\epsilon_1 \dots \epsilon_{n-1}&\epsilon_n\epsilon&\lambda_1 \dots \lambda_i&1&\lambda_{i+1} \dots \lambda_r\\
u_1 \dots u_{n-1}&u_n+1&w_1 \dots w_i&q-1&w_{i+1} \dots w_r
\end{pmatrix}\\
+&\epsilon^{-1}D_1\sum_{i=1}^{r-1}\prod_{j=1}^i(\Delta^{q-1}_{q-1,w_j}+1)S_d
\begin{pmatrix}
\epsilon_1  \dots  \epsilon_{n-1}&\epsilon_n\epsilon&\lambda_1 \dots \lambda_i& \lambda_{i+1}&\lambda_{i+2} \dots \lambda_r\\
u_1  \dots  u_{n-1} &u_n+1&w_1 \dots w_i& q+w_{i+1}-1&w_{i+2} \dots w_r
\end{pmatrix} \\
&=0.
\end{align*}
We denote by $(*)$ this formula.

We denote by $H_r$ the following claim: for any tuples
of positive integers $U$ and $W=(w_1,\dots,w_r)$ of depth $r$, $\fe \in(\Fq^\times)^{\depth(U)}$ of any depth,
$\fla=(\lambda_1,\dots,\lambda_r)\in(\Fq^\times)^{r}$,  and $\epsilon \in \Fq^\times$ such that
$w(U)+w(W)+q=w$, the AMZV's $\zeta_A\begin{pmatrix}
\fe & \epsilon & \fla\\
U & q & W
\end{pmatrix}$ and $\zeta_A\begin{pmatrix}
\fe & \epsilon \lambda_1 & \fla_-\\
U & q+w_1 & W_-
\end{pmatrix}$  belong to $\langle\mathcal{AT}_w^1\rangle$.

We claim that $H_r$ holds for any $r \geq 0$. The proof is by induction on $r$. For $r=0$, we know that $W=\emptyset$. Letting $U=(u_1,\dots,u_n)$ and $\fe=(\epsilon_1,\dots,\epsilon_n)$ we apply the formula $(*)$  to get an explicit expression for $\mathcal
B_{\fe,U}(R_{\epsilon})$ given by
\begin{align*}
    S_d\begin{pmatrix}
    \fe & \epsilon\\
    U & q
    \end{pmatrix}&+\epsilon^{-1}D_1S_d\begin{pmatrix}
    \fe & \epsilon& 1\\
    U & 1 & q - 1
    \end{pmatrix} +\epsilon^{-1}D_1S_d\begin{pmatrix}
   \epsilon_1 & \dots & \epsilon_{n-1} & \epsilon_n \epsilon & 1 \\
    u_1 & \dots & u_{n-1} & u_n + 1 & q - 1
    \end{pmatrix}=0.
\end{align*}
Since $u_i \leq w(U)=w-q \leq q-2$, we deduce that $\zeta_A\begin{pmatrix}
\fe & \epsilon\\
U & q
\end{pmatrix}\in\langle \mathcal{AT}_w^1 \rangle$ as required.

Suppose that $H_{r'}$ holds for any $r'<r$. We now show that $H_r$ holds. We proceed again by induction on $w_1$. For $w_1=1$, letting $U=(u_1,\dots,u_n)$ and $\fe=(\epsilon_1,\dots,\epsilon_n)$ we apply the formula $(*)$ to get an explicit expression for $\mathcal B_{\fe,U} \mathcal C_{\fla,W} (R_{\epsilon})$. As $w(U)+w(W)=w-q \leq q-2$, by induction we deduce that all the terms except the first two ones in this expression belong to $\langle \mathcal{AT}_w^1 \rangle$. Thus for any $\epsilon \in \Fq^\times$,
\begin{equation} \label{eq:alg_2q-2}
\zeta_A\begin{pmatrix}
\fe & \epsilon & \fla\\
U & q & W
\end{pmatrix}+\zeta_A\begin{pmatrix}
\fe & \epsilon \lambda_1 & \fla_-\\
U & q + 1 & W_-
\end{pmatrix} \in \langle \mathcal{AT}_w^1 \rangle.
\end{equation}
We take $\epsilon=1$. As the first term lies in $\mathcal{AT}_w^1$ by definition, we deduce that
\begin{align*}
\zeta_A\begin{pmatrix}
\fe & \lambda_1 & \fla_-\\
U & q + 1 & W_-
\end{pmatrix} \in \langle \mathcal{AT}_w^1 \rangle.
\end{align*}
Thus in \eqref{eq:alg_2q-2} we now know that the second term lies in $\langle \mathcal{AT}_w^1 \rangle$, which implies that
	\[ \zeta_A\begin{pmatrix}
\fe & \epsilon & \fla\\
U & q & W
\end{pmatrix} \in \langle \mathcal{AT}_w^1 \rangle. \]

We suppose that $H_r$ holds for all $W'=(w_1',\dots,w_r')$ such that $w_1'<w_1$. We have to show that $H_r$ holds for all $W=(w_1,\dots,w_r)$. The proof follows the same line as before. Letting $U=(u_1,\dots,u_n)$ and $\fe=(\epsilon_1,\dots,\epsilon_n)$ we apply the formula $(*)$ to get an explicit expression for $\mathcal B_{\fe,U} \mathcal C_{\fla,W} (R_{\epsilon})$. As $w(U)+w(W)=w-q \leq q-2$, by induction we deduce that all the terms except the first two ones in this expression belong to $\langle \mathcal{AT}_w^1 \rangle$. Thus for any $\epsilon \in \Fq^\times$,
\begin{equation} \label{eq:alg_2q-2 2}
\zeta_A\begin{pmatrix}
\fe & \epsilon & \fla\\
U & q & W
\end{pmatrix}+\zeta_A\begin{pmatrix}
\fe & \epsilon \lambda_1 & \fla_-\\
U & q + w_1 & W_-
\end{pmatrix} \in \langle \mathcal{AT}_w^1 \rangle.
\end{equation}
We take $\epsilon=1$ and deduce
\begin{align*}
\zeta_A\begin{pmatrix}
\fe & \lambda_1 & \fla_-\\
U & q + w_1 & W_-
\end{pmatrix} \in \langle \mathcal{AT}_w^1 \rangle.
\end{align*}
Thus in \eqref{eq:alg_2q-2 2} the second term lies in $\langle \mathcal{AT}_w^1 \rangle$, which implies
	\[ \zeta_A\begin{pmatrix}
\fe & \epsilon & \fla\\
U & q & W
\end{pmatrix} \in \langle \mathcal{AT}_w^1 \rangle. \]
The proof is complete.
\end{proof}

\begin{remark} \label{rem:2q-1}
The condition $w \leq 2q-2$ is essential in the previous proof as it allows us to significantly simplify the expression of $\mathcal B_{\fe,U} \mathcal C_{\fla,W} (R_{\epsilon})$ (see Eq. \eqref{eq:alg_2q-2 2}). For $w=2q-1$ the situation is already complicated but we can manage to prove Proposition \ref{prop:Brown adhoc}. Unfortunately, we are not able to extend it to $w=2q$.
\end{remark}

\begin{proposition}
Let either $w \leq 3q-3$, or $w=3q-2,q=2$. Then $\dim_K \mathcal{AZ}_w \geq s(w)$.
\end{proposition}

\begin{proof}
We outline a proof of this theorem and refer the reader to \cite{IKLNDP23} for more details. For $1 \leq w \leq 3q-2$, we denote by $\mathcal I_w'$ the set of tuples $\fs=(s_1,\ldots,s_r) \in \mathbb N^r$ of weight $w$ as follows:
\begin{itemize}
\item For $1 \leq w \leq 2q-2$, $\mathcal I_w'$ consists of tuples $\fs=(s_1,\ldots,s_r) \in \mathbb N^r$ of weight $w$ where $s_i \neq q$ for all $i$.

\item For $2q-1 \leq w \leq 3q-3$, $\mathcal I_w'$ consists of tuples $\fs=(s_1,\ldots,s_r) \in \mathbb N^r$ of weight $w$ of the form
\begin{itemize}
\item either $s_i \neq q, 2q-1, 2q$ for all $i$,

\item or there exists a unique integer $1 \leq i <r$ such that $(s_i,s_{i+1})=(q-1,q)$.
\end{itemize}

\item For $w = 3q-2$ and $q>2$, $\mathcal I_w'$ consists of tuples $\fs=(s_1,\ldots,s_r) \in \mathbb N^r$ of weight $w$ of the form
\begin{itemize}
\item either $s_i \neq q, 2q-1, 2q, 3q-2$ for all $i$,

\item or there exists a unique integer $1 \leq i <r$ such that $(s_i,s_{i+1}) \in \{(q-1,q),(2q-2,q)\}$, but $\fs \neq (q-1,q-1,q)$,

\item or $\fs=(q-1,2q-1)$.
\end{itemize}

\item For $q=2$ and $w = 3q-2=4$, $\mathcal I_w'$ consists of the following tuples: $(2,1,1)$, $(1,2,1)$ and $(1,3)$.
\end{itemize}
We denote by $\mathcal{AT}_w'$ the subset of AMZV's given by
	\[ \mathcal{AT}_w':=\left\{\zeta_A\begin{pmatrix}
	\fe \\ \fs
	\end{pmatrix} : \fs \in \mathcal I_w', \text{ and } \epsilon_i=1 \text{ whenever } s_i \in \{q,2q-1\} \right\}.\]
Thus, if either $w \leq 3q-3$, or $w=3q-2, q=2$, then one shows that
	\[ |\mathcal{AT}_w'|=s(w). \]

Further, for $w \leq 3q-3$ and any $(\fs;\fe)=(s_1,\dots,s_r;\epsilon_1,\dots,\epsilon_r) \in \N^r \times (\Fq^\times)^r$, if $\zeta_A \begin{pmatrix}
 \fve  \\
\fs  \end{pmatrix} \in \mathcal{AT}_w'$, then $\zeta_A \begin{pmatrix}
s_1 & \dots & s_{r-1} \\
\epsilon_1 & \dots & \epsilon_{r-1}
\end{pmatrix}$ belongs to $\mathcal{AT}_{w-s_r}'$. This property allows us to apply Theorem~\ref{theorem: linear independence} and show by induction on $w \leq 3q-3$ that the AMZV's in $\mathcal{AT}_w'$ are all linearly independent over $K$. The proof is similar to that of Theorem~\ref{thm: trans ACMPL}. We apply Theorem \ref{theorem: linear independence} and reduce to solve a system of $\sigma$-linear equations. By direct but complicated calculations, we show that there does not exist any non-trivial solutions and we are done. For $w=3q-2$ and $q=2$, it can be treated separately by the same method.
\end{proof}

\begin{remark}
1) We note that the MZV's $\zeta_A(1,2q-2)$ and $\zeta_A(2q-1)$ (resp. $\zeta_A(1,3q-3)$ and $\zeta_A(3q-2)$) are linearly dependent over $K$ by \cite[Theorem 3.1]{LRT14}. This explains the above ad hoc construction of $\mathcal{AT}_w'$.

\noindent 2) Despite extensive numerical experiments, we cannot find a suitable basis $\mathcal{AT}_w'$ for the case $w=3q-1$.
\end{remark}



\end{document}